\documentclass{amsart}
\usepackage{a4wide}

\usepackage{tabularx,ragged2e,booktabs,caption}
\usepackage{amsthm}                   	
\usepackage{amsmath}                  	
\usepackage{amsfonts}                 	
\usepackage{amssymb}                  	
\usepackage{mathrsfs}                 	
\usepackage{mathtools}			
\usepackage{pstricks-add}
\usepackage{xy}
\usepackage{graphicx,subfigure}
\usepackage{hyperref}
\xyoption{all}
\usepackage{multirow}

\newsavebox\TBox
\def\textoverline#1{\savebox\TBox{#1}
  \makebox[0pt][l]{#1}\rule[1.3\ht\TBox]{\wd\TBox}{0.4pt}}

\font\myfont=cmr6
\def\st#1 {  {\hbox{\myfont \begin{tabular}{c} {\, $\mbox{\myfont #1}$}  \end{tabular}  }} }
\def\state#1#2 {  {\hbox{\myfont \begin{tabular}{c}\noalign{\vskip -0.1cm} {\; #1}
\\[-0.5mm] \noalign{\vskip -0.1cm} {\; #2} \end{tabular}  }} }
\def\statev#1#2#3 {  {\hbox{\myfont \begin{tabular}{c}\noalign{\vskip 0.1cm} {\; #1}
\\[-1.0mm] \noalign{\vskip -0.0cm} {\; #2} \\[-1.0mm] \noalign{\vskip -0.0cm} {\; #3} \end{tabular}  }} }

\font\myfonta=cmr10
\def\vect#1 {  {\hbox{\myfonta \begin{tabular}{c} {\, #1}  \end{tabular}  }} }

\font\myfontb=cmr10
\def\vectv#1#2#3 {  {\hbox{\myfontb \begin{tabular}{c}\noalign{\vskip 0.1cm} {\; #1}
\\[-1.1mm] \noalign{\vskip -0.0cm} {\; #2} \\[-1.1mm] \noalign{\vskip -0.0cm} {\; #3} \end{tabular}  }} }

\usepackage{epstopdf}

\theoremstyle{definition}
\newtheorem{definition}{Definition}[section]
\newtheorem{remark}[definition]{Remark}

\theoremstyle{plain}
\newtheorem{theorem}[definition]{Theorem}
\newtheorem{lemma}[definition]{Lemma}
\newtheorem{conjecture}[definition]{Conjecture}
\newtheorem{problem}[definition]{Problem}
\newtheorem{corollary}[definition]{Corollary}
\newtheorem{proposition}[definition]{Proposition}
\newtheorem*{theorem*}{Theorem}
\newtheorem{algo}[definition]{Algorithm}
\numberwithin{equation}{section}

\newcommand{\mybinom}[2]{\tiny\left\{
\substack{#1\\#2}\right\}}

\newcommand{\R}{\ensuremath{\mathbb{R}}}     	
\newcommand{\N}{\ensuremath{\mathbb{N}}}     	
\newcommand{\Z}{\ensuremath{\mathbb{Z}}}     	
\newcommand{\D}{\ensuremath{\mathcal{D}}}     	
\newcommand{\nP}{\ensuremath{\mathcal{P}}}
\newcommand{\nQ}{\ensuremath{\mathcal{Q}}}
\newcommand{\nS}{\ensuremath{\mathcal{S}}}   	
\newcommand{\nR}{\ensuremath{\mathcal{R}}}   	
\newcommand{\eps}{\varepsilon}
\newcommand{\scr}{\scriptsize}
\newcommand{\Bold}{\boldsymbol}
\newcommand{\B}{\ensuremath{\boldsymbol{B}}}
\newcommand{\BO}{\ensuremath{\boldsymbol{O}}}
\newcommand{\BA}{\ensuremath{\boldsymbol{A}}}

\begin{document}
\author{J\"org Thuswaldner }\author{Shu-qin Zhang}
\address{Chair of Mathematics and Statistics, University of Leoben, Franz-Josef-Strasse 18, A-8700 Leoben, Austria}
\email{joerg.thuswaldner@unileoben.ac.at}
\email{shuqin.zhang@unileoben.ac.at}
\title[Self-affine tiles]{On self-affine tiles whose boundary is a sphere}

\dedicatory{Dedicated to Val\'erie Berth\'e on the occasion of her 50$^{\mathit{th}}$ birthday}

\date{\today}
\thanks{Supported by FWF project P29910, by FWF-RSF project I3466, and by the FWF doctoral program W1230}

\maketitle
	
\allowdisplaybreaks

\begin{abstract}
Let $M$ be a $3\times 3$ integer matrix each of whose eigenvalues is greater than $1$ in modulus and let $\mathcal{D}\subset\Z^3$ be a set with $|\D|=|\det M|$, called {\em digit set}.
The set equation $MT = T+\D$ uniquely defines a nonempty compact set $T\subset \R^3$. If $T$ has positive Lebesgue measure it is called a $3$-dimensional {\em self-affine tile}. In the present paper we study topological properties of $3$-dimensional self-affine tiles with {\em collinear digit set}, {\em i.e.}, with a digit set of the form $\D=\{0,v,2v,\ldots, (|\det M|-1)v\}$ for some $v\in\Z^3\setminus\{0\}$. 
We prove that the boundary of such a tile $T$ is homeomorphic to a $2$-sphere whenever its set of \emph{neighbors} 
in a lattice tiling which is induced by $T$ in a natural way contains $14$ elements. The combinatorics of this lattice tiling is then the same as the one of the {\em bitruncated cubic honeycomb}, a body-centered cubic lattice tiling by {\it truncated octahedra}. We give a characterization of $3$-dimensional self-affine tiles with collinear digit set having $14$ neighbors in terms of the coefficients of the characteristic polynomial of $M$. In our proofs we use results of R.~H.~Bing on the topological characterization of spheres. 
\end{abstract}

\section{Introduction}
Let $m\in \N$ and suppose that $M$ is an $m\times m$ integer matrix which is {\em expanding}, {\em i.e.}, each of its eigenvalues is greater than $1$ in modulus. Let $\D\subset\Z^m$ be a set of cardinality $|\det M|$ which is called {\em digit set}.  By a result of Hutchinson \cite{Hutchinson81}, there exists a unique nonempty compact subset $T=T(M,\D)$ of $\R^m$ such that 
\begin{equation}\label{eq:SetEquationTile}
MT=T+\D.
\end{equation}
If $T$ has positive Lebesgue measure (which by  Bandt~\cite{Bandt91} is always the case if $\D$ is a complete set of \begin{figure}[h]
\vskip -0.4cm
\includegraphics[trim={0 0 0 -0.5cm},clip=true,width=.45\textwidth]{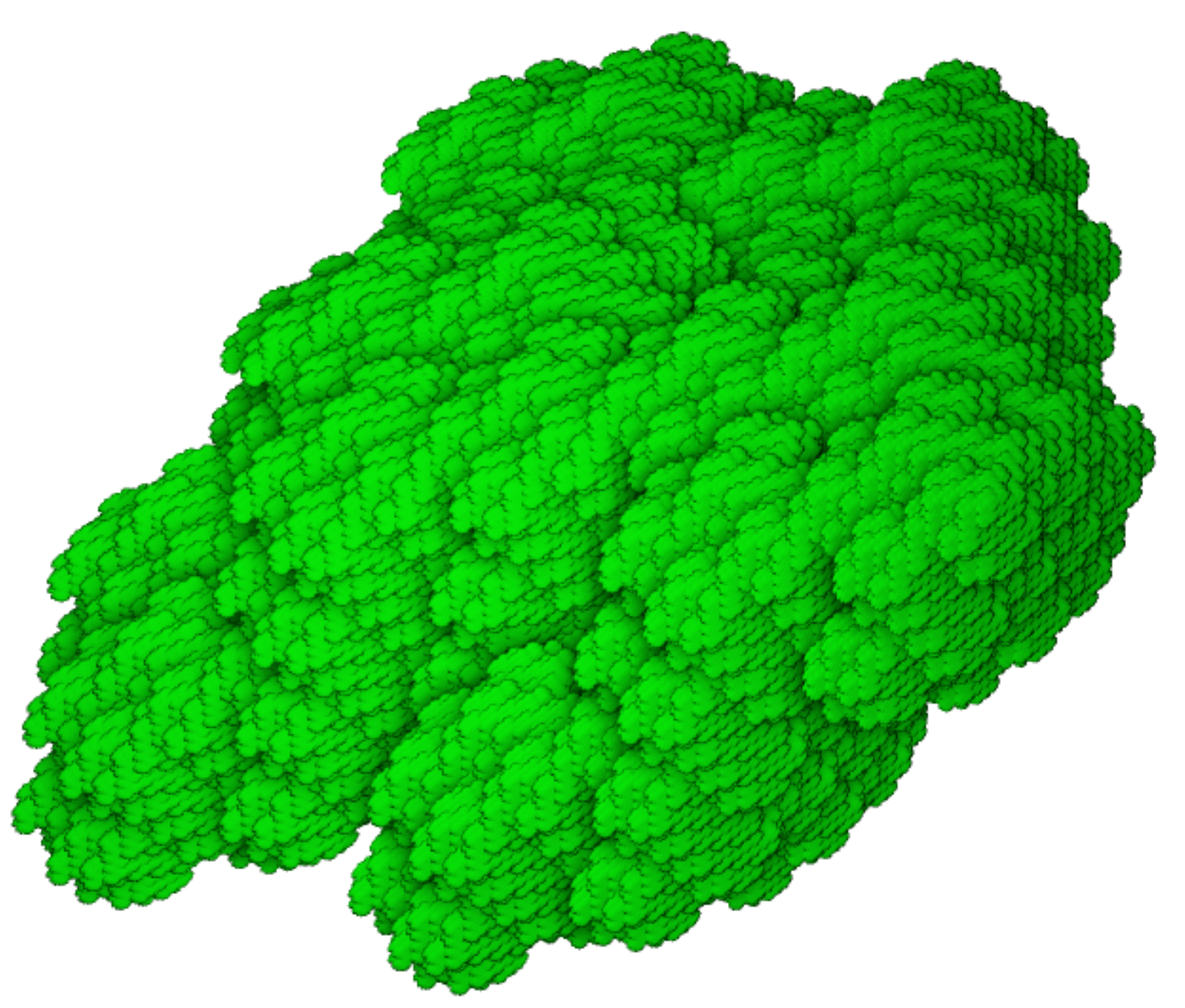} 
\hskip 0.5cm
\includegraphics[trim={-8 00 00 0},clip=true,origin=c,width=.45\textwidth]{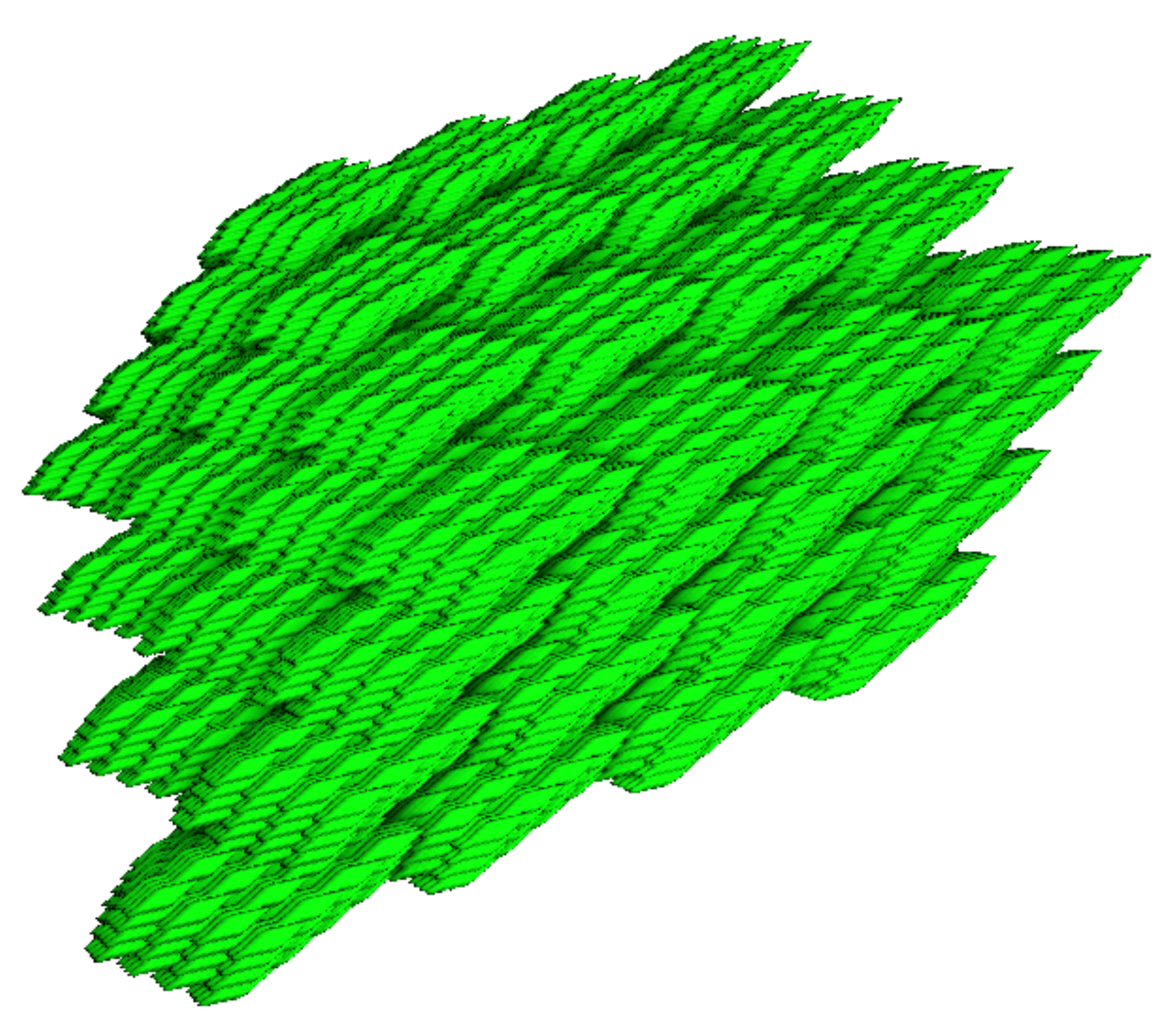}  
\caption{Two examples of $3$-dimensional self-affine tiles.\label{fig:3d}}
\end{figure}
coset representatives of $\Z^m/M\Z^m$) we call it a {\em self-affine tile}. Images of two $3$-dimensional self-affine tiles with typical ``fractal'' boundary are shown in Figure~\ref{fig:3d} (they were created using IFStile~\cite{wiki:xxx}).
Initiated by the work of
Kenyon~\cite{Kenyon92} self-affine tiles have been studied extensively in the literature. A systematic theory of self-affine tiles including the lattice tilings they often induce has been established in the 1990ies by Gr\"ochenig and Haas~\cite{GroechenigHaas94} as well as Lagarias and Wang~\cite{LagariasWang96b,LagariasWang96a,LagariasWang97}.  Since then, self-affine tiles have been investigated in many contexts. One field of interest, the one to which the present paper is devoted, is the topology of self-affine tiles. Based on the pioneering work of Hata~\cite{Hata85} on topological properties of attractors of iterated function systems many authors explored the topology of self-affine tiles.  For instance, Kirat and Lau~\cite{KiratLau00} and Akiyama and Gjini~\cite{AkiyamaGjini04,AkiyamaGjini05} dealt with connectivity of tiles. Later, finer topological properties of $2$-dimensional self-affine tiles came into the focus of research. Bandt and Wang~\cite{BandtWang01} gave criteria for a self-affine tile to be homeomorphic to a disk (see also Lau and Leung~\cite{LauLeung07}), Ngai and Tang~\cite{NgaiTang05} dealt with planar connected self-affine tiles with disconnected interior, and Akiyama and Loridant~\cite{AkiyamaLoridant11} provided parametrizations of the boundary of planar tiles.

Only a few years ago first results on topological properties of $3$-dimensional self-affine tiles came to the fore. Bandt~\cite{Bandt10} studied the combinatorial topology of $3$-dimensional twin dragons. Very recently, Conner and Thuswaldner~\cite{ConnerThuswaldner0000} gave criteria for a $3$-dimensional self-affine tile to be homeomorphic to a $3$-ball by using upper semi-continuous decompositions and a criterion of Cannon~\cite{Cannon73} on tame embeddings of $2$-spheres. Deng {\it et al.}~\cite{DengLiuNgai18} showed that a certain class of $3$-dimensional self-affine tiles is homeomorphic to a $3$-ball.

Let $M$ be an expanding $m\times m$ integer matrix. We say that $\D$ is a {\em collinear digit set} for $M$ if there is a vector $v\in\Z^m\setminus\{0\}$ such that
\begin{equation}\label{eq:collinearD}
\D=\{0,v,2v,\ldots, (|\det M|-1)v\}.
\end{equation}
If $\D$ has this form we call a self-affine tile $T=T(M,\D)$ a {\em self-affine tile with collinear digit set} (such tiles have been studied by many authors in recent years, see {\it e.g.}~Lau and Leung~\cite{LauLeung07}).
In the present paper we establish a general characterization of $3$-dimensional self-affine tiles with collinear digit set whose boundary is homeomorphic to a $2$-sphere. In its proof we use a result of Bing~\cite{Bing51} that provides a topological characterization of $m$-spheres for $m\le 3$ (although Bing does not mention self-affine sets, his characterization is very well suited for self-affine structures). 

Before we state our main results we introduce some notation. Let $T=T(M,\D)$ be a self-affine tile in $\R^m$ with collinear digit set and define the set of \emph{neighbors} of $T$ by 
\begin{equation}\label{eq:neigh}
\nS=\{\alpha\in \Z[M,\D]\setminus\{0 \}; \; T\cap (T+\alpha)\neq \emptyset\},
\end{equation}
where 
\[
\Z[M,\D] = \Z[\D,M\D, \dots, M^{m-1}\D] \subset \Z^m
\]
is the smallest $M$-invariant lattice containing $\D$. This definition is motivated by the fact that the collection $\{T+ \alpha;\, \alpha \in \Z[M,\D]\}$ often tiles the space $\R^m$ with overlaps of Lebesgue measure $0$ ({\it cf.\ e.g.}~Lagarias and Wang~\cite{LagariasWang97} and note that one can always achieve that $\Z[M,\D]=\Z^m$ by applying an affine transformation; see \cite[Lemma~2.1]{LagariasWang96b}). The translated tiles $T+\alpha$ with $\alpha\in\nS$ are then those tiles of this tiling which touch the ``central tile'' $T$. It is clear that $\nS$ is a finite set since $T$ is compact by definition. Set 
\begin{equation}\label{eq:bgamma}
\B_{\alpha}=T\cap(T+\alpha) \qquad(\alpha \in \Z[M,\D]\setminus\{0\}). 
\end{equation}
More generally, for $\ell \ge 1$ and a subset $\Bold{\alpha} = \{\alpha_1,\ldots, \alpha_{\ell}\} \subset \Z[M,\D]\setminus\{0\}$ we define the {\em $(\ell+1)$-fold intersections} by
\[
\B_{\Bold{\alpha}} = \B_{\alpha_1,\ldots, \alpha_{\ell}} = T \cap (T+\alpha_1) \cap \cdots \cap (T+\alpha_{\ell}) \qquad (\Bold{\alpha}\subset \Z[M,\D]\setminus\{0\}).
\]
Compactness of $T$ again yields that there exist only finitely many sets $\Bold{\alpha}\subset \Z[M,\D]$ with $\B_{\Bold{\alpha}}\not=\emptyset$.

\begin{theorem}\label{Main-1}
Let $T=T(M,\D)$ be a $3$-dimensional self-affine tile with collinear digit set and assume that the characteristic polynomial $\chi(x)=x^3+Ax^2+Bx+C$ of $M$ satisfies $1\le A\le B < C$. Then $\{T+ \alpha;\, \alpha \in \Z[M,\D]\}$ tiles the space $\R^3$ with overlaps of Lebesgue measure $0$.
If $T$ has $14$ neighbors then the following assertions hold.
 \begin{enumerate}
 \item \label{Main1.sphere} The boundary $\partial T$ is homeomorphic to a $2$-sphere. 
 \item \label{Main1.disk}  If $\alpha\in \Z[M,\D]\setminus\{0\}$, the $2$-fold intersection $\B_{\alpha}$ is homeomorphic to a closed disk for each $\alpha\in \nS$ and empty otherwise.
\item \label{Main1.loop} If $\Bold{\alpha}\subset\Z[M,\D]\setminus\{0\}$ contains two elements, the $3$-fold intersection $\B_{\Bold{\alpha}}$ is either homeomorphic to an arc or empty. The $36$ sets $\Bold{\alpha}$ with $\B_{\Bold{\alpha}}\not=\emptyset$ can be given explicitly.
 \item \label{Main1.point} If $\Bold{\alpha}\subset\Z[M,\D]\setminus\{0\}$ contains three elements, the $4$-fold intersection $\B_{\Bold{\alpha}}$ is either a single point or empty. The $24$ sets $\Bold{\alpha}$ with $\B_{\Bold{\alpha}}\not=\emptyset$ can be given explicitly.
 \item \label{Main1.empty} If $\Bold{\alpha}\subset\Z[M,\D]\setminus\{0\}$ contains $\ell \ge 4$ elements, the $(\ell+1)$-fold intersection $\B_{\Bold{\alpha}}$ is always empty.
 \end{enumerate}  
\end{theorem}

\begin{remark}
Theorem~\ref{Main-1}~(\ref{Main1.sphere}) and (\ref{Main1.disk}) imply that for $\alpha\in \nS$ the boundary $\partial_{\partial T} \B_{\alpha}$ is a simple closed curve. We denote by $\partial_X$  the boundary taken w.r.t.\ the subspace topology on $X \subset \R^3$.
\end{remark}

\begin{remark}
Since the characteristic polynomial $\chi$ of $M$ in Theorem~\ref{Main-1} satisfies $1\le A\le B < C$, the matrix $M$ is expanding (see Lemma~\ref{lem:expand} below). The fact that $\chi(-1) >0$ implies that $M$ has a real eigenvalue which is less than $-1$. The remaining eigenvalues of $M$ may be real or nonreal.
\end{remark}

There is an interesting relation to polyhedral geometry and crystallography (see also Remark~\ref{rem:crystal}). The arrangement of the tiles in the tiling $T+\mathbb{Z}^3$ in Theorem~\ref{Main-1} with 14 neighbors, 36 $3$-fold intersections, and 24 $4$-fold intersections coincides with the arrangement of the \emph{bitruncated cubic honeycomb} (also known as {\it truncoctahedrille}), a body-centered cubic lattice tiling induced by a  {\it truncated octahedron} $O$ (see {\it e.g.}~Conway {\it et al.}~\cite[p.~295 and p.~329f.]{CBG:08} or Barnes~\cite[p.~68]{Barnes:12}; Cromwell~\cite[Figure~2.19 (b)]{Cromwell:97} and \cite[p.~333ff]{CBG:08} mention the related {\it muoctahedron} discovered by Petrie and Coxeter). Denote this body-centered lattice by $\Lambda$. Then by Theorem~\ref{Main-1} there exists a homeomorphism $h:O\to T$ that maps $\ell$-fold intersections to $\ell$-fold intersections in the respective tiling homeomorphically. Extending $h$ equivariantly we obtain the following result.

\begin{corollary}\label{cor:O}
If the self-affine tile $T$ satisfies the conditions of Theorem~\ref{Main-1} then
$\partial T + \mathbb{Z}^3$ is homeomorphic to $\partial O + \Lambda$.
\end{corollary}

\begin{figure}[h]
\includegraphics[trim={0 70 0 50},clip=true,origin=c,width=.45\textwidth]{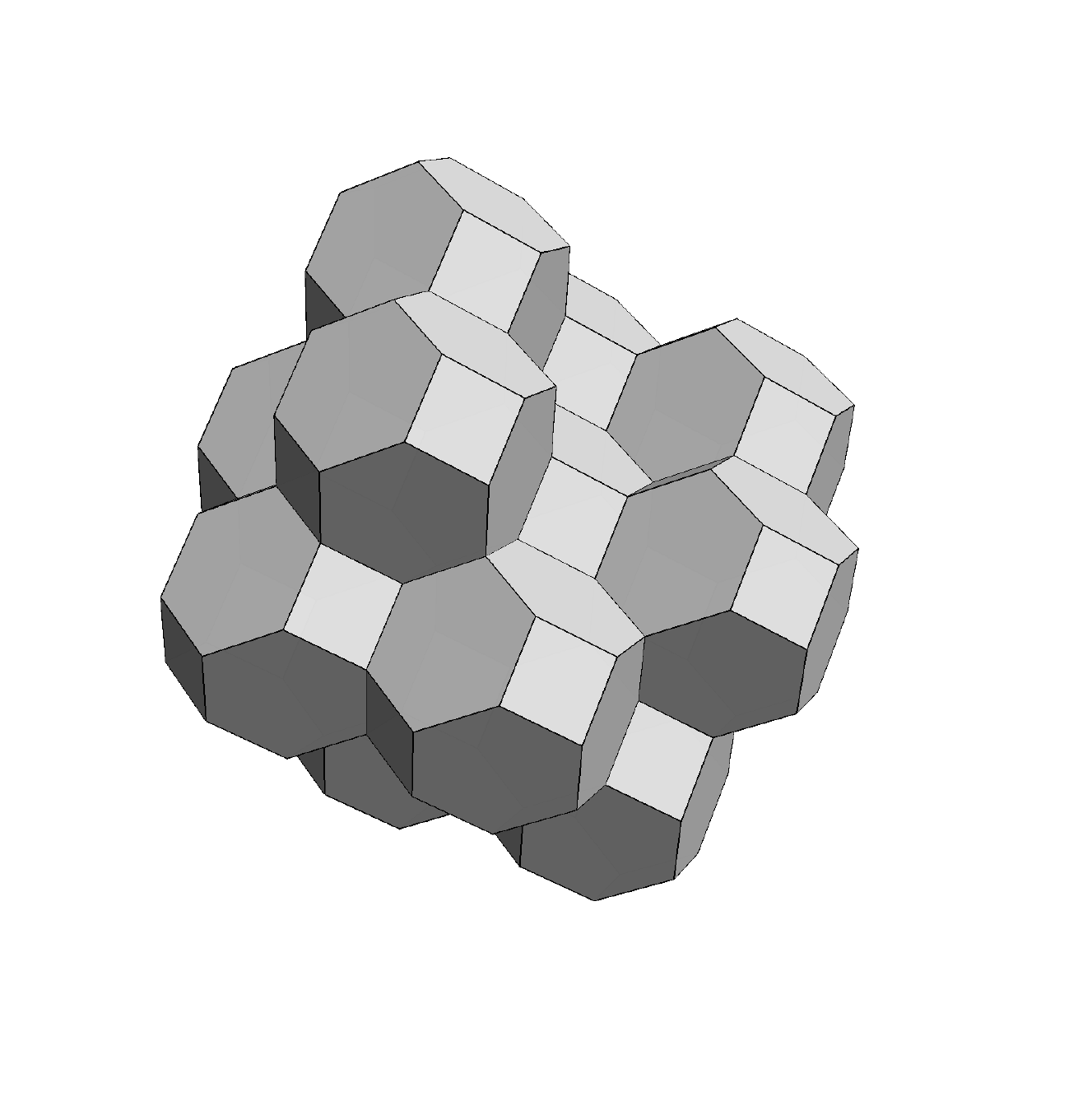}  
\caption{A patch of the bitruncated cubic honeycomb.\label{fig:honey}}
\end{figure}
A patch from $O + \Lambda$ is depicted in Figure~\ref{fig:honey}.  Corollary~\ref{cor:O} says that the polyhedral structure of all tiles satisfying the conditions of Theorem~\ref{Main-1} is the same. However, the way how the intersections $B_\alpha$, $\alpha\in \nS$, are subdivided when $T$ is subdivided according to the set equation \eqref{eq:SetEquationTile} heavily depends on the coefficients of the characteristic polynomial $\chi$. This is why it is difficult to treat the whole class of self-affine tiles at once in Theorem~\ref{Main-1}.

Theorem~\ref{Main-1} raises the question when $3$-dimensional self-affine tiles with collinear digit set have $14$ neighbors. This question is answered as follows.

\begin{theorem}\label{count-neigh-General}
Let $T=T(M,\D)$ be a $3$-dimensional self-affine tile with collinear digit set and assume that the characteristic polynomial $\chi(x)=x^3+Ax^2+Bx+C$ of $M$ satisfies $1\le A\le B < C$.  

Then $T$ has $14$ neighbors if and only if $A, B, C$ satisfy one of the following conditions.
\begin{itemize}
\item[(1)] $1\leq A<B<C$, $B\geq 2A-1$, and $C\geq 2(B-A)+2$;
\item[(2)] $1\leq A<B<C$, $B<2A-1$, and $C\geq A+B-2$.
\end{itemize}
\end{theorem}

The paper is organized as follows. In Section~\ref{sec_graphandneighbor} we prove Theorem~\ref{count-neigh-General}. The main ingredient of this proof are certain graphs that contain information on the neighbors of $T$. These graphs also can be used to define so-called {\em graph-directed iterated function systems} in the sense of Mauldin and Williams~\cite{MauldinWilliams88} whose attractor is the collection  $\{\B_\alpha;\, \alpha \in \nS\}$. We will also establish graphs that describe the nonempty $\ell$-fold intersections $\B_{\Bold{\alpha}}$. All these results will be needed in Section~\ref{sec:topo}, the core part of the present paper, where we will combine them with Bing's results from~\cite{Bing51} and other topological results including dimension theory to establish Theorem~\ref{Main-1}. 
In Section~\ref{sec:perspective} we discuss perspectives for further research.

\section{Intersections and neighbors}\label{sec_graphandneighbor}

After providing basic definitions in Section~\ref{sec:normal}, we deal with graphs that describe the intersections of a self-affine tile with its neighbors. In Section~\ref{sec:21} we define these graphs and in Sections~\ref{sec:contact} and \ref{sec:neighbor} we compute them for the class of self-affine tiles relevant for us. This will lead to the proof of Theorem~\ref{count-neigh-General}. Finally, Section~\ref{sec:34} deals with $\ell$-fold intersections of tiles.

\subsection{Basic definitions}\label{sec:normal}
To establish our theory we need to impose some conditions on a self-affine tile. For this reason we recall the following definition that goes back to Bandt and Wang~\cite{BandtWang01}.

\begin{definition}\label{def:Zm}
Let $M$ be an expanding $m\times m$ integer matrix and assume that $\D \subset \Z^m$ is a complete set of coset representatives of $\Z^m/M\Z^m$ which satisfies $\Z[M,\D]=\Z^m$. If the self-affine tile $T=T(M,\D)$ defined by the set equation $MT=T+\D$ tiles $\R^m$ w.r.t.\ the lattice $\Z^m$ in the sense that ($\mu_m$ is the Lebesgue measure on $\R^m$)
\begin{equation}\label{eq:tilingproperty}
\mu_m((T+\alpha_1)\cap(T+\alpha_2))=0
\qquad(\alpha_1, \alpha_2\in \Z^m \text{ distinct})
\end{equation}  
we call $T$ a {\em $\Z^m$-tile}. 
\end{definition}

We will deal with general $\Z^m$-tiles only in Section~\ref{sec:21} (and in parts of Section~\ref{sec:34}). After that we will restrict ourselves to the class of self-affine tiles to which our main results are dedicated (which, as we will see, are $\Z^3$-tiles). We first show that the matrices $M$ occurring there are indeed expanding. This result is a special case of results from Kirat~{\it et al.}~\cite{KLR:04} but for the sake of completeness we include its short proof.

\begin{lemma}[{{\em cf.}~\cite[Proposition~2.6~(iii)]{KLR:04}}]\label{lem:expand}
Let $M$ be a $3\times 3$ integer matrix whose characteristic polynomial $\chi(x)=x^3+Ax^2+Bx+C$ satisfies $1\le A\le B < C$. Then $M$ is expanding.
\end{lemma}

\begin{proof}
We have to show that the roots $x_1,x_2,x_3$ of $\chi$ satisfy $|x_i|>1$ for $i\in\{1,2,3\}$. Since $\chi(-C) < 0$ and $\chi(-1) >0$ there is a root of $\chi$, say $x_1$, with $x_1\in(-C,-1)$ and \emph{a fortiori} $|x_1| > 1$.  Moreover, for $x \in [0,1]$ we obviously have $\chi(x)>0$ and for $x\in [-1,0)$ we get the same from $\chi(x) \ge x^3 +  x^2 + (C-1) x + C > 0$. Thus $\chi$ has no root in $[-1,1]$ and, hence, if $x_2,x_3\in\mathbb{R}$ then $|x_2|,|x_3|>1$ and we are done. If $x_2,x_3\not\in\mathbb{R}$ then $|x_2|=|x_3|$. Because $|x_1|<C$ and $|x_1x_2x_3|=C$ we gain $|x_2|=|x_3|>1$ and the lemma is proved.
\end{proof}

For the tiles of our main results we now define a simple normal form. Let $M'$ be a $3\times 3$ integer matrix with characteristic polynomial $x^3+Ax^2+Bx+C$ satisfying $1\le A \le B < C$. By Lemma~\ref{lem:expand} the matrix $M'$ is expanding. Moreover, let $\D'\subset \Z^3$ be a collinear digit set as in \eqref{eq:collinearD} for some $v\in \Z^3\setminus\{0\}$. Assume that $T'=T'(M',\D')$  is a self-affine tile with collinear digit set. This entails that $\{v,M'v,M'^2v\}$ is a basis of $\R^3$ because $T'$ has positive Lebesgue measure (and, hence, nonempty interior by ~\cite[Theorem~1.1]{LagariasWang96a}). Denote by $E$  the matrix of the change of basis from the standard basis of $\R^3$ to $\{v,M'v,M'^2v\}$ and set
\begin{equation}\label{digit}
M=E^{-1}M'E=\begin{pmatrix}
0&0&-C\\
1&0&-B\\
0&1&-A\\
\end{pmatrix} \quad\hbox{and}\quad \D=E^{-1}\D'=\left\{\begin{pmatrix}
0\\
0\\
0\\
\end{pmatrix}, \begin{pmatrix}
1\\
0\\
0\\
\end{pmatrix},\dots, \begin{pmatrix}
C-1\\
0\\
0\\
\end{pmatrix}\right\}.
\end{equation}
As $M'$ is expanding, the same is true for $M$. 
Define $T$ by $MT = T + \D$. Then we have $T=E^{-1} T'$ and, because $E$ is invertible, this implies that $T$ is a self-affine tile. The linear mapping induced by $E^{-1}$ maps $\Z[M',\D']$ to $\Z^3$. Moreover, $\partial T= E^{-1} \partial T'$ and for $\{\alpha_1,\ldots,\alpha_\ell\}\subset \Z[M',\D']$ we have 
\[
E^{-1} (T' \cap (T'+\alpha_1) \cap \dots \cap (T'+\alpha_\ell) ) = T \cap (T + E^{-1} \alpha_1) \cap \dots \cap (T +E^{-1} \alpha_\ell).
\] 
Thus it is sufficient to prove Theorem~\ref{Main-1} and Theorem~\ref{count-neigh-General} for self-affine tiles of the form $T=T(M,\D)$ and in all what follows we may restrict our attention to the following class of self-affine-tiles. 

\begin{definition}\label{def:ABC}
A self-affine tile $T$ given by $MT=T + \D$ with $M$ and $\D$ as in \eqref{digit}, where $A,B,C \in \Z$ satisfy $1\leq A\leq B<C$, is called {\em $ABC$-tile}.
\end{definition}

The tiles in Figure~\ref{fig:3d} are approximations of $ABC$-tiles for the choice $(A,B,C)=(1,1,2)$ (left hand side) and $(A,B,C)=(1,2,4)$ (right hand side). 

We need the following lemma which is an easy consequence of Barat {\it et al.}~\cite[Theorem~3.3]{BaratBertheLiardetThuswaldner06}.

\begin{lemma}\label{ABCZ3}
Each $ABC$-tile is a $\Z^3$-tile.\footnote{Another way to prove this would be via the general result~\cite[Theorem~6.2]{LagariasWang97}. This would also require several new notations. So we decided to do it this way.}
\end{lemma}

\begin{proof}
Each $ABC$-tile $T$ is defined as $T=T(M,\D)$ with $M$ and $\D$ as in \eqref{digit} with $1\leq A\leq B<C$. It is straightforward to check that $\D$ is a complete set of coset representatives of $\Z^3/M\Z^3$ and that $\Z[M,\D]=\Z^3$. Thus it remains to show that $\{T+\alpha;\, \alpha\in\Z^3\}$ tiles $\R^3$. Let 
$$
\Delta(M,\D) = \bigcup_{\ell \ge 0} ((\D-\D) +  M(\D-\D) + \cdots + M^\ell(\D-\D)).
$$
We claim that $\Delta(M,\D) = \Z^3$.
Obviously, $\Delta(M,\D) \subset \Z^3$. We have to prove the reverse inclusion. Since $1\leq A\leq B<C$, Barat {\em et al.}~\cite[Theorem~3.3]{BaratBertheLiardetThuswaldner06} implies that $x^3+Ax^2+Bx+C$ is the basis of a so-called {\it canonical number system}. In view of Barat {\em et al.}~\cite[Definition~3.2 and the paragraph above it]{BaratBertheLiardetThuswaldner06} this is equivalent to the fact that $(M,\D)$ is a {\em matrix numeration system}. However, by definition this means that each $z\in \Z^3$ can be represented in the form $z=d_0+M d_1+\cdots+M^\ell d_\ell$ with some $\ell \ge 0$ and $d_0,\ldots, d_\ell \in \D$. Thus $\Z^3 \subset \Delta(M,\D)$ and the claim is proved.

The result now follows from \cite[Theorem~1.2~(ii)]{LagariasWang96a}.  
\end{proof}

In view of the transformation in \eqref{digit} this lemma proves the tiling assertion in Theorem~\ref{Main-1}.

\subsection{Graphs related to the boundary of a tile}\label{sec:21}
If $M$ and $\D$ are given in a way that $T=T(M,\D)$ is a $\Z^m$-tile we obviously have
\begin{equation}\label{boun_1}
\partial T=\bigcup_{\alpha\in\mathcal{S}}\B_{\alpha}.
\end{equation} 
Here $\mathcal{S}$ and $\B_{\alpha}$, $\alpha\in\nS$, are defined as in \eqref{eq:neigh} and \eqref{eq:bgamma}, respectively; note that $\Z[M,\D]=\Z^m$ in these definitions because $T$ is a $\Z^m$-tile. By the definition of $\B_\alpha$ and the set equation \eqref{eq:SetEquationTile} we get 
\begin{equation}
\begin{split}
\label{boun_2}
\B_{\alpha} 
=T\cap (T+\alpha)
 &
=M^{-1}(T+\D)\cap M^{-1}(T+\D+M\alpha)
\\  &
=M^{-1}\bigcup_{d,d'\in \D}(\B_{M\alpha+d'-d}+d).      
\end{split}
\end{equation}
This subdivision of $\B_{\alpha}$ has been noted for instance by Strichartz and Wang~\cite{StrichartzWang99} and Wang~\cite{Wang99}.

For a directed labeled graph $G$ with set of vertices $V$, set of edges $E$, and set of edge-labels $L$ we denote an edge leading from  $v \in V$ to $v' \in V$ labeled by $\ell \in L$ by $v\xrightarrow{\ell} v'$. In this case $v$ is called a \emph{predecessor} of $v'$ and $v'$ is called a \emph{successor} of $v$. Following {\it e.g.}~Diestel~\cite[Chapter~1]{Diestel:05} we distinguish between walks and paths of $G$. A (finite or infinite) sequence $v_0\xrightarrow{\ell_1}v_1\xrightarrow{\ell_2}v_2\xrightarrow{\ell_3}\cdots$ of consecutive edges in $G$ is called a \emph{walk}. A walk whose vertices $v_0,v_1,v_2,\ldots$ are pairwise distinct is called a \emph{path}. Each path of $G$ can be regarded as a subgraph of $G$. 
If $G$ is undirected and not labeled, then an edge is denoted by $v\,\mbox{---}\,v'$. Walks and paths in $G$ are defined analogously as sequences of consecutive vertices with or without possible repetitions, respectively, as above.

\begin{definition}[{{\em cf.}~\cite[Definition~3.2]{ScheicherThuswaldner03}}]\label{Graph}
Let $M$ be an expanding integer matrix and let $\D$ be a complete set of coset representatives of $\Z^m/M\Z^m$. Define a directed labeled graph $G(\Z^m)$ as follows. The vertices of $G(\Z^m)$ are the elements of $\Z^m$, and there is a labeled edge 
\begin{align}\label{eq:arrow}
\alpha\xrightarrow{d|d'}\alpha'\quad\text{  if and only if  }  M\alpha+d'-d=\alpha'  \text{ with }\alpha,\alpha' \in \Z^m \text{ and } d, d' \in \mathcal{D}.
\end{align}
\end{definition}

For $\Gamma\subset\Z^m$ we denote by $G(\Gamma)$ the subgraph of $G(\Z^m)$ induced by $\Gamma$.
In \eqref{eq:arrow} the vector $d'$ is determined by $\alpha,\alpha',d$. Thus we sometimes just write $\alpha\xrightarrow{d}\alpha'$ instead of $\alpha\xrightarrow{d|d'}\alpha'$.  We will write $\alpha \in G(\Gamma)$ to indicate that $\alpha$ is a vertex of $G(\Gamma)$ and $\alpha\xrightarrow{d}\alpha' \in G(\Gamma)$ to indicate that 
$\alpha\xrightarrow{d}\alpha'$ is an edge of $G(\Gamma)$. For walks we will use an analogous notation.
The following symmetry property follows from Definition~\ref{Graph}. If $\alpha,\alpha',-\alpha,-\alpha'\in \Gamma$ then
\begin{equation}\label{inverse}
\alpha\xrightarrow{d|d'}\alpha'\in G(\Gamma) \quad\Longleftrightarrow\quad -\alpha\xrightarrow{d'|d}-\alpha'\in G(\Gamma).
\end{equation}
By ${\rm Red}(\Gamma)$ we denote the largest subset of $\Gamma$ for which $G(\Gamma)$ has no sink (a sink is a vertex that has no successor). Thus $G({\rm Red}(\Gamma))$ emerges from $G(\Gamma)$ by successively removing all sinks.

The graph $G(\nS)$, where $\nS$ is the set of neighbors of $T$ defined in \eqref{eq:neigh}, is called \emph{neighbor graph}. From \eqref{boun_2} we see that $\{\B_\alpha;\, \alpha\in\nS\}$ is the attractor of a graph-directed iterated function system (in the sense of Mauldin and Williams~\cite{MauldinWilliams88}) directed by the graph $G(\nS)$, that is,
the nonempty compact sets $\B_\alpha$, $\alpha\in\nS$, are uniquely determined by the set equations
\begin{equation}\label{eq:bgammaseteq}
\B_\alpha=\bigcup_{
\begin{subarray}{c}d\in\D, \alpha'\in\nS \\ \alpha\xrightarrow{d} \alpha' \in G(\nS)\end{subarray}}M^{-1}(\B_{\alpha'}+d) \qquad(\alpha\in\nS).
\end{equation}
The union in \eqref{eq:bgammaseteq} is extended over all $d,\alpha'$ with $\alpha\xrightarrow{d}\alpha'\in G(\nS)$. Thus by \eqref{boun_1} the boundary $\partial T$ is determined by the graph $G(\nS)$. This fact was used implicitly in Wang~\cite{Wang99}  in order to establish a formula for the Hausdorff dimension of the boundary of a $\Z^m$-tile~$T$.

Fix a basis $\{e_1,e_2,\dots,e_m\}$ of the lattice $\Z^m$, set $R_0=\{0, \pm e_1,\dots,\pm e_m\}$, and define the nested sequence $(R_n)_{n\ge 0}$ of subsets of $\Z^m$ inductively by 
\begin{equation}\label{eq:cnb}
R_n:=\{k\in \Z^m;\; (Mk+\D)\cap(\ell+\D)\neq\emptyset\text{ for } \ell \in R_{n-1}\}\cup R_{n-1}.
\end{equation}
We know from Gr\"ochenig and Haas~\cite[Section~4]{GroechenigHaas94} (see also Duvall~{\em et al.}~\cite{DuvallKeeslingVince00}) that $R_n$ stabilizes after finitely many steps, that is, $R_{n-1}=R_n$ holds for $n$ large enough. Therefore, $\nR={\rm Red}\left(\bigcup_{n\geq 0}R_n\right)$ is a finite set. We call $\nR$ the {\it contact set} of the $\Z^m$-tile $T$ and $G(\nR)$ its \emph{contact graph}.
Also the set $\nR$ can be used to define $\partial T$. Indeed, we have
\begin{equation*}
\partial T=\bigcup_{\alpha\in \nR} \B_\alpha
\end{equation*}
(see~{\em e.g.}~\cite{ScheicherThuswaldner03}). In~\cite[Section~4]{GroechenigHaas94} as well as in \cite{ScheicherThuswaldner03} it is explained why $\nR$ is called ``contact set''. The elements of $\nR$ turn out to be neighbors in a tiling of certain approximations $T_n$ of the self-affine tile $T$, which also form tilings w.r.t.~the lattice $\Z^m$ for each $n\ge 0$. However, we will not need this interpretation in the sequel. 

In the graph $G(\nS)$ there cannot occur any {\em sink}. Indeed, if $\alpha\in\nS$ would be a sink, for this $\alpha$ the right hand side of the set equation \eqref{eq:bgammaseteq} would be empty. However, this yields $\B_\alpha=T\cap(T+\alpha)=\emptyset$, a contradiction to $\alpha\in\nS$. The contact graph $G(\nR)$ contains no sink by definition. Scheicher and Thuswaldner \cite{ScheicherThuswaldner03} proved that $\nS$ can be determined by the following algorithm.

\begin{algo}[{{\em cf.}~\cite[Algorithm~3.6]{ScheicherThuswaldner03}}]\label{alg:ST}
For $p\ge 0$ define the nested sequence $(S_p)_{p\ge 0}$ of subsets of $\Z^m$ recursively by $S_0:=\nR$ and $S_p:={\rm Red}(S_{p-1} + S_0)$ for $p\ge 1$, where the occurring sum is the Minkowski sum. Then there exists a smallest $q\ge 0$ such that $S_q=S_{q+1}$. For this $q$ we have $\nS=S_q\setminus\{0\}$.
\end{algo}

\subsection{The contact set}\label{sec:contact}
Let $T$ be an $ABC$-tile. Since $T$ is a $\Z^3$-tile by Lemma~\ref{ABCZ3} we may apply the theory of Section~\ref{sec:21}  to calculate its contact set $\nR$.  In order to define $R_0$ we use the basis $\{(1,0,0)^t, (A,1,0)^t, (B,A,1)^t\}$ of the lattice $\Z^3$.

\begin{lemma}\label{Red(G(R))}
Let $T$ be an $ABC$-tile and let $R_0=\{(0,0,0)^t,\pm(1,0,0)^t, \pm(A,1,0)^t, \pm(B,A,1)^t\}$. Then $R_3=R_2$, \emph{i.e.}, the contact set $\nR$ is equal to $R_2$. In particular,\begin{itemize}
\item[(1)] for 
$1\le A<B$
the set $\nR$ has the $15$ elements
\begin{equation}\label{Rp1}
\begin{split}
\{(0,0,0)^t,\pm(1,0,0)^t,  \pm(A-1,&1,0)^t,  \pm(A,1,0)^t,  \pm(B-A, A-1, 1)^t,
\\
 & \pm(B-A+1,A-1,1)^t,\pm(B-1,A,1)^t, \pm(B,A,1)^t\};
\end{split}
\end{equation}
\item[(2)] for 
$1\le A=B$
the set $\nR$ has the $13$ elements 
\begin{equation}\label{Rp2}
\{(0,0,0)^t,\pm(1,0,0)^t,\pm(A-1,1,0)^t,\pm(A,1,0)^t,\pm(0, A-1, 1)^t,
\pm(A-1,A,1)^t, \pm(A,A,1)^t\}.
\end{equation}
\end{itemize}
\end{lemma}

\begin{proof}
If $1\le A<B$ then let $\nR'$ be the set in \eqref{Rp1}, and if $1\le A=B$ then let $\nR'$ be the set in \eqref{Rp2}.
Each edge $\alpha\xrightarrow{d|d'} \alpha'$ of Table~\ref{table-contact}  as well as its  ``negative version'' $-\alpha\xrightarrow{d'|d}-\alpha'$ exists in $G(\Z^3)$ by the definition of this graph. Thus, by inspection of Table~\ref{table-contact} we see that for each $\alpha \in \nR'$ there is a path of length at most two in $G(\Z^3)$ that starts in $\alpha$ and ends in an element of $R_0$. This implies by \eqref{eq:cnb} that $\nR' \subset R_2$. Also from Table~\ref{table-contact} we see that $\nR'={\rm Red}(\nR')$. Moreover, by direct calculation we gain that each predecessor (in $G(\Z^3)$) of each element of $\nR'$ is contained in $\nR'$. Because of \eqref{eq:cnb} and the fact that $R_0\subset \nR'$, this implies that $\nR \subset \nR'$. Summing up we get ${\rm Red}(\nR')=\nR' \subset R_2\subset R_3 \subset \nR \subset \nR'$ which proves the result.
\end{proof}

Figure~\ref{Gamma2_1} shows the graph $G(\nR \setminus\{0\})$ under the condition $1\le A<B<C$.
The fact that $0\in\nR$ is a natural consequence of the way this set is defined. However, it will often be more convenient for us to work with $\nR\setminus\{0\}$ instead of $\nR$ (like for instance in Figure~\ref{Gamma2_1}).
 
\begin{table}[h]
{\scr 
{\centering
\begin{tabular}{|l|c|c|l|}
\hline 
Edge & Labels&Exists under \\  & &condition 
\\
\hline 
\vspace{-1mm}&&\\ 
$000\rightarrow 000$ &$\{0|0, 1|1,\dots,(C-1)|(C-1)\}$ & --\\
 \hline 
 \vspace{-1mm}&&\\ 
$000\rightarrow 100$ &$\{0|1, 1|2,\dots,(C-2)|(C-1)\}$ & --\\
 \hline 
 \vspace{-1mm}&&\\ 
$100\rightarrow A10$ &$\{0|A, 1|(A+1),\dots,(C-A-1)|(C-1)\}$ & --\\
 \hline 
 \vspace{-1mm}&&\\
$100\rightarrow (A-1)10$ & $\{0|(A-1),1|A,\dots, (C-A)|(C-1)\}$& --\\
\hline 
\vspace{-1mm}&&\\
 $BA1\rightarrow \overline{1}00$&$\{0|C-1\}$&--\\
\hline
\vspace{-1mm}&&\\
 $A10\rightarrow BA1$ &$\{0|B,1|(B+1),\dots,(C-B-1)|(C-1)\}$&  --\\
 \hline 
\vspace{-1mm}&&\\
$A10\rightarrow (B-1)A1$ & $\{0|(B-1),1|B,\dots,(C-B)|(C-1)\}$&-- \\
 \hline 
 \vspace{-1mm}&&\\
$(B-1)A1\rightarrow \overline{A}~\overline{1}0$ & $\{0|(C-A),1|(C-A+1),\dots,(A-1)|(C-1)\}$&-- \\
 \hline 
 \vspace{-1mm}&&\\
$(B-1)A1\rightarrow \overline{A-1}~\overline{1}0$ & $\{0|(C-A+1), 1|(C-A+2),\dots, (A-2)|(C-1)\}$& $A\geq 2$\\
 \hline 
 \vspace{-1mm}&&\\
$(B-A)(A-1)1\rightarrow \overline{B}~\overline{A}~\overline{1}$ & $\{0|(C-B),1|(C-B+1),\dots,(B-1)|(C-1)\}$&-- \\
 \hline 
 \vspace{-1mm}&&\\
$(B-A)(A-1)1\rightarrow \overline{B-1}~\overline{A}~\overline{1}$ & $\{0|(C-B+1),1|(C-B+2),\dots,(B-2)|(C-1)\}$& $B\geq 2$ \\
  \hline 
  \vspace{-1mm}&&\\
$(A-1)10\rightarrow (B-A){(A-1)}1$ & $\{0|(B-A),1|(B-A+1),\dots,(C-B+A-1)|(C-1)\}$&  \\
\hline 
\vspace{-1mm}&&\\
$\begin{matrix}
 (B-A+1)(A-1)1\\ 
 \downarrow \\
\overline{B-A}~\overline{A-1}~\overline{1}\\
 \end{matrix}
$ & $\{0|(C-B+A),1|(C-B+A+1),\dots,(B-A-1)|(C-1)\}$& 
$A\neq B$  \\ 
  \hline
  \vspace{-1mm}&&\\ 
  $\begin{matrix}
 (A-1)10\\ 
 \downarrow \\
(B-A+1){(A-1)}1\\
 \end{matrix}
$ & $\{0|(B-A+1),1|(B-A+2),\dots,(C-B+A-2)|(C-1)\}$
 & $\begin{matrix}
A\neq B
\\
\end{matrix}$ \\
\hline 
\vspace{-1mm}&&\\
$\begin{matrix}
 (B-A+1)(A-1)1\\ 
 \downarrow \\
\overline{B-A+1}~\overline{A-1}~\overline{1}\\
 \end{matrix}
$ & $\{0|(C-B+A-1),1|(C-B+A),\dots,(B-A
)|(C-1)\}$&
 $A\neq B$
 \\
\hline
  \end{tabular}

  \medskip

    \caption{The contact graph $G(\nR)$. The triple $abc$ stands for the node $(a,b,c)^t$ and $\overline{a}=-a$. The last column of the table contains the condition under which the respective edge exists. For each edge $\alpha\xrightarrow{d|d'} \alpha'$ in the table there exists the additional edge $-\alpha\xrightarrow{d'|d}-\alpha'\in G(\nR)$.  \label{table-contact} }
 }}
\end{table}

\begin{figure}[htbp]
\hskip 0.8cm
\xymatrix@C=2.2pc{*[o][F-]{\st{\textoverline{P}} }\ar[rrrr]^[*0.9]{\txt{\footnotesize $A,A+1,...,C-1$}} \ar[ddrr]^(.45)*[@]!LD!/^-4pt/[*0.9]{\txt{\footnotesize $A-1,A,...,C-1$}} & & & & *[o][F-]{\st{\textoverline{Q}} }\ar[rrrr]^[*0.9]{\txt{\footnotesize $B,B+1,...,C-1$}}\ar@/^6ex/[ddddrrrr]^(.35)*[@!-25]!LD!/^-5pt/[*0.9]{\txt{\footnotesize $B-1,B,...,C-1$}} & & & & *[o][F-]{\st{\textoverline{N}} }\ar@/^8ex/[dddddddd]^*[@!-90]!LD!/^2pt/[*0.9]{\txt{\footnotesize $C-1$}}\\
& & & & & & &\\
& & *[o][F-]{\st{\textoverline{Q-P}} }\ar[rr]^(0.47)[*0.8]{\txt{\footnotesize $B-A+1,..,C-1$}}\ar[dddd]^(0.5)[*0.9]
{\rotatebox{-90}{\footnotesize $B-A, ..., C-1$}}& & *[o][F-]{\st{\textoverline{N-Q+P}} }\ar[rr]^(0.53)[*0.8]{\txt{\footnotesize $C-B+A,...,C-1$}} \ar@<1ex>[dddd]^(0.5)[*0.9]{
\rotatebox{-90}{\footnotesize $C-B+A-1,...,C-1$}} & &*[o][F-]{\st{N-Q} }\ar[uurr]_(0.8)*[@]!LD!/^19pt/[*0.9]{\txt{\footnotesize $0,1,...,B-1$}}\ar[ddrr]^(0.5)*[@]!LD!/^7pt/[*0.9]{\txt{\footnotesize $0,1,...,B-2$}} & & &\\
& & & & & & & & &\\
 *[o][F-]{\st{N-P} }\ar@/^6ex/[uuuurrrr]^(0.7)*[@!+23]!RD!/^-7pt/[*0.9]{\txt{\footnotesize $0,1,...,A-1$}}\ar[uurr]_(0.5)*[@]!RD!/^63pt/[*0.9]{\txt{\footnotesize $0,1,...,A-2$}}& & & & & & & & *[o][F-]{\st{\textoverline{N-P}} }\ar[ddll]^(0.4)*[@!37]!LD!/^-10pt/[*0.9]{\txt{\footnotesize $C-A+1,...,C-1$}} \ar@/^6ex/[ddddllll]_(0.6)*[@!23]!LD!/^28pt/[*0.9]{\txt{\footnotesize $C-A,...,C-1$}}& &\\
& & & & & & & & &\\
& &*[o][F-]{\st{\textoverline{N-Q}} }\ar[uull]_(0.4)*[@!-37]!LD!/^-7pt/[*0.9]\txt{{\footnotesize $C-B+1,..., C-1$}}\ar[ddll]^(0.5)*[@!36]!LD!/^-10pt/[*0.9]\txt{\footnotesize $C-B,...,C-1$} & & *[o][F-]{\st{N-Q+P} }\ar@<1ex>[uuuu]^(0.5)[*0.9]{
\rotatebox{90}{\footnotesize $0,...,B-A$}} \ar[ll]^[*0.8]{\txt{\footnotesize $0,...,B-A-1$}}& & *[o][F-]{\st{Q-P} }\ar[ll]^(0.47)[*0.8]{\txt{\footnotesize $0,...,C-B+A-2$}}\ar[uuuu]^(0.5)[*0.9]{\rotatebox{90}{\footnotesize $0,...,C-B+A-1$}}
 & & &\\
& & & & & & & & &\\
*[o][F-]{\st{N} } \ar@/^8ex/[uuuuuuuu]^[*0.9]{\txt{\rotatebox{-90}{\footnotesize $0$}}}& & & & *[o][F-]{\st{Q}  }\ar[llll]^[*0.9]{\txt{\footnotesize $0,1,..., C-B-1$}}\ar@/^6ex/[uuuullll]_(0.20)*[@!-23]!LD!/^-10pt/[*0.9]{\txt{\footnotesize $0,1,..., C-B$}}& & & &*[o][F-]{\st{P} }\ar[llll]^[*0.9]{\txt{\footnotesize $0,1,... ,C-A-1$}}\ar[uull]_(0.35)*[@!-37]!LD!/^-5pt/[*0.9]{\txt{\footnotesize $0,1,... , C-A$}}
}
\caption{The reduced contact graph $G(\nR \setminus \{0\})$ under the condition $1<A<B<C$.
Here we set $P=(1,0,0)^t, ~Q=(A,1,0)^t,~ N=(B,A,1)^t$. To obtain $G(\nR \setminus \{0\})$ under the condition $1=A<B<C$ from the graph in the figure we just remove the edge from $N-P$ to $\overline{Q-P}$ and the edge from $\overline{N-P}$ to $Q-P$.
If, in addition, the conditions of Theorem~\ref{count-neigh-General} are satisfied then this graph coincides with the neighbor graph $G(\nS)$. In this case each vertex $\alpha$ of the depicted graph corresponds to the nonempty $2$-fold intersection $\B_\alpha=T \cap (T + \alpha)$.  
\label{Gamma2_1}}
\end{figure}

\subsection{The neighbor set}\label{sec:neighbor}
We now turn to the proof of Theorem~\ref{count-neigh-General}, {\it i.e.}, we characterize all triples $A,B,C$ with $1\le A\le B < C$ for which $\nS$ has $14$ elements. According to Lemma~\ref{Red(G(R))} we know the set $\nR$ explicitly.  To establish Theorem~\ref{count-neigh-General} we will have to apply one step of Algorithm~\ref{alg:ST}. If $A=B$ it will turn out that already after one step we produce a reduced set that has at least $17$ elements which entails that $\nS$ has at least $16$ elements (since $0$ is to be removed and since the sequence of graphs produced by the algorithm is nested). If $A\not=B$, according to Figure~\ref{Gamma2_1} the contact set has $15$ elements. Thus there will occur the following two cases. In the first case the first step of the algorithm will produce a reduced set with more than $15$ elements. This entails that $\nS$ has more than $14$ elements. In the second case the first step of the algorithm will produce a reduced set with exactly $15$ elements which has to be $\nR$ again (since it has to contain $\nR$). In this case the algorithm stops after one step and we conclude that $\nR\setminus\{0\}=\nS$ has $14$ elements.

We recall that $\nS \supset {\rm Red}(\nR+\nR)\setminus\{0\}$. 
We first deal with the cases $A,B,C$ that satisfy none of the conditions of Theorem~\ref{count-neigh-General}. By taking the complement of the union of conditions (1) and (2) we conclude that we have to deal with the following three cases.
\begin{itemize}
\item[(i)] $1\le A=B < C$,
\item[(ii)] $1\le A < B < C$, $C<2(B-A)+2$, 
\item[(iii)] $1\le A < B < C$, $B<2A-1$, and $C<A+B-2$.
\end{itemize}
Indeed, for each of these cases we have to show that $|\nS|>14$. For (i) this is done in Lemma~\ref{Equal-Condi}, for (ii) it follows from Lemma~\ref{cycle-C1}, and (iii) is covered by Lemma~\ref{cycle-C2}. Thus the following three lemmas imply that $|\nS|>14$ if none of the two conditions of Theorem~\ref{count-neigh-General} is satisfied. 

\begin{lemma}\label{Equal-Condi}
If $1\leq A=B<C$, then ${\rm Red}(\nR+\nR)$ has at least $17$ elements.
\end{lemma}
\begin{proof}
Let $s_1=(2A-1,A+1,1)^t$ and $s_2=(-1, A-1, 1)^t$. We claim that $G(\nR+\nR)$ contains the cycle 
\begin{equation}\label{eq:fig4statesubs}
s_1\xrightarrow{0|C-1}s_2\xrightarrow{A-1|C-A}-s_1\xrightarrow{C-1|0}-s_2\xrightarrow{C-A|A-1}s_1.
\end{equation}
The elements $\pm s_1,\pm s_2$ are elements of $(\nR+\nR)\setminus \nR$ by Lemma~\ref{Red(G(R))}. 
Moreover, the edges claimed in \eqref{eq:fig4statesubs} exist because each label occurring in \eqref{eq:fig4statesubs} is an element of $\D$, and
$M s_1+(C-1,0,0)^t=s_2$ and $M s_2+(C-A,0,0)^t-(A-1,0,0)^t=-s_1$. 
Thus $\pm s_1,\pm s_2\in {\rm Red}(\nR+\nR)$. From Lemma~\ref{Red(G(R))} we know that $\nR$ has $13$ elements and $\pm s_1,\pm s_2\not\in \nR$. Since $\nR \subset {\rm Red}(\nR+\nR)$ this implies that ${\rm Red}(\nR+\nR)$ has at least $17$ elements. 
\end{proof}

\begin{lemma}\label{cycle-C1}
If $1\le A < B < C$ and $C<2(B-A)+2$, then ${\rm Red}(\nR+\nR)$ has at least $17$ elements.
\end{lemma}

\begin{proof}
Let $s=(2(B-A)+2, 2(A-1), 2)^t$. We claim that $G(\nR+\nR)$ contains the cycle 
\begin{equation}\label{eq:figadd1subss-s}
s\xrightarrow{2(B-A)-C+1|C-1}-s\xrightarrow{C-1|2(B-A)-C+1}s.
\end{equation}
From Lemma \ref{Red(G(R))} we easily see that $\pm s\in (\nR+\nR)\setminus\nR$. All the labels occurring in the cycle \eqref{eq:figadd1subss-s} are elements of $\D$ by the conditions in the statement of the lemma. The existence of the cycle now follows from verifying \eqref{eq:arrow} for each edge occurring in \eqref{eq:figadd1subss-s}.
This implies the result as in the previous lemma because $\nR$ has $15$ elements by Lemma~\ref{Red(G(R))} and we exhibited $2$ more elements that belong to ${\rm Red}(\nR+\nR)$.
\end{proof}

\begin{lemma}\label{cycle-C2}
If $1\le A < B < C$, $B<2A-1$, and $C<A+B-2$, then ${\rm Red}(\nR+\nR)$ has at least $18$ elements. 

\end{lemma}
\begin{proof}
Let $t_1=(A+B-2, A+1, 1)^t,~ t_2=(B-2A+1, A-2, 1)^t,~ t_3=(-2B+A+1, 1-2A, -2)^t$.
We claim that $G(\nR+\nR)$ contains the cycle 
\begin{equation}\label{eq:figadd1subs}
t_1 \xrightarrow{2A-B-2|C-1} t_2 \xrightarrow{B-A-1|C-B}t_3 \xrightarrow{C-1|A+B-C-3} t_1.
\end{equation}
The proof is finished in the same way as the proof of Lemma~\ref{cycle-C1}.
\end{proof}

Before we can turn to the case where condition (1) or (2) of Theorem~\ref{count-neigh-General} is satisfied we need an easy preparation.

\begin{lemma}\label{product-neigh}
For $1\leq A<B<C$ let $\nR$ be the contact set of the $ABC$-tile $T=T(M,\D)$. The Minkowski sum $\nR+\nR$ has $65$ elements. We partition $\nR+\nR$ into $2 \times 10$ subsets according to their second and third coordinates. Indeed, let
\begin{itemize}
\item[(1)] ${\bf G1}:=\{(x,2A,2)^t;~~2B-2\leq x\leq 2B\}$;
\item[(2)] ${\bf G2}:=\{(x,A+1,1)^t;~~A+B-2\leq x\leq A+B\}$;
\item[(3)] ${\bf G3}:=\{(x,2A-1,2)^t;~~2B-A-1\leq x\leq 2B-A+1\}$;
\item[(4)] ${\bf G4}:=\{(x,2A-2,2)^t;~~2B-2A\leq x\leq 2B-2A+2\}$;
\item[(5)] ${\bf G5}:=\{(x,2,0)^t;~~2A-2\leq x\leq 2A\}$;
\item[(6)] ${\bf G6}:=\{(x,A-2,1)^t;~~B-2A\leq x\leq B-2A+2\}$;
\item[(7)] ${\bf G7}:=\{(x,0,0)^t;~~ 0\leq x\leq 2\}$;
\item[(8)] ${\bf G8}:=\{(x,A,1)^t;~~B-2\leq x\leq B+1\}$;
\item[(9)] ${\bf G9}:=\{(x,1,0)^t;~~A-2\leq x\leq A+1\}$;
\item[(10)] ${\bf G10}:=\{(x,A-1,1)^t;~~B-A-1\leq x\leq B-A+2\}$.
\end{itemize}
Then $\nR+\nR$ is the union of the $20$ sets\footnote{Note that $(0,0,0)^t$ occurs in ${\bf G7}$ and in $-{\bf G7}$.} $\pm{\bf G1},\dots, \pm{\bf G10}$. Moreover, $\nR$ is a subset of the union of the sets $\pm{\bf G7},\dots, \pm{\bf G10}$.
\end{lemma}

\begin{proof}
This is an immediate consequence of the definition of $\nR$, see Lemma~\ref{Red(G(R))}.
\end{proof}

\begin{lemma}\label{ReduceG^{(2)}} 
The elements of ${\bf \pm G1, \pm G2, \pm G3, \pm G4, \pm G5, \pm G6}$ are not in ${\rm Red}(\nR+\nR)$ if condition (1) or (2) of Theorem~\ref{count-neigh-General} holds.
\end{lemma}
\begin{proof}
We first prove the lemma for $A,B,C$ satisfying condition (1) of Theorem~\ref{count-neigh-General}. We split the set ${\bf G6}$ into ${\bf G6.1}=\{(B-2A,A-2,1)^t\}$, ${\bf G6.2}=\{(B-2A+1,A-2,1)^t\}$, and ${\bf G6.3}=\{(B-2A+2,A-2,1)^t\}$. Now we look at the collection of eight sets of vertices given by
\[
\mathcal{G} =\{ \pm{\bf G1}, \pm{\bf G2}, \pm{\bf G3}, \pm{\bf G4}, \pm{\bf G5}, \pm {\bf G6.1}, \pm{\bf G6.2}, \pm{\bf G6.3}\}.
\] 
We prove the following claim. Suppose that $\gamma$ is contained in some $X\in\mathcal{G}$. Then there exists an edge $\gamma\rightarrow \gamma'\in G(\nR+\nR)$ only if $\gamma'$ is contained in a set $Y\in\mathcal{G}$ such that there is an edge between the vertices $X$ and $Y$ in the graph depicted in Figure~\ref{fig:acyclic123456}. Since this graph contains no cycles this claim will imply the result.
\begin{figure}[h]
\hskip 0cm
\xymatrix{
*+[F-]{\st{$\pm${\bf  G1}} }& & *+[F-]{\st{$\pm${\bf G6.1}} }\ar[ll]& & *+[F-]{\st{$\pm${\bf G2}} }\ar[ll]\\
&& *+[F-]{\st{$\pm${\bf G5}} }\ar[rr]\ar[d]\ar[ull]&& *+[F-]{\st{$\pm${\bf G3}} }\ar[u]\\
*+[F-]{\st{$\pm${\bf G6.3}} }\ar[rr]&&*+[F-]{\st{$\pm${\bf G4}} }\ar[lluu]\ar[rru]&&*+[F-]{\st{$\pm${\bf G6.2}} }\ar[u]\\
}
\caption{The possible edges leading away from the sets of vertices $\pm{\bf G1}$, $\pm{\bf G2}$, $\pm{\bf G3}$, $\pm{\bf G4}$, $\pm{\bf G5}$, $\pm {\bf G6.1}$, $\pm{\bf G6.2}$, and $\pm{\bf G6.3}$. \label{fig:acyclic123456}}
\end{figure}
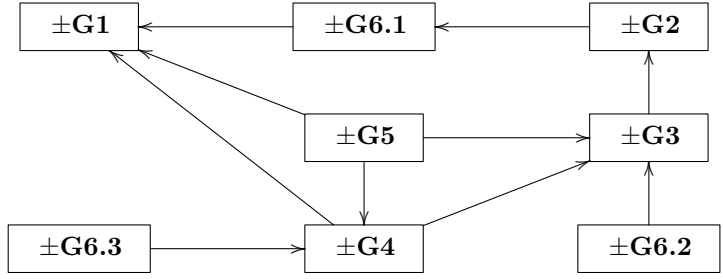

The case {\bf $\pm{\bf G1}$}: The elements in ${\bf G1}$ are sinks in $G(\nR+\nR)$ under condition~(1).
Indeed, let $s=(x, 2A, 2)^t\in {\bf G1}$, then $Ms=(-2C, x-2B, 0)^t$. Since $2B-2 \le x \le 2B$, the possible successors of $s$ are of the form 
$s'\in Ms+\D-\D= \{(-2C+d, x-2B, 0)^t;\; d\in \D-\D\}$ with 
$2B-2 \le x \le 2B$.
According to its second and third coordinates, in the cases $x=2B-1$ and  $x=2B$ the element $s' \in \nR+\nR$ can only belong to $-{\bf G9}$ and ${\bf G7}$, respectively. However, as the first coordinate of $s'$ varies between $-3C+1$ and $-C-1$, this is impossible for both cases. Hence, for these choices of $x$ the element $s$ is a sink of $G(\nR+\nR)$. For the case $x=2B-2$ the element $s'\in \nR+\nR$ can only be an element of $-{\bf G5}$. However, since $C>B\geq 2A-1$, we have $-C-1<-2A$ this is impossible also in this case. Thus this case is done since by symmetry of $\nR+\nR$ also the elements in $-{\bf G1}$ are sinks in $G(\nR+\nR)$ under condition~(1).

The case {\bf $\pm{\bf G2}$}: The elements in ${\bf G2}$ can only have successors in $G(\nR+\nR)$ that are contained in $\pm{\bf G6.1}$ under condition~(1). To prove this let $s=(x,A+1,1)^t$, then a possible successor of $s$ is of the form
$
s'=Ms+\D-\D=\{(d-C, x-B, 1)^t;~d\in \D-\D\} 
$
with 
$A+B-2 \le x \le A+B$.
For the case $x=A+B-2$, looking at the second and third coordinate, the successor $s'$ can only be contained in ${\bf G6}$. The first coordinate $d-C$ of $s'$ satisfies $1-2C\le d-C \le -1$. Since condition (1) is in force, $B-2A \ge -1$, thus $s'$ is in ${\bf G6}$ only if $B=2A-1$. In this case, $d-C=-1=B-2A$, hence, $s'\in{\bf G6.1}$.
For $x=A+B-1$, the successor $s'$ can only fall into ${\bf G10}$. This would imply $-1=B-A-1$ and, hence, $A=B$ which contradicts condition (1). Thus in this case we have no successor. 
Finally, for $x=A+B$, the possible successor $s'$ can only be contained in ${\bf G8}$ by its second and third coordinate. But by its first coordinate also this possibility is excluded. Again, this case is done by symmetry.

So far we proved the claim for the edges leading away from $\pm{\bf G1}$ and $\pm{\bf G2}$ in Figure~\ref{fig:acyclic123456}. The remaining cases are routine calculations of the same kind and we omit them.  

Condition $(2)$ of Theorem~\ref{count-neigh-General} can be treated in the same way. In this case we have to subdivide the relevant vertices into nine sets. The corresponding graph, which is acyclic again, is depicted in Figure~\ref{fig:con2}.
\begin{figure} 
\hskip 0cm
\xymatrix{
& &&&\\
*+[F-]{\st{$\pm${\bf G5.1}} }
 \ar@{}[r]^(.25){}="a"^(.54){}="b" \ar "a";"b"
&*+[F-]{\st{$\pm\left({\bf G1 \setminus G1.1}\right)$} }&*+[F-]{\st{$\pm${\bf G6}} }\ar[d]
\ar[dl]
 \ar@{}[l]^(.21){}="a"^(.52){}="b" \ar "a";"b"
 \ar@{}[lld]^(.1){}="a"^(.87){}="b" \ar "a";"b"
&*+[F-]{\st{$\pm${\bf G2.1}} }\ar[l]\\
*+[F-]{\st{$\pm${\bf G1.1}} }\ar[u]&*+[F-]{\st{$\pm${\bf G4}} }\ar[u]\ar[l]\ar[r]&*+[F-]{\st{$\pm${\bf G3}} }
 \ar@{}[rd]^(.19){}="a"^(.8){}="b" \ar "a";"b"
&&\\
&*+[F-]{\st{$\pm\left({\bf G5 \setminus G5.1 }\right)$} }\ar[u]\ar[ur]&& *+[F-]{\st{$\pm\left({\bf G2 \setminus G2.1 }\right)$} }\\
}
\caption{The graph corresponding to condition (2). Here we use the additional notations ${\bf G1.1}=\{(2B-2,2A,2)^t\},~{\bf G2.1}=\{(A+B-2,A+1,1)^t\}, ~{\bf G5.1}=\{(2A, 2, 0)^t\}$.\label{fig:con2}}
\end{figure}
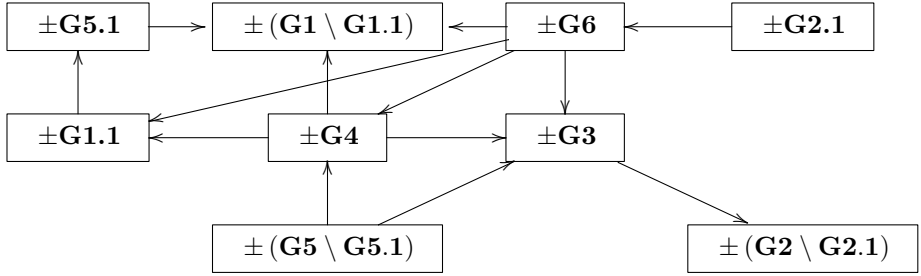 
\end{proof}

We are now able to finish the proof of Theorem~\ref{count-neigh-General}.

\begin{proof}[Proof of Theorem~\ref{count-neigh-General}]
To prove the ``only if'' part, we have to show that $|\nS| >14$ if none of the two conditions of the theorem are in force. Because $\nS \supset {\rm Red}(\nR+\nR)\setminus\{0\}$, this follows immediately from Lemmas~\ref{Equal-Condi},~\ref{cycle-C1}, and~\ref{cycle-C2}.

To prove the ``if'' part, we apply Algorithm~\ref{alg:ST}, and the first step is to calculate ${\rm Red}(\nR+\nR)$. From Lemma \ref{product-neigh}, we already know that $\nR+\nR$ has $65$ elements. Since  $\nR\setminus\{0\}$ has $14$ elements by Lemma~\ref{Red(G(R))} we have to show that ${\rm Red}(\nR+\nR) = \nR$. For this it suffices to prove that no element of $(\nR+\nR) \setminus \nR$ is contained in ${\rm Red}(\nR+\nR)$.

By Lemma~\ref{ReduceG^{(2)}}, none of the elements contained in $\pm{\bf G1} \cup \dots \cup \pm{\bf G6}$ is an element of ${\rm Red}(\nR+\nR)$. Thus it remains to show that each element contained in $\pm({\bf G7} \cup {\bf G8} \cup {\bf G9} \cup {\bf G10}) \setminus \nR$ is not contained in ${\rm Red}(\nR+\nR)$. By symmetry we can confine ourselves to proving the claim that $({\bf G7} \cup {\bf G8} \cup {\bf G9} \cup {\bf G10}) \setminus \nR$  does not contain an element of ${\rm Red}(\nR+\nR)$.
 
Assume that condition $(1)$ of the theorem is in force. For condition $(2)$, we can prove the result in the same way.

We start with ${\bf G7} \setminus \nR = \{(2,0,0)^t\}$. Let $s=(2, 0, 0)^t$. Then a possible successor $s'$ of $s$ in $G(\nR+\nR)$ must satisfy 
$
s'\in Ms+\D-\D=\{(d,2,0)^t; d\in \D-\D\}.
$
By the second and third coordinate of $Ms+\D-\D$, we know that $s'$ has to belong to ${\bf G5}$. So it cannot be in ${\rm Red}(\nR+\nR)$ by Lemma~\ref{ReduceG^{(2)}}. 

For the elements $s\in \{(B-2, A, 1)^t,(B+1, A, 1)^t\}={\bf G8} \setminus\nR$ the successor has to be of the form
$
s' = Ms+\D-\D=
\{(d-C, x-B, 0)^t;~d\in \D-\D\}$ 
with $x\in \{B-2,B+1\}$.
For $x=B-2$, we have $s'\in -{\bf G5}$, so $(B-2, A, 1)^t\not\in{\rm Red}(\nR+\nR)$ by Lemma~\ref{ReduceG^{(2)}}. 
For $x=B+1$, the successor of $s$ can only be contained in ${\bf G9}$. However, the first coordinate of the elements of ${\bf G9}$  varies between $A-2$ and $A+1$, thus $s$ is a sink if $A\geq 2$ which always holds under condition (1) of the theorem.

For the elements $s\in\{(A-2, 1, 0)^t,(A+1,1,0)^t\}={\bf G9} \setminus\nR$, the successor has to be of the form
 $
 s'\in Ms+\D-\D=\{(d, x, 1)^t;~d\in \D-\D\}$
with $x\in \{A-2,A+1\}$.
This implies that $s'\in {\bf G2} \cup {\bf G6}$ and, hence, $s\not\in {\rm Red}(\nR+\nR)$ by Lemma~\ref{ReduceG^{(2)}}.

For the elements $s\in\{(B-A-1, A-1, 1)^t,~(B-A+2, A-1, 1)^t\}={\bf G10} \setminus\nR$ the successor has to be of the form
$
s'\in Ms+\D-\D=\{(d-C, x-B, -1)^t;~d\in \D-\D\}$ 
with $x\in\{B-A-1,B-A+2\}$.
This implies that $s'\in (-{\bf G6}) \cup (-{\bf G2})$ and, hence, $s\not\in {\rm Red}(\nR+\nR)$ by Lemma~\ref{ReduceG^{(2)}}. Summing up, we proved the claim.
\end{proof}

\subsection{The directed graphs of multiple intersections}\label{sec:34}
Let $T$ be a $\Z^m$-tile and let $\nS$ be the set of neighbors of $T$. 
For $\ell \ge 1$, the union of all \emph{$(\ell+1)$-fold intersections with $T$} is then given by 
\begin{equation}\label{l-fold-union}
\mathcal{I}_{\ell}=\bigcup_{\{\alpha_1,\dots,\alpha_\ell\}\subset\nS } \B_{\alpha_1,\dots,\alpha_{\ell}},
\end{equation}  
where the union is extended over all subsets of $\nS$ containing $\ell$ (pairwise distinct) elements.
We can subdivide  $\B_{\alpha_1,\dots,\alpha_{\ell}}$ by
\begin{equation}\label{eq:setqeB}
\begin{split}
\B_{\alpha_1,\dots,\alpha_{\ell}}&=M^{-1}\Big((T+\D)\cap (T+\D+M\alpha_1)\cap \dots\cap (T+\D+M\alpha_{\ell})\Big) \\
&=M^{-1}\Big(\bigcup_{d,d_1,\dots,d_{\ell}\in \D}(T\cap(T+M\alpha_1+d_1-d)\cap\dots\cap(T+M\alpha_{\ell}+d_{\ell}-d))+d\Big)\\
&=M^{-1}\Big(\bigcup_{\begin{subarray}{c}
d,d_1,\dots,d_{\ell}\in \D\\
\alpha_i'=M\alpha_i+d_i-d\\
i=1,2,\dots,\ell
\end{subarray}
}(\B_{\alpha_1',\alpha_2',\dots,\alpha_{\ell}'}+d)\Big)
\end{split}
\end{equation}
and by Definition \ref{Graph} we can rewrite this as 
\begin{equation}\label{L-foldGra1}
\begin{split}
\B_{\alpha_1,\dots,\alpha_{\ell}}&=\bigcup_{d\in \D}\bigcup_{
d_1,\dots,d_{\ell}\in \D
}\bigcup_{\begin{subarray}{c}
\alpha_i\xrightarrow{d|d_i}  \alpha_i'\\
i=1,2,\dots,\ell
\end{subarray}
}M^{-1}(\B_{\alpha_1',\alpha_2',\dots,\alpha_{\ell}'}+d).
\end{split}
\end{equation}
Of course, $\B_{\alpha_1,\ldots,\alpha_\ell}$ can be nonempty only if $\{\alpha_1,\ldots,\alpha_\ell\} \subset \nS$. 
The sets $\mathcal{I}_{\ell}$ can be determined by using the following graphs.

\begin{definition}\label{def:powerG2G3G4}
Let $\Gamma\subset\Z^m$. The $\ell$-fold power $G_\ell(\Gamma):=\times_{j=1}^\ell G(\Gamma)$ is defined by recursively removing all sinks from the following graph $G'_\ell(\Gamma)$:
\begin{itemize}
\item The vertices of $G'_\ell(\Gamma)$ are the sets $\{\alpha_1,\dots,\alpha_\ell\}$ consisting of $\ell$ (pairwise distinct) elements $\alpha_i$ of $G(\Gamma)$.
\item There exists an edge
$$\{\alpha_{11},\dots,\alpha_{1\ell}\}\xrightarrow{d}\{\alpha_{21},\dots,\alpha_{2\ell}\}$$
in $G'_\ell(\Gamma)$ if and only if there exist the edges
$$\alpha_{1i}\xrightarrow{d|d_i}\alpha_{2i} \quad (1\leq i\leq \ell)$$
in $G(\Gamma)$ for certain $d_1,\dots,d_\ell\in \D$.
\end{itemize}
\end{definition}
Using this definition we can write \eqref{L-foldGra1} as
\begin{equation}\label{L-foldGra2}
\begin{split}
\B_{\alpha_1,\dots,\alpha_{\ell}}&=\bigcup_{
\begin{subarray}{c}
d \in \D,\, \{\alpha_1',\dots,\alpha_{\ell}'\} \subset \nS \\
\{\alpha_1,\dots,\alpha_{\ell}\}\xrightarrow{d}\{\alpha_1',\dots,\alpha_{\ell}'\}\in \times_{j=1}^{\ell}G(\nS)
\end{subarray}
}
M^{-1}(\B_{\alpha_1',\alpha_2',\dots,\alpha_{\ell}'}+d)
\end{split}
\end{equation}
which can be regarded as the defining equation for the collection of nonempty compact sets $\{ \B_{\Bold{\alpha}};\, \Bold{\alpha} \in  \times_{j=1}^{\ell}G(\nS)\}$ as attractor of a graph-directed iterated function system (in the sense of Mauldin and Williams~\cite{MauldinWilliams88}) directed by the graph $\times_{j=1}^{\ell}G(\nS)$. We will often need the $k$-fold iteration of this set equation. To write this iteration in a convenient way we define the functions 
\[
f_d:\R^m \to \R^m  \quad x \mapsto M^{-1}(x+d) \qquad(d\in \D),
\]
which are contractions w.r.t.\ some suitable norm because $M$ is an expanding matrix. Since we will often deal with compositions of these functions, we will use the abbreviation 
\[
f_{d_1d_2\dots d_{k}}=\begin{cases} f_{d_1}\circ\cdots\circ f_{d_{k}}, & k > 0, \\
\mathrm{id},& k=0
\end{cases}
\]
for $d_1,\ldots, d_{k}\in\D$. With this notation we get
 
\begin{equation}\label{L-foldGra3}
\begin{split}
\B_{\alpha_1,\dots,\alpha_{\ell}}&=\bigcup_{\{\alpha_1,\dots,\alpha_{\ell}\}\xrightarrow{d}\{\alpha_1',\dots,\alpha_{\ell}'\}\in \times_{j=1}^{\ell}G(\nS)\
}f_d(\B_{\alpha_1',\alpha_2',\dots,\alpha_{\ell}'})\\
&=\bigcup_{\{\alpha_1,\dots,\alpha_{\ell}\}\xrightarrow{d_1}\cdots\xrightarrow{d_k}\{\alpha_1^{(k)},\dots,\alpha_{\ell}^{(k)}\}\in \times_{j=1}^{\ell}G(\nS)\
}f_{d_1\dots d_k}(\B_{\alpha_1^{(k)},\dots,\alpha_{\ell}^{(k)}}),
\end{split}
\end{equation}
where the latter union is extended over all walks of length $k$ in $\times_{j=1}^{\ell}G(\nS)$ starting at $\{\alpha_1,\dots,\alpha_{\ell}\}$.
We can now characterize $\mathcal{I}_{\ell}$ as follows (see \cite[Appendix]{StrichartzWang99} or \cite[Proposition~6.2]{ScheicherThuswaldner03}).
\begin{proposition}\label{Char-boundary}
Let $\ell\geq 1$ and choose $\alpha_{01},\dots,\alpha_{0\ell} \in \Z^m\setminus\{0\}$ pairwise distinct. Then the following three assertions are equivalent.
\begin{itemize}
\item[(1)] $$x=\sum_{j\geq 1} M^{-j}d_j\in \B_{\alpha_{01},\dots,\alpha_{0\ell}}.$$
\item[(2)] There exists an infinite walk 
$$\{\alpha_{01},\dots,\alpha_{0\ell}\}\xrightarrow{d_1}\{\alpha_{11},\dots,\alpha_{1\ell}\}\xrightarrow{d_2}\{\alpha_{21},\dots,\alpha_{2\ell}\}\xrightarrow{d_3}\cdots$$
in $\times_{r=1}^{\ell}G(\nS)$.
\item[(3)] There exist $\ell$ infinite walks 
$$\alpha_{0i}\xrightarrow{d_1}\alpha_{1i}\xrightarrow{d_2}\alpha_{2i}\xrightarrow{d_3}\cdots\quad (1\leq i\leq \ell)$$
in $G(\nS)$.
\end{itemize}
\end{proposition}

The set equation \eqref{L-foldGra3} yields a sequence of collections of sets that cover $\B_{\alpha_1,\dots,\alpha_{\ell}}$. Namely, for $\Bold{\alpha}^{(0)}=\{\alpha_1,\ldots,\alpha_\ell\} \in \times_{j=1}^{\ell}G(\nS)$ we define 
\begin{equation}\label{eq:CkAlpha}
\mathcal{C}_k(\Bold{\alpha}^{(0)}) :=
\left\{
f_{d_1\dots d_{k-1}}(\B_{\Bold{\alpha}^{(k-1)}}) ;\; \Bold{\alpha}^{(0)} \xrightarrow{d_1}\cdots\xrightarrow{d_{k-1}}\Bold{\alpha}^{(k-1)}\in \times_{j=1}^{\ell}G(\nS)
\right\}
\end{equation}
and set
\begin{equation}\label{eq:Cellk}
\mathcal{C}^{(\ell)}_k = \bigcup_{\Bold{\alpha}\in \times_{j=1}^{\ell}G(\nS)}\mathcal{C}_k(\Bold{\alpha}).
\end{equation}
We will call $\mathcal{C}_k(\Bold{\alpha})$ the {\em collection of $(k-1)$-th subdivisions} of $\B_{\Bold{\alpha}}$. If $k=2$ we will just call it the {\em collection of subdivisions} of $\B_{\Bold{\alpha}}$. The elements of these collections will be called {\em subtiles} of $\B_{\Bold{\alpha}}$. The collection of subdivisions of $\B_{\Bold{\alpha}} \cup \B_{\Bold{\alpha'}}$ is the union of the collections of subdivisions of $\B_{\Bold{\alpha}}$ and $\B_{\Bold{\alpha'}}$. It should now be clear what we mean by the collections of ($(k-1)$-th) subdivisions of a set $X=M^{-r}(\B_{\Bold{\alpha}} + a)$ with $a\in \Z^m$ and $r\in \N$.

We will need the following lemma.

\begin{lemma}\label{lem:piecesresults}
\mbox{}
\begin{enumerate}
\item \label{lem:piecesresults1}
For $\Bold{\alpha}
\in \times_{j=1}^{\ell}G(\nS)$  and $k\ge 1$ the collection $\mathcal{C}_k(\Bold{\alpha})$ forms a covering of $\B_{\Bold{\alpha}}$.
\item \label{lem:piecesresults2}
Let $k\ge 1$ be given and let
$
X_1,X_2 \in \mathcal{C}_k^{(\ell)}
$
be distinct. Then the intersection $X_1\cap X_2$ is either empty or there exist $\ell' > \ell$, $c\in \Z^m$, and  $\Bold{\alpha}\in \times_{j=1}^{\ell'}G(\nS)$ with $X_1\cap X_2 = M^{-k+1}(\B_{\Bold{\alpha}} + c)$.
\end{enumerate}
\end{lemma}

\begin{proof}
Assertion~\eqref{lem:piecesresults1} follows immediately from \eqref{L-foldGra3} and the definition of $\mathcal{C}_k(\Bold{\alpha})$.

To prove assertion~\eqref{lem:piecesresults2} we conclude from \eqref{L-foldGra3} that  $X_i = M^{-k+1}(T + \beta_{i0}) \cap \cdots \cap M^{-k+1}(T + \beta_{i\ell})$ holds with $\beta_{ij}\in \Z^m$ for $i\in \{1,2\}$ and $j\in\{0,\ldots,\ell\}$. Here $\beta_{ij}$ are pairwise distinct for fixed $i$ and $j\in\{0,\ldots,\ell\}$.
Since $X_1$ and $X_2$ are distinct elements of $\mathcal{C}_k^{(\ell)}$ there exists $\ell' > \ell$  and $\ell'+1$ distinct elements 
\[
\gamma_0,\ldots, \gamma_{\ell'} \in \{\beta_{10},\ldots \beta_{1\ell},\beta_{20},\ldots,\beta_{2\ell}\}
\]
such that $X_1 \cap X_2  = M^{-k+1}(T + \gamma_0) \cap \cdots \cap M^{-k+1}(T + \gamma_{\ell'})$ and, hence, $X_1 \cap X_2$ is either empty or an element of $\mathcal{C}_k(\Bold{\alpha})$ for some $\Bold{\alpha}\in \times_{j=1}^{\ell'}G(\nS)$ as claimed.
\end{proof}

\begin{table}
{\scr 
{\centering
\begin{tabular}{|l|c|c|l|}
\hline
\multicolumn{1}{|c|}{Vertex}&Successors& Label&\multicolumn{1}{|c|}{Conditions}\\
\hline
\vspace{-2mm}&&&\\
{\multirow{3}{*}{$\mybinom{{Q-P}}{N-P}$}} &{$\mybinom{\overline{Q}}{N-Q}$}&$\{0,1,\dots,A-1\}$&- \\\cline{2-3}\cline{3-3}\cline{4-3}\vspace{-2mm}&&&\\
&$\mybinom{\overline{Q-P}}{N-Q}$&$\{0,1,\dots, A-2\}$&$A\geq 2$\\\cline{2-3}\cline{3-3}\cline{4-3}\vspace{-2mm}&&&\\
&{$\mybinom{\overline{Q-P}}{N-Q+P}$}&$\{0,1,\dots,A-2\}$&$A\geq 2$\\

\hline
\vspace{-2mm}&&&\\
{\multirow{3}{*}{$\mybinom{{\overline{P}}}{Q-P}$}} &{$\mybinom{\overline{Q}}{N-Q}$}&$\{A,A+1,\dots,C-B+A-1\}$&- \\\cline{2-3}\cline{3-3}\cline{4-3}\vspace{-2mm}&&&\\
&$\mybinom{\overline{Q-P}}{N-Q}$&$\{A-1,A,\dots, C-B+A-1\}$&-\\\cline{2-3}\cline{3-3}\cline{4-3}\vspace{-2mm}&&&\\
&{$\mybinom{\overline{Q-P}}{N-Q+P}$}&$\{A-1,A,\dots,C-B+A-2\}$&- \\

\hline
\vspace{-2mm}&&&\\
{\multirow{3}{*}{$\mybinom{\overline{N-Q+P}}{\overline{P}}$}} &{$\mybinom{\overline{Q}}{N-Q}$}&$\{C-B+A,\dots,C-1\}$&- \\\cline{2-3}\cline{3-3}\cline{4-3}\vspace{-2mm}&&&\\
&$\mybinom{\overline{Q-P}}{N-Q}$&$\{C-B+A,\dots, C-1\}$&- \\\cline{2-3}\cline{3-3}\cline{4-3}\vspace{-2mm}&&&\\
&{$\mybinom{\overline{Q-P}}{N-Q+P}$}&$\{C-B+A-1,\dots,C-1\}$&- \\

\hline
\vspace{-2mm}&&&\\
{\multirow{3}{*}{$\mybinom{P}{Q}$}} &{$\mybinom{{Q-P}}{N-P}$}&$\{0,1,\dots,C-B\}$&- \\\cline{2-3}\cline{3-3}\cline{4-3}\vspace{-2mm}&&&\\
&$\mybinom{{Q-P}}{N}$&$\{0,1,\dots, C-B-1\}$&-\\\cline{2-3}\cline{3-3}\cline{4-3}\vspace{-2mm}&&&\\
&{$\mybinom{{Q}}{N}$}&$\{0,1,\dots,C-B-1\}$&-\\
\hline

\vspace{-2mm}&&&\\
{\multirow{3}{*}{$\mybinom{{\overline{N-Q}}}{P}$}} &{$\mybinom{{Q-P}}{N-P}$}&$\{C-B+1,\dots,C-A\}$&- \\\cline{2-3}\cline{3-3}\cline{4-3}\vspace{-2mm}&&&\\
&$\mybinom{{Q-P}}{N}$&$\{C-B,\dots, C-A\}$&-\\\cline{2-3}\cline{3-3}\cline{4-3}\vspace{-2mm}&&&\\
&{$\mybinom{{Q}}{N}$}&$\{C-B,\dots,C-A-1\}$&- \\
\hline

\vspace{-2mm}&&&\\
{\multirow{3}{*}{$\mybinom{\overline{N-P}}{\overline{N-Q}}$}} &{$\mybinom{{Q-P}}{N-P}$}&$\{C-A+1,\dots,C-1\}$&$A\geq 2$\\\cline{2-3}\cline{3-3}\cline{4-3}\vspace{-2mm}&&&\\
&$\mybinom{{Q-P}}{N}$&$\{C-A+1,\dots, C-1\}$&$A\geq 2$ \\\cline{2-3}\cline{3-3}\cline{4-3}\vspace{-2mm}&&&\\
&{$\mybinom{{Q}}{N}$}&$\{C-A,\dots,C-1\}$&- \\
\hline

\vspace{-2mm}&&&\\
{\multirow{3}{*}{$\mybinom{{N-Q}}{N-Q+P}$}} &{$\mybinom{\overline{N-Q+P}}{\overline{N}}$}&$\{0,1,\dots,B-A\}$&- \\\cline{2-3}\cline{3-3}\cline{4-3}\vspace{-2mm}&&&\\
&$\mybinom{\overline{N-Q}}{\overline{N}}$&$\{0,1,\dots, B-A-1\}$&-\\\cline{2-3}\cline{3-3}\cline{4-3}\vspace{-2mm}&&&\\
&{$\mybinom{\overline{N-Q}}{\overline{N-P}}$}&$\{0,1,\dots,B-A-1\}$&-\\
\hline

\vspace{-2mm}&&&\\
{\multirow{3}{*}{$\mybinom{{\overline{Q-P}}}{N-Q}$}} &{$\mybinom{\overline{N-Q+P}}{\overline{N}}$}&$\{B-A+1,\dots,B-1\}$&$A\geq 2$ \\\cline{2-3}\cline{3-3}\cline{4-3}\vspace{-2mm}&&&\\
&$\mybinom{\overline{N-Q}}{\overline{N}}$&$\{B-A,\dots, B-1\}$&-\\\cline{2-3}\cline{3-3}\cline{4-3}\vspace{-2mm}&&&\\
&{$\mybinom{\overline{N-Q}}{\overline{N-P}}$}&$\{B-A,\dots,B-2\}$&$A\geq 2$ \\
\hline

\vspace{-2mm}&&&\\
{\multirow{3}{*}{$\mybinom{\overline{Q-P}}{\overline{Q}}$}} &{$\mybinom{\overline{N-Q+P}}{\overline{N}}$}&$\{B,B+1,\dots,C-1\}$&- \\\cline{2-3}\cline{3-3}\cline{4-3}\vspace{-2mm}&&&\\
&$\mybinom{\overline{N-Q}}{\overline{N}}$&$\{B,B+1,\dots, C-1\}$&- \\\cline{2-3}\cline{3-3}\cline{4-3}\vspace{-2mm}&&&\\
&{$\mybinom{\overline{N-P}}{\overline{N-Q}}$}&$\{B-1,\dots,C-1\}$&- \\
\hline

\vspace{-2mm}&&&\\
$\mybinom{\overline{Q-P}}{N-Q+P}$&$\mybinom{\overline{N-Q+P}}{\overline{N-Q}}$&$B-A$&-\\
\hline
\vspace{-2mm}&&&\\
$\mybinom{\overline{Q}}{N-Q}$&$\mybinom{\overline{N}}{\overline{N-P}}$&$B-1$&-\\
\hline
\vspace{-2mm}&&&\\
$\mybinom{\overline{N}}{\overline{N-P}}$&$\mybinom{{P}}{{Q}}$&$C-1$&-\\
\hline
\vspace{-2mm}&&&\\
$\mybinom{{Q-P}}{N}$&$\mybinom{\overline{P}}{{N-Q}}$&$0$&-\\
\hline
\vspace{-2mm}&&&\\
$\mybinom{{Q}}{N}$&$\mybinom{\overline{P}}{{N-P}}$&$0$&-\\
\hline
\vspace{-2mm}&&&\\
$\mybinom{\overline{P}}{{N-P}}$&$\mybinom{\overline{Q}}{\overline{Q-P}}$&$A-1$&-\\
\hline
\vspace{-2mm}&&&\\
$\mybinom{\overline{N}}{\overline{N-Q+P}}$&$\mybinom{{P}}{{N-Q+P}}$&$C-1$&-\\
\hline
\vspace{-2mm}&&&\\
$\mybinom{\overline{N}}{\overline{N-Q}}$&$\mybinom{{P}}{{N}}$&$C-1$&-\\
\hline
\vspace{-2mm}&&&\\
$\mybinom{P}{N}$&$\mybinom{\overline{P}}{{Q-P}}$&$0$&-\\
\hline
\end{tabular}
}}
\medskip
\caption{The graph $G_2(\nS)$ of triple intersections. To each edge $\Bold{\alpha}\xrightarrow{d}\Bold{\alpha}'$ in this table there exists a ``negative version'' $-\Bold{\alpha}\xrightarrow{C-1-d}-\Bold{\alpha}' \in G_2(\nS)$. Here we set $P=(1, 0, 0)^t, Q=(A, 1, 0)^t, N=(B, A, 1)^t$.
\label{tab:triple-graph}}
\end{table}

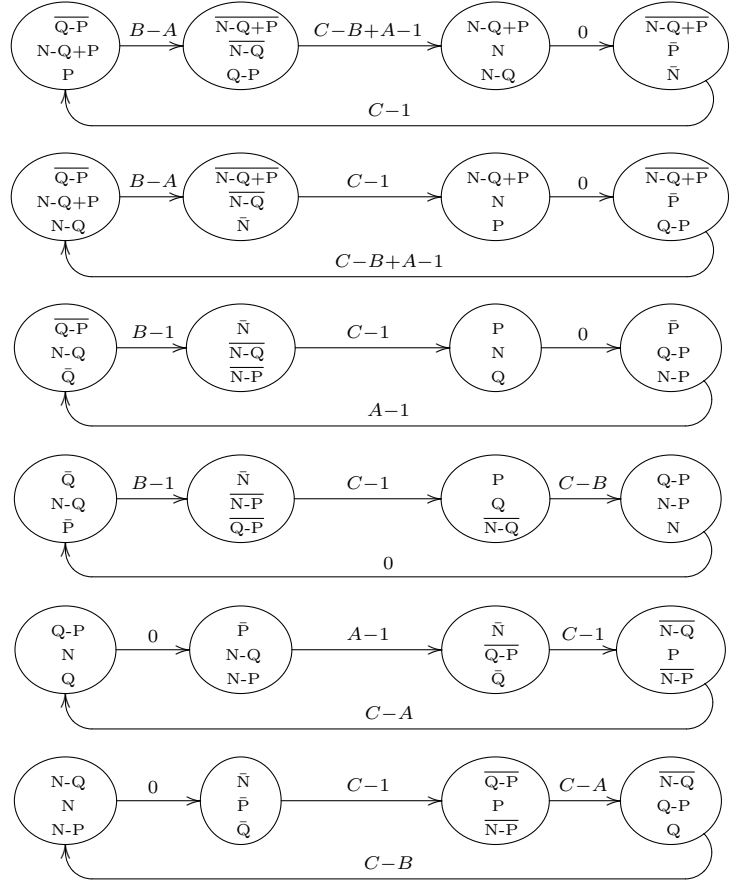
\begin{figure}[htbp]
\hskip 0.5cm
\xymatrix{*[o][F-]{\statev{\textoverline{Q-P}}{N-Q+P}{P} }\ar[r]^{B-A}&*[o][F-]{\statev{\textoverline{N-Q+P}}{\textoverline{N-Q}}{Q-P} }\ar[rr]^{C-B+A-1}&&*[o][F-]{\statev{N-Q+P}{N}{N-Q} }\ar[r]^{0}&*[o][F-]{\statev{\textoverline{N-Q+P}}{\=P}{\=N} }\ar `dr_l[ll] `_u[llll]_{C-1}[llll] \\
*[o][F-]{\statev{\textoverline{Q-P}}{N-Q+P}{N-Q} }\ar[r]^{B-A}&*[o][F-]{\statev{\textoverline{N-Q+P}}{\textoverline{N-Q}}{\=N} }\ar[rr]^{C-1}&&*[o][F-]{\statev{N-Q+P}{N}{P} }\ar[r]^{0}&*[o][F-]{\statev{\textoverline{N-Q+P}}{\=P}{Q-P} }\ar `dr_l[ll]`_u[llll]_{C-B+A-1}[llll]\\
*[o][F-]{\statev{\textoverline{Q-P}}{\mbox{\;\;}N-Q\mbox{\;\;}}{\=Q} }\ar[r]^{B-1}&*[o][F-]{\statev{\=N}{\mbox{\;\;}\textoverline{N-Q}\mbox{\;\;}}{\textoverline{N-P}} }\ar[rr]^{C-1}&&*[o][F-]{\statev{P}{\mbox{\;\;\;}N\mbox{\;\;\;}}{Q} }\ar[r]^{0}&*[o][F-]{\statev{\=P}{\mbox{\;\;}Q-P\mbox{\;\;}}{N-P} }\ar `dr_l[ll]`_u[llll]_{A-1}[llll]\\
*[o][F-]{\statev{\=Q}{\mbox{\;\;}N-Q\mbox{\;\;}}{\=P} }\ar[r]^{B-1}&*[o][F-]{\statev{\=N}{\mbox{\;\;}\textoverline{N-P}\mbox{\;\;}}{\textoverline{Q-P}} }\ar[rr]^{C-1}&&*[o][F-]{\statev{P}{\mbox{\;\;}Q\mbox{\;\;}}{\mbox{\;\;}\textoverline{N-Q}\mbox{\;\;}} }\ar[r]^{C-B}&*[o][F-]{\statev{Q-P}{\mbox{\;\;}N-P\mbox{\;\;}}{N} }\ar `dr_l[ll]`_u[llll]_{0}[llll]\\
*[o][F-]{\statev{\mbox{\;\;}Q-P\mbox{\;\;}}{N}{Q} }\ar[r]^{0}&*[o][F-]{\statev{\=P}{\mbox{\;\;}N-Q\mbox{\;\;}}{N-P} }\ar[rr]^{A-1}&&*[o][F-]{\statev{\=N}{\mbox{\;\;}\textoverline{Q-P}\mbox{\;\;}}{\=Q} }\ar[r]^{C-1}&*[o][F-]{\statev{\mbox{\;\;}\textoverline{N-Q}\mbox{\;\;}}{\mbox{\;\;}P\mbox{\;\;}}{\textoverline{N-P}} }\ar `dr_l[ll]`_u[llll]_{C-A}[llll]\\
*[o][F-]{\statev{\mbox{\;\;}N-Q\mbox{\;\;}}{N}{N-P} }\ar[r]^{0}&*[o][F-]{\statev{\mbox{\;\;}\=N\mbox{\;\;}}{\=P}{\=Q} }\ar[rr]^{C-1}&&*[o][F-]{\statev{\mbox{\;\;}\textoverline{Q-P}\mbox{\;\;}}{P}{\textoverline{N-P}} }\ar[r]^{C-A}&*[o][F-]{\statev{\textoverline{N-Q}}{\mbox{\;\;}Q-P\mbox{\;\;}}{Q} }\ar `dr_l[ll]`_u[llll]_{C-B}[llll]\\
}
\caption{The graph $G_3(\nS)$ of 4-fold intersections of $T$ under the conditions of Theorem~\ref{count-neigh-General}. Here we set $P=(1, 0, 0)^t, Q=(A, 1, 0)^t, N=(B, A, 1)^t$.
  \label{quad-graph}}
\end{figure}

The following lemma is derived by direct calculation.

\begin{lemma}\label{3-4-fold}
Let $T$ be an $ABC$-tile with neighbor graph $G(\nS)$. If $T$ has $14$ neighbors the following assertions hold.
 \begin{itemize}
 \item The graph $G_2(\nS)=\times_{j=1}^2G(\nS)$ has $36$ vertices and is given by\footnote{To save space, in this table and in what follows we will often write $\mybinom{X}{Y}$ instead of $\{X,Y\}$.} Table~\ref{tab:triple-graph}.
 \item The graph $G_3(\nS)=\times_{j=1}^3G(\nS)$ has $24$ vertices and is given by Figure~\ref{quad-graph}.  \item The graph $G_\ell(\nS)=\times_{j=1}^\ell G(\nS)$ is empty for $\ell\ge 4$.
 \end{itemize} 
 By construction, all these graphs are symmetric in the sense that there exists an edge $\Bold{\alpha}\xrightarrow{d}\Bold{\alpha}'$ if and only if $-\Bold{\alpha}\xrightarrow{C-1-d}-\Bold{\alpha}'$.
\end{lemma}

\begin{proof}
By Theorem~\ref{count-neigh-General}, we know that if $G(\mathcal{S})$ has $14$ vertices then it is given by Table~\ref{table-contact} (where the vertex $0$ has to be removed, see also Figure~\ref{Gamma2_1}). So the graphs $G_2(\nS)$ and $G_3(\nS)$ can be constructed from $G(\nS)$ by direct calculation using Definition~\ref{def:powerG2G3G4}. The fact that $G_4(\nS)$ (and, hence, $G_\ell(\nS)$ for $\ell \ge 5$) is empty can be seen easily from $G_3(\nS)$. The symmetry assertion is immediate from the construction of these graphs.
\end{proof}

\section{Topological results}\label{sec:topo}

In this section we establish Theorem~\ref{Main-1}. By Section~\ref{sec:normal} it suffices to prove this theorem for $ABC$-tiles. Since the tiling assertion has already been established in Lemma~\ref{ABCZ3}, it remains to prove assertions \eqref{Main1.sphere} to \eqref{Main1.empty}.  In Section~\ref{sec:easy} we will dispose of the easy cases~\eqref{Main1.point} and~\eqref{Main1.empty} and
Section~\ref{sec:prep3} contains preparations for the proof of~\eqref{Main1.loop}. 
The subsequent subsections contain the crucial parts of the proof. We will prove~\eqref{Main1.loop} in Proposition~\ref{lem:Main1.loop}, \eqref{Main1.sphere} in Proposition~\ref{prop:Main1.surface}, and~\eqref{Main1.disk} in Proposition~\ref{lem:Main1.disk}.

Throughout this section we assume that $T=T(M,\D)$ is an $ABC$-tile with $14$ neighbors.

\subsection{Proof of the easy cases: Theorem~\ref{Main-1}~\eqref{Main1.point} and \eqref{Main1.empty}}\label{sec:easy}
Theorem~\ref{Main-1}~\eqref{Main1.empty} follows immediately from Lemma~\ref{3-4-fold} because the fact that $G_\ell(\nS)$ is empty for $\ell\ge 4$ implies in view of Proposition~\ref{Char-boundary} that there are no points in which five or more tiles of the tiling $T +\Z^3$ intersect.

To prove Theorem~\ref{Main-1}~\eqref{Main1.point} assume that $\Bold{\alpha} \subset \Z^3\setminus\{0\}$ contains three elements. If $\Bold{\alpha}\not\in G_3(\nS)$, then $\B_{\Bold{\alpha}} = \emptyset$ by Proposition \ref{Char-boundary}. If $\Bold{\alpha}\in G_3(\nS)$, then by Lemma~\ref{3-4-fold} (see also Figure \ref{quad-graph}) there exists exactly one infinite walk in $G_3(\nS)$ starting from the vertex $\Bold{\alpha}$. Thus by Proposition \ref{Char-boundary}
the set $\B_{\Bold{\alpha}}$ is a singleton.

Later we will need the following result on $4$-fold intersections.

\begin{lemma}\label{lem:2fourfold}
If $\Bold{\alpha} \in G_2(\nS)$ then the $3$-fold intersection $\B_{\Bold{\alpha}}$ contains exactly two different points that are $4$-fold intersections. If $\Bold{\alpha} $ has more than one outgoing edge in $G_2(\nS)$ then these two points are located in two different subtiles of the first subdivision of $\B_{\Bold{\alpha}}$. 
\end{lemma}

\begin{proof}
First note that for each $\Bold{\alpha} \in G_2(\nS)$ there are exactly two elements  $\Bold{\beta} \in G_3(\nS)$ with $\Bold{\alpha} \subset \Bold{\beta}$. Because $\B_{\Bold{\beta}}$ is a single point for each $\Bold{\beta} \in G_3(\nS)$ by Theorem~\ref{Main-1}~\eqref{Main1.point}, this proves the first assertion.

Let $\Bold{\beta}_1,\Bold{\beta}_2\in G_3(\nS)$ be given with $\Bold{\alpha} \subset \Bold{\beta}_i$ for $i\in\{1,2\}$. Then the edge leading away from $\Bold{\beta}_1$ in $G_3(\nS)$ has a different labeling than the edge leading away from $\Bold{\beta}_2$ in $G_3(\nS)$. Since there are no $5$-fold intersections this means that the points $\B_{\Bold{\beta}_1}$ and $\B_{\Bold{\beta}_2}$ are located in two different subtiles of $\B_{\Bold{\alpha}}$ and the second assertion is proved as well.
\end{proof}

\subsection{Preparatory results on $3$-fold intersections}\label{sec:prep3}
In this subsection we show that each non\-empty $3$-fold intersection as well as each set
\begin{equation}\label{Cgamma}
L_\alpha=\bigcup_{
\begin{subarray}{c}
\beta \in \nS
\\
\{\alpha,\beta\}\in G_2(\nS)
\end{subarray}
} \B_{\alpha,\beta} \qquad(\alpha\in \nS)
\end{equation}
is a Peano continuum. (We note already here that we will prove in Lemma~\ref{lem:bdov2} that $L_\alpha=\partial_{\partial T} \B_{\alpha}$.) Moreover, we provide some combinatorial results on the subdivision structure of $L_\alpha$. All this will be needed in order to prove Theorem~\ref{Main-1}~\eqref{Main1.loop}.

We start with a definition.

\begin{definition}[{{\em cf.\ e.g.}~\cite[Definition~6.6]{SiegelThuswaldner10}}]\label{def:chains}
Let $\mathcal{K}=\{X_1,\ldots, X_\nu\}\subset \R^m$ be a finite collection of sets.
\begin{itemize}
\item The collection $\mathcal{K}$ forms a \emph{regular chain} if $|X_i\cap X_{i+1}| =1$ for each $i\in\{1,\ldots,\nu-1\}$ and $X_i \cap X_j =\emptyset$ if $|i-j| \ge 2$. 
\item The collection $\mathcal{K}$ forms a \emph{circular chain} if $|X_i\cap X_{i+1}| =1$ for each $i\in\{1,\ldots,\nu-1\}$, $|X_1 \cap X_{\nu}|=1$, and $X_i \cap X_j =\emptyset$ if $2\le |i-j|\le \nu-2$.
\item The {\em Hata graph} of $\mathcal{K}$ is an undirected graph. Its vertices are the elements of $\mathcal{K}$ and there is an edge between $X_i$ and $X_j$ if and only if $i\not=j$ and $X_i \cap X_j \not=\emptyset$.
\end{itemize}
\end{definition}

We need the following result on connectedness of the attractor of a graph-directed iterated function system in the sense of Mauldin and Williams~\cite{MauldinWilliams88}.

\begin{lemma}[{{\em cf.}~\cite[Theorem~4.1]{LuoAkiyamaThuswaldner04}}]\label{lem:LAT}
Let $\{S_1, \ldots, S_q\}$ be the attractor of a graph-directed iterated function system with (directed) graph $G$ with set of vertices $\{1,\ldots,q\}$, set of edges $E$, and contractions $F_e$ ($e\in E$) as edge labels, {\em i.e.}, the nonempty compact sets $S_1,\ldots,S_q$ are uniquely defined by
\[
S_i = \bigcup_{i \xrightarrow{e} j}F_e(S_j) \qquad i\in\{1,\ldots,q\},
\]
where the union is taken over all edges in $G$ starting from $i$.
Then $S_i$ is a Peano continuum or~a single point for each 
$i\in\{1,\ldots,q\}$ if and only if for each 
$i\in\{1,\ldots,q\}$ the \emph{successor collection} 
$$
\big\{ F_e(S_j); \; i \xrightarrow{e} j \hbox{ is an edge in $G$ starting from $i$}
\big\}
$$
 of~$i$ has a connected Hata graph.
\end{lemma}

Let $\ell \ge 1$ and assume that each edge label $d\in \D$ of $G_\ell(\nS)$ is interpreted as the contraction $f_d$. Then by the set equation \eqref{L-foldGra3} the graph $G_\ell(\nS)$ defines a graph-directed iterated function system with attractor $\{B_{\Bold{\alpha}};\, \Bold{\alpha} \in G_\ell(\nS)\}$. 
The following lemma gives first topological information on the set of $3$-fold intersections.

\begin{lemma}\label{lem:2lc}\mbox{}
\begin{enumerate}
\item \label{lem2lc1}
For each vertex $\Bold{\alpha}\in G_2(\nS)$, the set $\B_{\Bold{\alpha}}$ is a Peano continuum.
\item \label{lem2lc2}
For each $\alpha\in \nS$, the set $L_\alpha$ is a Peano continuum.
\end{enumerate}
\end{lemma}

\begin{proof}
To prove assertion~\eqref{lem2lc1}, we want to apply Lemma~\ref{lem:LAT} to $\{ \B_{\Bold{\alpha}};\,  \Bold{\alpha}\in G_2(\nS)\}$. Thus we have to show that the Hata graph $H(\Bold{\alpha})$ of the successor collection of each vertex $\Bold{\alpha} \in G_2(\nS)$ is connected. 
(Note that $\B_{\Bold{\alpha}}$ cannot be a singleton because each vertex of $G_2(\nS)$ is the starting point of infinitely many infinite walks.) For convenience, we multiply each element of these successor collections by $M$. This has no effect on $H(\Bold{\alpha})$. 

From Table~\ref{tab:triple-graph} we see that $G_2(\nS)$ has $36$ vertices. If $A\ge 2$, then $18$ of them have only one outgoing edge, if $A=1$ this is the case for $24$ vertices. For these ``trivial'' vertices the graph $H(\Bold{\alpha})$ is a single vertex and, hence, it is connected. Thus we have to deal with the remaining ``nontrivial'' vertices of $G_2(\nS)$ ($18$ for $A\ge 2$ and $12$ for $A=1$).

Let $X_1,X_2$ be two elements of a (multiplied by $M$) successor collection of a ``nontrivial'' vertex $\Bold{\alpha}\in G_2(\nS)$. Then there are $a_1,a_2\in \D$ and $\Bold{\beta}_1,\Bold{\beta}_2\in G_2(\nS)$ such that $X_i= \B_{\Bold{\beta}_i}+a_i$ for $i\in\{1,2\}$. To check if there is an edge in $H(\Bold{\alpha})$ connecting $X_1$ and $X_2$, we note that by the definition of $G_3(\nS)$ and the fact that $G_\ell(\nS) = \emptyset$ for $\ell \ge 4$ we have 
\begin{equation}\label{eq:conneqiv}
\begin{split}
X_1\cap X_2 \not=\emptyset 
&\quad\Longleftrightarrow\quad
\B_{\Bold{\beta}_1} \cap (\B_{\Bold{\beta}_2} + a_2 -a_1) \not= \emptyset  \\
&\quad\Longleftrightarrow\quad  
(\Bold{\beta}_1 \cup (\Bold{\beta}_2 + a_2 -a_1) \cup \{a_2-a_1\}) \setminus \{0\} \in G_3(\nS).
\end{split}
\end{equation}
Thus the graph $H(\Bold{\alpha})$ can be set up by checking the graphs $G_2(\nS)$ and $G_3(\nS)$.

It turns out that the Hata graphs $H(\Bold{\alpha})$ for the nontrivial vertices of $\Bold{\alpha}\in G_2(\nS)$ all have the same structure. Indeed, let (recalling that $P=(1, 0, 0)^t, Q=(A, 1, 0)^t, N=(B, A, 1)^t$)
\begin{equation}\label{eq:3igbo}
\begin{array}{rclrclrcl}
{\Bold{\zeta}_1}&=&\{\overline{Q},{N-Q} \}, &
\Bold{\eta}_1&=&\{\overline{Q-P},{N-Q}\}, &
\Bold{\vartheta}_1&=&\{\overline{Q-P},{N-Q+P}\}, \\
\Bold{\zeta}_2&=&\{{Q-P},{N-P}\},& 
\Bold{\eta}_2&=&\{{Q-P},{N}\},&
\Bold{\vartheta}_2&=&\{{Q},{N}\},\\
\Bold{\zeta}_3&=&\{\overline{N},\overline{N-Q+P}\}, &
\Bold{\eta}_3&=&\{\overline{N},\overline{N-Q}\}, &
\Bold{\vartheta}_3&=&\{\overline{N-P},\overline{N-Q}\}
\end{array}
\end{equation}
be vertices of $G_2(\nS)$ and set
\begin{equation}\label{eq:Vi}
V_i=\left\{
{\Bold{\zeta}_i} + d,{\Bold{\eta}_i} + d,{\Bold{\vartheta}_i} +d;\; d\in \D
\right\} 
\qquad(1\le i\le 3).
\end{equation}
Then using \eqref{eq:conneqiv} and inspecting the graph $G_3(\nS)$ we gain that 
\[
(\B_{\Bold{\zeta}_i} + d) \cap  (\B_{\Bold{\eta_i}} + d),\;
(\B_{\Bold{\eta}_i}+d)\cap (\B_{\Bold{\vartheta_i}}+ d), \;
(\B_{\Bold{\vartheta}_i}+d)\cap \left(\B_{\Bold{\zeta}_i}+(d+P)\right) 
\]
contain a single element for $(1\le i\le 3, \; d\in \D)$ and all the other intersections of the form  $\B_{\Bold{\gamma}} \cap \B_{\Bold{\gamma'}}$ with $\Bold{\gamma},\Bold{\gamma'} \in V_i$ are empty. Thus we conclude that $V_i$ is a regular chain whose Hata graph is the {\it path graph} depicted in Figure~\ref{fig:chain}.
\begin{figure}[h]
\xymatrix{*+[F-]{\st{$B_{\Bold{\vartheta}_i}$} }\ar@{-}[d]\ar@{-}[ddr]& *+[F-]{\st{$B_{\Bold{\vartheta}_i}+P$} }\ar@{-}[d]\ar@{-}[ddr]&*+[F-]{\st{$B_{\Bold{\vartheta}_i}+2P$} }\ar@{-}[d]\ar@{-}[ddrr]&&\dots && *+[F-]{\st{$B_{\Bold{\vartheta}_i}+(C-1)P$} }\ar@{-}[d]\\
*+[F-]{\st{$B_{\Bold{\eta}_i}$} }\ar@{-}[d]& *+[F-]{\st{$B_{\Bold{\eta}_i}+P$} }\ar@{-}[d]&*+[F-]{\st{$B_{\Bold{\eta}_i}+2P$} }\ar@{-}[d]&&\dots && *+[F-]{\st{$B_{\Bold{\eta}_i}+(C-1)P$} }\ar@{-}[d]\\
*+[F-]{\st{$B_{\Bold{\zeta}_i}$} }& *+[F-]{\st{$B_{\Bold{\zeta}_i}+P$} }&*+[F-]{\st{$B_{\Bold{\zeta}_i}+2P$} }&&\dots && *+[F-]{\st{$B_{\Bold{\zeta}_i}+(C-1)P$} }\ar@{-}[uull]\\
}
\caption{The Hata graph of the regular chain $V_i$ ($1\le i\le 3$).\label{fig:chain}}
\end{figure}
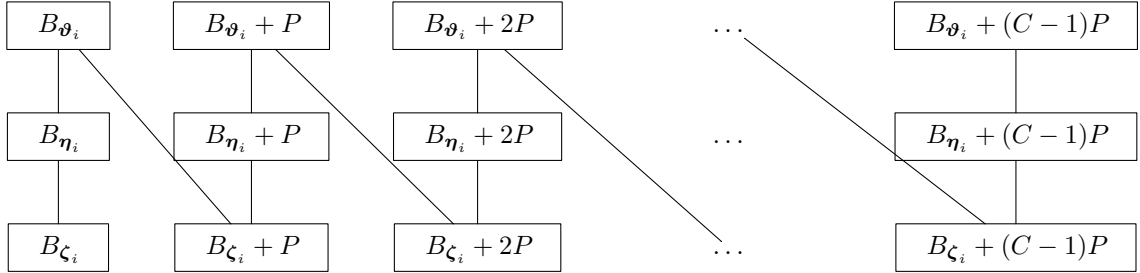

We can read off from the graph $G_2(\nS)$ in Table~\ref{tab:triple-graph} that each nontrivial vertex $\Bold{\alpha}$ has a Hata graph $H(\Bold{\alpha})$ which is a path graph that is a subgraph of the Hata graph of $V_i$ for some $i\in\{1,2,3\}$. Thus $H(\Bold{\alpha})$ is a connected graph and the proof of \eqref{lem2lc1} is finished.

To prove assertion~\eqref{lem2lc2}, it suffices to show that it holds for ${\alpha}\in\{P,~Q,~N,~Q-P,~N-Q,~N-P,~N-Q+P\}$ by the symmetry mentioned in Lemma~\ref{3-4-fold}.  From the definition of $L_{{\alpha}}$ we get
 
\begin{equation}\label{eq:circchainLgamma}
\begin{split}
L_{P}&=  \B_{\mybinom{\overline{Q-P}}{P}}  \cup \B_{\mybinom{P}{N-Q+P}}\cup \B_{\mybinom{P}{N}} \cup \B_{\mybinom{P} {Q}}\cup \B_{\mybinom{\overline{N-Q}}{P} }\cup \B_{\mybinom{\overline{N-P}}{P} },\\
L_{Q} &=  \B_{\mybinom{\overline{N-Q}}{Q} } \cup \B_{\mybinom{Q-P}{Q} } \cup \B_{\mybinom{Q}{N} } \cup \B_{\mybinom{P}{Q} },\\
L_{N} &=  \B_{\mybinom{N-P}{N}}\cup \B_{\mybinom{N-Q}{N}}\cup \B_{\mybinom{N-Q+P}{N} }\cup \B_{\mybinom{P}{N} }\cup \B_{\mybinom{Q}{N} }\cup \B_{\mybinom{Q-P}{N} },\\
L_{Q-P} &= \B_{\mybinom{\overline{N-Q+P}}{Q-P} } \cup \B_{\mybinom{\overline{N-Q}}{Q-P}}\cup \B_{\mybinom{Q-P}{Q} }\cup \B_{\mybinom{Q-P}{N} }\cup \B_{\mybinom{Q-P}{N-P} }\cup \B_{\mybinom{\overline{P}}{Q-P} },\\
L_{{N-Q}} &= \B_{\mybinom{\overline{Q-P}}{N-Q} } \cup \B_{\mybinom{\overline{Q}}{{N-Q}}}\cup \B_{\mybinom{\overline{P}}{{N-Q}} }\cup \B_{\mybinom{N-Q}{{N-P}} }\cup \B_{\mybinom{N-Q}{N} }\cup \B_{\mybinom{N-Q}{N-Q+P} },\\
L_{{N-P}} &= \B_{\mybinom{Q-P}{N-P} } \cup \B_{\mybinom{\overline{P}}{N-P}}\cup \B_{\mybinom{{N-Q}}{N-P} }\cup \B_{\mybinom{N-P}{N} },\\
L_{N-Q+P} &= \B_{\mybinom{P}{N-Q+P} } \cup \B_{\mybinom{N-Q+P}{N} }\cup \B_{\mybinom{N-Q}{N-Q+P}}\cup \B_{\mybinom{\overline{Q-P}}{N-Q+P} }.\\
\end{split}
\end{equation}
Each union on the right hand side is ordered in a way that consecutive sets have nonempty intersection (indeed, by using \eqref{eq:conneqiv} and the graph $G_3(\nS)$ we see that the collection of the elements of each union even forms a circular chain). Thus each of the sets $L_\alpha$ in \eqref{eq:circchainLgamma} is a connected union of finitely many Peano continua and, hence, a Peano continuum.
\end{proof}

Next we prove a combinatorial result. The collection $\mathcal{L}_{\alpha,k}$ defined in the following lemma is the set of pieces of the $(k-1)$-th subdivision of the set $L_\alpha$. Thus this result already hints at the fact that $L_\alpha$ is a simple closed curve. Recall that for $\Bold{\alpha} \in \times_{j=1}^{\ell}G(\nS)$  the collection $\mathcal{C}_k(\boldsymbol{\alpha})$ is defined in \eqref{eq:CkAlpha}.

\begin{lemma}\label{lem:chain}
For each $\alpha\in \nS$ the collection 
\begin{equation}\label{eq:Pkclosed}
\mathcal{L}_{\alpha,k} = \bigcup_{\beta: \, \{\alpha,\beta\}\in G_2(\nS)} \mathcal{C}_k(\{\alpha,\beta\})
\end{equation}
forms a circular chain for each $k\ge 1$ (if the sets in this collection are ordered properly).
\end{lemma}

\begin{proof}
Let $\alpha\in \nS$ be arbitrary but fixed throughout this proof. Let $H(\mathcal{L}_{\alpha,k})$ be the Hata graph of $\mathcal{L}_{\alpha,k}$. Using induction on $k$ we will prove first that $H(\mathcal{L}_{\alpha,k})$ consists of a single cycle. In the proof of Lemma~\ref{lem:2lc}~\eqref{lem2lc2} we showed that this is true for $k=1$. To perform the induction step we assume $H(\mathcal{L}_{\alpha,k})$ consists of a single cycle for some $k\in\N$. To prove that the same holds for $H(\mathcal{L}_{\alpha,k+1})$, we examine each edge of $H(\mathcal{L}_{\alpha,k})$ carefully. Let 
\begin{equation}\label{eq:genericEdge}  
X_1\mbox{---} X_2
\end{equation} 
be an arbitrary edge in $H(\mathcal{L}_{\alpha,k})$. This edge represents two sets $X_i = M^{-k+1}(\B_{\Bold{\beta}_i} + a_i)$ with  $a_1,a_2\in \D$ and $\Bold{\beta}_1,\Bold{\beta}_2\in G_2(\nS)$ that have nonempty intersection. When we pass from $H(\mathcal{L}_{\alpha,k})$ to $H(\mathcal{L}_{\alpha,k+1})$ each vertex $X_i$ is replaced by a path $\bullet \mbox{---} \cdots \mbox{---} \bullet$ (possibly degenerated to a single vertex) whose vertices are the elements of the subdivision of $X_i$. Indeed, from the proof of Lemma~\ref{lem:2lc}~\eqref{lem2lc1} we know that $X_i$ is subdivided according to the graph $G_2(\nS)$ into a finite collection of sets that forms a regular chain.
Thus passing from $H(\mathcal{L}_{\alpha,k})$ to $H(\mathcal{L}_{\alpha,k+1})$ the edge \eqref{eq:genericEdge} is transformed to a subgraph consisting of two disjoint finite path graphs that are connected by at least one edge.

{\it We claim that this subgraph is itself a path graph.} If we multiply each vertex of $H(\mathcal{L}_{\alpha,k})$ by $M^{k-1}$ and shift it by an appropriate vector in $\Z^3$ then the structure of the Hata graph as well as the way a given vertex subdivides into its subtiles is not changed. Thus we may assume w.l.o.g.\ that the edge \eqref{eq:genericEdge} has the form
\begin{equation}\label{eq:genericEdge2}
\B_{\alpha_1,\alpha_2} \mbox{---} \B_{\alpha_1,\alpha_3}
\end{equation} 
with $\{\alpha_1,\alpha_2,\alpha_3\} \in G_3(\nS)$. To prove the claim, for each $\Bold{\alpha} \in G_3(\nS)$ and each distinct $\Bold{\beta}_1,\Bold{\beta}_2 \subset \Bold{\alpha}$, we have to show that the subdivision of $\B_{\Bold{\beta}_1} \cup \B_{\Bold{\beta}_2}$ has a Hata graph which is a path graph. We will denote this Hata graph by $H(\Bold{\beta}_1,\Bold{\beta}_2)$. (Note that $\Bold{\beta}_1,\Bold{\beta}_2$ are always vertices of $G_2(\nS)$.)

Inspecting the graphs $G_2(\nS)$ and $G_3(\nS)$ we see that we have the following three cases to distinguish.
\begin{enumerate}
\item[(i)] Both, $\Bold{\beta}_1$ and $\Bold{\beta}_2$ have only one outgoing edge in $G_2(\nS)$.
\item[(ii)] Exactly one of the two vertices, $\Bold{\beta}_1$ and $\Bold{\beta}_2$ have only one outgoing edge in $G_2(\nS)$.
\item[(iii)] Both, $\Bold{\beta}_1$ and $\Bold{\beta}_2$ have more than one outgoing edge in $G_2(\nS)$.
\end{enumerate}

We show that $H(\Bold{\beta}_1,\Bold{\beta}_2)$ is a path graph for each of these cases separately.

Case (i) is trivial because the subdivision of both,  $\B_{\Bold{\beta}_1}$ and $\B_{\Bold{\beta}_2}$, consists  of only one element. Thus $H(\Bold{\beta}_1,\Bold{\beta}_2)$ is of the form $\bullet \mbox{---} \bullet$ and we are done.

Case (ii): 
Assume w.l.o.g.\ that $\Bold{\beta}_1$ has more than one outgoing edge in $G_2(\nS)$ (we call it the nontrivial vertex). Then $\Bold{\beta}_2$ has only one outgoing edge in $G_2(\nS)$ (we call it the trivial vertex). We know from the proof of Lemma~\ref{lem:2lc}~\eqref{lem2lc1} that the subdivision of $\B_{\Bold{\beta}_1}$ has a Hata graph $H(\Bold{\beta}_1)$ which is a path graph $Y_1 \mbox{ --- } \cdots \mbox { --- } Y_r$ for some $r\ge 2$. The Hata graph $H(\Bold{\beta}_2)$ is a single vertex $Z$ by assumption. We have to show  that the Hata graph $H(\Bold{\beta}_1,\Bold{\beta}_2)$ of the subdivision of $\B_{\Bold{\beta}_1}\cup\B_{\Bold{\beta}_2}$ is a path graph. 
We know that $H(\Bold{\beta}_1,\Bold{\beta}_2)$ consists of the path $H(\Bold{\beta}_1)$ and the vertex $H(\Bold{\beta}_2)$ together with some edges connecting these two subgraphs. Thus we have to prove that the only connection between these two subgraphs is a single edge of the form $Y_{i_0} \mbox{ --- } Z$ for $i_0\in \{1,r\}$.
To do this, we have to show that $Y_j \cap Z = \emptyset$ for $j\not=i_0$ and $Y_{i_0} \cap Z \not= \emptyset$. Since all occurring vertices are triple intersections these intersections are nonempty if and only if they correspond to vertices of $G_3(\nS)$.

We illustrate this for an example. Assume that $A\ge 2$ and let $\Bold{\beta}_1=\{Q-P,~N-P\}$ and $\Bold{\beta}_2=\{Q-P,~N\}$. Then $H({\Bold{\beta}_1})$ (multiplied by $M$) is the subpath of the graph in Figure~\ref{fig:chain} for the choice $i=1$ given by 
$
Y_1=\{\overline{Q},~N-Q\}  \mbox{ --- } \cdots  \mbox{ --- } \{\overline{Q},~N-Q + (A-1)P\}=Y_r.
$
The graph $H({\Bold{\beta}_2})$ is the vertex  $\{\overline{P},~N-Q\}$. Since 
$
B_{\{\overline{Q},~N-Q\}} \cap B_{\{\overline{P},~N-Q\}} = B_{\{\overline{P},\overline{Q},~N-Q\}} 
$
with $\{\overline{P},\overline{Q},~N-Q\} \in G_3(\nS)$, we see that $Y_1 \cap Z \not= \emptyset$ in this case. All the other intersections are easily seen to be not in $G_3(\nS)$; most of them would even correspond to $5$-fold intersections which do not exist.

The calculation we have done corresponds to the first line of Table~\ref{tab:iner}. Each line in this table corresponds to a possible constellation. In the fifth column of this table we indicate if the single vertex $H(\Bold{\beta}_2)$ has nonempty intersection with the first\footnote{Since the path graph $H(\Bold{\beta}_1)$ is undirected, we are free which end of the path we regard as ``first'' and ``last'' vertex. The choice which one is the first and which one is the last is indicated in the second column of Table~\ref{tab:iner}.} vertex $Y_1$ or the last vertex $Y_r$ of $H(\Bold{\beta}_1)$. 
All this can easily be verified by checking the intersections occurring in $G_3(\nS)$ as we did above.

\begin{table}[h]
{\scr 
\begin{tabular}{|c|c|c|c|c|c|}\hline 
\multicolumn{1}{|c|}{{$\begin{matrix}
\text{Nontrivial }\\
\text{vertex}\\
\end{matrix}$}} & \multicolumn{1}{|c|}{{ $\begin{matrix}
\text{First and Last }\\
\text{vertex of its}\\
\text{ subdivision}\end{matrix}$ }} & {$\begin{matrix}
\text{Trivial}\\
\text{vertex}\\
\end{matrix}$} & Its subdivision & $\begin{matrix}
\text{First/ }\\
\text{Last}\\
\end{matrix}$ & \multicolumn{1}{|c|}{{Condition}}\\ \hline \vspace{-2mm}&&&&&\\
{\multirow{3}{*}{$\mybinom{Q-P}{N-P}$}} &{\multirow{3}{*}{$\begin{matrix}
{\mybinom{\overline{Q}}{N-Q}},\\
{\mybinom{\overline{Q}}{N-Q}}+x_1\\
\end{matrix}$}} &$\mybinom{Q-P}{N}$&$\mybinom{\overline{P}}{{N-Q}}$&first& {\multirow{3}{*}{$A\geq 2$}}\\\cline{3-3}\cline{4-3}\cline{5-3}
\vspace{-2mm}&&&&&\\
&& $\mybinom{\overline{P}}{N-P}$&$\mybinom{\overline{Q}}{\overline{Q-P}}+x_1$&last& \\\cline{3-3}\cline{4-3}\cline{5-3}
\vspace{-2mm}&&&&&\\ 
&& $\mybinom{N-P}{N}$&$\mybinom{\overline{Q}}{\overline{P}}$&first&\\
\hline \vspace{-2mm}&&&&&\\
{\multirow{2}{*}{$\mybinom{\overline{P}}{Q-P}$}} &{\multirow{2}{*}{$\begin{matrix}
{\mybinom{\overline{Q-P}}{N-Q}}+x_1,\\
{\mybinom{\overline{Q-P}}{N-Q}}+x_2\\
\end{matrix}$}} &$\mybinom{\overline{N-Q+P}}{Q-P}$&$\mybinom{{N-Q}}{{N-Q+P}}+x_2$&last& {\multirow{2}{*}{-}}\\\cline{3-2}\cline{4-2}\cline{5-2}\vspace{-2mm}&&&&&\\ 
&& $\mybinom{\overline{P}}{{N-P}}$&$\mybinom{\overline{Q}}{\overline{Q-P}}+x_1$&first&\\
\hline
\vspace{-2mm}&&&&&\\
{\multirow{3}{*}{$\mybinom{\overline{N-Q+P}}{\overline{P}}$}} &{\multirow{3}{*}{$\begin{matrix}
{\mybinom{\overline{Q-P}}{N-Q+P}}+x_2,\\
{\mybinom{\overline{Q-P}}{N-Q+P}}+x_7\\
\end{matrix}$}} &${\mybinom{\overline{P}}{\overline{N}}}$&$\mybinom{\overline{Q-P}}{P}+x_7$&last& {\multirow{3}{*}{-}}\\\cline{3-3}\cline{4-3}\cline{5-3}
\vspace{-2mm}&&&&&\\
&& $\mybinom{\overline{N}}{\overline{N-Q+P}}$&$\mybinom{P}{N-Q+P}+x_7$&last& \\\cline{3-3}\cline{4-3}\cline{5-3}
\vspace{-2mm}&&&&&\\ 
&& $\mybinom{\overline{N-Q+P}}{Q-P}$&$\mybinom{{N-Q}}{{N-Q+P}}+x_2$&first&\\
\hline
\vspace{-2mm}&&&&&\\
{\multirow{3}{*}{$\mybinom{{P}}{{Q}}$}} &{\multirow{3}{*}{$\begin{matrix}
{\mybinom{{Q-P}}{N-P}},\\
{\mybinom{{Q-P}}{N-P}}+x_3\\
\end{matrix}$}} &${\mybinom{Q}{N}}$&$\mybinom{\overline{P}}{N-P}$&first& {\multirow{3}{*}{-}}\\\cline{3-3}\cline{4-3}\cline{5-3}
\vspace{-2mm}&&&&&\\
&& $\mybinom{Q}{\overline{N-Q}}$&$\mybinom{N-P}{N}+x_3$&last& \\\cline{3-3}\cline{4-3}\cline{5-3}
\vspace{-2mm}&&&&&\\ 
&& $\mybinom{P}{N}$&$\mybinom{\overline{P}}{{Q-P}}$&first&\\
\hline
\vspace{-2mm}&&&&&\\
{\multirow{2}{*}{$\mybinom{\overline{N-Q}}{{P}}$}} &{\multirow{2}{*}{$\begin{matrix}
{\mybinom{{Q-P}}{N}}+x_3,\\
{\mybinom{{Q-P}}{N}}+x_4\\
\end{matrix}$}} &${\mybinom{\overline{N-Q}}{Q}}$&$\mybinom{{N}}{N-P}+x_3$&first& {\multirow{2}{*}{-}}\\\cline{3-3}\cline{4-3}\cline{5-3}
\vspace{-2mm}&&&&&\\
&& $\mybinom{P}{\overline{N-P}}$&$\mybinom{Q-P}{Q}+x_4$&last&\\
\hline
\vspace{-2mm}&&&&&\\
{\multirow{3}{*}{$\mybinom{\overline{N-Q}}{\overline{N-P}}$}} &{\multirow{3}{*}{$\begin{matrix}
{\mybinom{{Q}}{N}+x_4},\\
{\mybinom{{Q}}{N}}+x_7\\
\end{matrix}$}} &$\mybinom{\overline{N}}{\overline{N-P}}$&$\mybinom{{P}}{Q}+x_7$&last& {\multirow{3}{*}{$A\geq 2$}}\\\cline{3-3}\cline{4-3}\cline{5-3}
\vspace{-2mm}&&&&&\\
&& $\mybinom{\overline{N}}{\overline{N-Q}}$&$\mybinom{P}{N}+x_7$&last& \\\cline{3-3}\cline{4-3}\cline{5-3}
\vspace{-2mm}&&&&&\\ 
&& $\mybinom{\overline{N-P}}{P}$&$\mybinom{{Q-P}}{{Q}}+x_4$&first&\\
\hline
\vspace{-2mm}&&&&&\\
{\multirow{3}{*}{$\mybinom{{N-Q}}{{N-Q+P}}$}} &{\multirow{3}{*}{$\begin{matrix}
{\mybinom{\overline{N}}{\overline{N-Q+P}}},\\
{\mybinom{\overline{N}}{\overline{N-Q+P}}}+x_5\\
\end{matrix}$}} &$\mybinom{{N-Q+P}}{\overline{Q-P}}$&$\mybinom{\overline{N-Q+P}}{\overline{N-Q}}+x_5$&last& {\multirow{3}{*}{-}}\\\cline{3-3}\cline{4-3}\cline{5-3}
\vspace{-2mm}&&&&&\\
&& $\mybinom{{N-Q}}{{N}}$&$\mybinom{\overline{P}}{\overline{N}}$&first& \\\cline{3-3}\cline{4-3}\cline{5-3}
\vspace{-2mm}&&&&&\\ 
&& $\mybinom{{N-Q+P}}{N}$&$\mybinom{\overline{N-Q+P}}{\overline{P}}$&first&\\
\hline
\vspace{-2mm}&&&&&\\
{\multirow{2}{*}{$\mybinom{{N-Q}}{\overline{Q-P}}$}} &{\multirow{2}{*}{$\begin{matrix}
{\mybinom{\overline{N}}{\overline{N-Q}}}+x_5,\\
{\mybinom{\overline{N}}{\overline{N-Q}}}+x_6\\
\end{matrix}$}} &${\mybinom{\overline{Q}}{N-Q}}$&$\mybinom{\overline{N}}{\overline{N-P}}+x_6$&last& {\multirow{2}{*}{$A\geq 2$}}\\\cline{3-3}\cline{4-3}\cline{5-3}
\vspace{-2mm}&&&&&\\
&& $\mybinom{N-Q+P}{\overline{Q-P}}$&$\mybinom{\overline{N-Q+P}}{\overline{N-Q}}+x_5$&first&\\
\hline
\vspace{-2mm}&&&&&\\
{\multirow{3}{*}{$\mybinom{\overline{Q}}{\overline{Q-P}}$}} &{\multirow{3}{*}{$\begin{matrix}
{\mybinom{\overline{N-P}}{\overline{N-Q}}}+x_6,\\
{\mybinom{\overline{N-P}}{\overline{N-Q}}}+x_7\\
\end{matrix}$}} &$\mybinom{\overline{Q}}{{N-Q}}+x_6$&$\mybinom{\overline{N}}{\overline{N-P}}+x_6$&first& {\multirow{3}{*}{-}}\\\cline{3-3}\cline{4-3}\cline{5-3}
\vspace{-2mm}&&&&&\\
&& $\mybinom{\overline{N}}{\overline{Q}}$&$\mybinom{\overline{N-P}}{{P}}+x_7$&last& \\\cline{3-3}\cline{4-3}\cline{5-3}
\vspace{-2mm}&&&&&\\ 
&& $\mybinom{\overline{Q-P}}{\overline{N}}$&$\mybinom{\overline{N-Q}}{{P}}+x_7$&last&\\
\hline
\end{tabular}
\medskip
\caption{The constellations of Case (ii) in the proof of Lemma~\ref{lem:chain} where we deal with the subdivision of a pair $\{\Bold{\beta}_1,\Bold{\beta}_2\}$ of vertices exactly one of which, say $\Bold{\beta}_1$, is nontrivial. This table contains all possible constellations of this type modulo symmetry (recall that  $\Bold{\alpha}\xrightarrow{d}\Bold{\alpha}' \in G_2(\mathcal{S})$ if and only if $-\Bold{\alpha}\xrightarrow{C-1-d}-\Bold{\alpha}'\in G_2(\mathcal{S})$). The first column contains the possibilities for $\Bold{\beta}_1$ that can occur in such a constellation. The second column contains the first and the last element of the subdivision of $\Bold{\beta}_1$. The third column contains $\Bold{\beta}_2$, whose (trivial) subdivision is contained in the fourth column. 
The fifth column describes if the first or the last element of the subdivision of $\Bold{\beta}_1$ intersects $\Bold{\beta}_2$. Finally, the sixth column gives the condition under which a given constellation exists. 
For abbreviation we set $x_1=(A-1)P, x_2=(C-B+A-1)P, x_3=(C-B)P, x_4=(C-A)P, x_5=(B-A)P, x_6=(B-1)P, x_7=(C-1)P$. 
\label{tab:iner} }
}
\end{table}

Case (iii): 
By inspecting the graph $G_2(\nS)$ we see that in this case both vertices, $\Bold{\beta}_1$ and $\Bold{\beta}_2$, have the same three vertices as successors. These three vertices are of the form given in \eqref{eq:3igbo} for some $i\in \{1,2,3\}$. Moreover, by inspecting the labels of the edges going out of $\Bold{\beta}_1$ and $\Bold{\beta}_2$ we see that the collection of successors of $\B_{\Bold{\beta}_1}\cup\B_{\Bold{\beta}_2}$ is a (consecutive) subcollection of $M^{-1} V_i$ with $V_i$ as in \eqref{eq:Vi}.  Hence, the Hata graph $H(\Bold{\beta}_1,\Bold{\beta}_2)$ is a path graph which is a subgraph of the graph depicted in Figure~\ref{fig:chain}. 

Summing up this finishes the proof of the claim. 

We now show that $H(\mathcal{L}_{\alpha,k+1})$ is a cycle. To this end, let $\mathcal{L}_{\alpha,k}=\{Y_1,\ldots, Y_p\}$ be the set of vertices of $H(\mathcal{L}_{\alpha,k})$ and $Y_i \mbox{ --- } Y_{i+1}$ for $1\le i \le p$ (whe always assume that $Y_0:=Y_p$ and $Y_{p+1}:=Y_1$; note that $p\ge 4$ by \eqref{eq:circchainLgamma}) its set of edges. Each vertex $Y_i$ of $H(\mathcal{L}_{\alpha,k})$ becomes a path $l_i$ in $H(\mathcal{L}_{\alpha,k+1})$. If $l_i$ is a single vertex, then the above claim (see Case (i) and Case (ii)) implies that this vertex is connected with a terminating vertex of $l_{i-1}$ and with a terminating vertex of $l_{i+1}$. If $l_i$ is a (nondegenerate) path $Z_1  \mbox{ --- }  \cdots  \mbox{ --- }  Z_2$, then a terminating vertex of $l_{i-1}$ is connected to $Z_r$ for some $r\in \{1,2\}$ and a terminating vertex of $l_{i+1}$ is connected to $Z_s$ for some $s\in \{1,2\}$ (see Case (ii) and Case (iii)). We have to show that $r\not=s$. Indeed, suppose on the contrary that both paths are connected to the same vertex, say $Z_1$. Then the element $Z_1$ of the subdivision of $Y_i$ contains 
two disjoint\footnote{Note that $Y_{i-1}$ and $Y_{i+1}$ are disjoint for any $i\in\{1,\ldots,p\}$ because $p\ge 4$.} $4$-fold intersections (one with an element of $Y_{i-1}$ and one with an element of $Y_{i+1}$) which would contradict Lemma~\ref{lem:2fourfold}. Thus the paths $l_i$ ($1\le i\le p$) are arranged in a circular order and, hence, $H(\mathcal{L}_{\alpha,k+1})$ is a cycle. 

Since the edges in $H(\mathcal{L}_{\alpha,k+1})$ correspond to nonempty $4$-fold intersections they represent single points by Theorem~\ref{Main-1}~\eqref{Main1.point}. This implies that $\mathcal{L}_{\alpha,k+1}$ is a circular chain and the induction proof is finished.
\end{proof}

\subsection{Topological characterization of $3$-fold intersections}

This subsection is devoted to the proof of Theorem~\ref{Main-1}~\eqref{Main1.loop}. In all what follows a theory developed by Bing~\cite{Bing51} in order to characterize $1$- and $2$-spheres will be of importance. To apply this theory we introduce some terminology. Let $\mathcal{X}$ be a Peano continuum. A {\em partitioning} of $\mathcal{X}$ is a collection of mutually disjoint open sets whose union is dense in $\mathcal{X}$. A partitioning is called \emph{regular} if each of its elements is the interior its closure. A sequence $\nP_1,\nP_2,\dots$ of partitionings is called a {\em decreasing sequence of partitionings} if $\nP_{k+1}$ is a refinement of $\nP_k$ and the maximum of the diameters of the elements of $\nP_k$ tends to $0$ as $k$ tends to infinity. The basis of our proof of Theorem~\ref{Main-1}~(\ref{Main1.loop}) is given by the following characterization of a simple closed curve due to Bing~\cite{Bing51}.

\begin{proposition}[{{\em cf.}~\cite[Theorem~8]{Bing51}}]\label{Bing-curve}
Let $\mathcal{X}$ be a Peano continuum. A necessary and sufficient condition that $\mathcal{X}$ be a simple closed curve is that one of its decreasing sequences of regular partitionings $\nP_1, \nP_2,\dots$ have the following properties (for $k\ge 1$):
\begin{enumerate}
\item \label{Bing-curve1} The boundary of each element of $\nP_k$ is a pair of distinct points.
\item \label{Bing-curve2} No three elements of $\nP_k$ have a boundary point in common.
\end{enumerate}
\end{proposition} 
 
We emphasize that throughout the remaining part of the proof of Theorem~\ref{Main-1} the ambient space will change and we always have to keep in mind with respect to which ambient space we will take boundaries or interiors. For this reason we will always make clear in which space we are working. As mentioned before, the boundary w.r.t.\ a given space $X$ will be denoted by $\partial_X$ (for the closure and the interior we do not use any notation to emphasize on the space; this space should always be clear from the context or will be mentioned explicitly).
 
Our first aim is to prove that the set $L_\alpha$ defined in \eqref{Cgamma} is a simple closed curve for each $\alpha\in\nS$ by using this proposition. To this end fix $\alpha \in \nS$. In order to construct the partitionings for $L_\alpha$, for each $\Bold{\beta}^{(0)} \in G_2(\nS)$ with $\alpha \in \Bold{\beta}^{(0)} $ set
\begin{equation}
\begin{split}
\nP_{\alpha, k}(\Bold{\beta}^{(0)})&=\{f_{d_1d_2\dots d_{k-1}}(\B_{\Bold{\beta}^{(k-1)}})^{\circ};\; \Bold{\beta}^{(0)}\xrightarrow{d_1}\Bold{\beta}^{(1)}\xrightarrow{d_2}\cdots\xrightarrow{d_{k-1}}\Bold{\beta}^{(k-1)}\in G_2(\nS)\} \; (k\ge 1).
\end{split}
\end{equation}
Here the interior $K^{\circ}$ of a set $K$ is taken w.r.t.\ the subspace topology on $L_\alpha$; this is why $\nP_{\alpha, k}(\Bold{\beta}^{(0)})$ depends on $\alpha$. 
Now the sequence $(\nP_{\alpha,k})_{k\geq 1}$ is defined by
\begin{equation}\label{3-partioning}
\nP_{\alpha,k}= \bigcup_{\begin{subarray}{c}\{\beta_1,\beta_2\}\in G_2(\nS)\\ \beta_1 = \alpha \end{subarray}} \nP_{\alpha, k}(\{\beta_1,\beta_2\})  \qquad (k\geq 1).
\end{equation}

We want to prove that  $(\nP_{\alpha, k})_{k\geq 1}$ is a decreasing sequence of regular partitionings of $L_\alpha$ for each $\alpha\in\nS$. To this end we need a result on the boundary and the interior of a $3$-fold intersection.  Recall the notation $\mathcal{L}_{\alpha,k}$  introduced in \eqref{eq:Pkclosed}.

\begin{lemma}\label{inner-bound2}
Let $\alpha\in\nS$ be given. For each vertex $\Bold{\beta}\in G_2(\nS)$ with $\alpha\in\Bold{\beta}$ we have $\overline{\B_{\Bold{\beta}}^{\circ}}=\B_{\Bold{\beta}}$, w.r.t.\ the subspace topology on $L_\alpha$. More generally, we have $\overline{X^\circ}=X$ for each $X\in \mathcal{L}_{\alpha,k}$ and each $k\ge1$.
\end{lemma}

\begin{proof}
Fix $\alpha\in\nS$. The ambient space in this proof is $L_\alpha$. First observe that \eqref{Cgamma} implies that the collection $\{\B_{\Bold{\gamma}};\,\Bold{\gamma} \in G_2(\nS) \hbox{ with } \alpha \in \Bold{\gamma}\}$ is a finite collection of compact sets which covers $L_\alpha$. Thus for each $\Bold{\beta}=\{\alpha,\alpha'\}\in G_2(\nS)$ we have
\[
\partial_{L_\alpha} \B_{\alpha,\alpha'} \subset \bigcup_{
\begin{subarray}{c}
\alpha''\not\in\{\alpha,\alpha'\}\\ 
\{\alpha,\alpha''\}\in G_2(\nS)
\end{subarray}
}\B_{\alpha,\alpha''}
\]
which implies that (since $\B_{\alpha,\alpha'}$ is closed in $\R^3$ and, hence, also closed in $L_\alpha$)
\begin{equation}\label{eq:incl23bf}
\partial_{L_\alpha} \B_{\alpha,\alpha'}\subset \bigcup_{\begin{subarray}{c}{\alpha'' \in\nS}\\{\{\alpha,\alpha',\alpha''\}\in G_3(\nS)}\end{subarray}}\B_{\alpha,\alpha',\alpha''}.
\end{equation}
Thus, since the sets $\B_{\alpha,\alpha',\alpha''}$ contain at most one point by Theorem~\ref{Main-1}~\eqref{Main1.point}, $\partial_{L_\alpha} \B_{\Bold{\beta}}$ is a finite set. Now choose $\varepsilon > 0$ and $x\in\B_{\Bold{\beta}}$ arbitrary. Subdivide $\B_{\Bold{\beta}}$ according to the set equation \eqref{L-foldGra3} for $r\in\N$ large enough to obtain a subtile $Z=M^{-r+1}(\B_{\Bold{\gamma}}+c) \in \mathcal{C}_r(\Bold{\beta})$ (with $\Bold{\gamma}\in G_2(\nS)$  and $c\in\Z^3$) of $\B_{\Bold{\beta}}$ with diameter less than $\varepsilon$ with $x\in Z$. Since $Z$ is a Peano continuum by Lemma~\ref{lem:2lc} it contains infinitely many points and, hence, there is a point $y\in Z$ with $y\in \B_{\Bold{\gamma}}^{\circ}$. Since $\varepsilon$ was arbitrary, $y$ can be chosen arbitrarily close to $x$. This proves the result for $\B_{\Bold{\beta}}$. 

The assertion for the elements of the subdivisions $\mathcal{L}_{\alpha,k}$, $k\ge1$, is proved in an analogous way. Indeed, the finite collection $\mathcal{L}_{\alpha,k}$ covers $L_\alpha$ which entails that for each $X\in \mathcal{L}_{\alpha,k} $ we have
\begin{equation}\label{boud-subdivide3fineOneSide}
\partial_{L_\alpha} X \subset \bigcup_{Y \in\mathcal{L}_{\alpha,k} \setminus\{X\}} (X\cap Y).
\end{equation}
By Lemma~\ref{lem:chain} the sets $X\cap Y$ contain at most one element, hence, $\partial_{L_\alpha}X$  is a finite set. Since $X=M^{-k+1}(\B_{\Bold{\beta}}+c)$ for some $\Bold{\beta}\in G_2(\nS)$ and some $c\in\Z^3$ we may now subdivide $\B_{\Bold{\beta}}$ according to the set equation \eqref{L-foldGra3} and argue as in the paragraph before.
\end{proof}

We are now in a position to prove that $(\nP_{\alpha,k})_{k\geq 1}$ has the desired properties.

\begin{lemma}\label{lem-partioning2}
The sequence $(\nP_{\alpha,k})_{k\geq 1}$ in \eqref{3-partioning} is a decreasing sequence of regular partitionings of $L_\alpha$ for each $\alpha\in\nS$.
\end{lemma}

\begin{proof}
The ambient space in this proof is $L_\alpha$. We first claim that $\nP_{\alpha,k}$ is a partitioning of $L_\alpha$ for each $k\geq 1$. To prove this we have to show that two distinct elements of $\nP_{\alpha,k}$ are disjoint and 
\begin{equation}\label{LaPka}
L_\alpha=\bigcup_{X\in \nP_{\alpha,k}} \overline{X}.
\end{equation}
For given distinct $X_1,X_2 \in \nP_{\alpha,k}$ we have $\overline{X_1},\overline{X_2} \in \mathcal{C}_k^{(2)}$ by Lemma~\ref{inner-bound2}. Lemma~\ref{lem:piecesresults} thus implies that $\overline{X_1}\cap \overline{X_2}$ is either empty or an affine copy of $\B_{\Bold{\beta}}$ for some $\Bold{\beta} \in G_3(\nS)$ and, hence, by Theorem~\ref{Main-1}~\eqref{Main1.point} the intersection $\overline{X_1}\cap \overline{X_2}$ contains at most one point $p$. Since $\overline{X_1}$ and $\overline{X_1}$ are Peano continua by Lemma~\ref{lem:2lc} they do not contain isolated points which implies that the point $p$ cannot be contained in the open set $X_1 \cap X_2$. Thus we conclude that $X_1 \cap X_2=\emptyset$. Since \eqref{LaPka} follows from the definition of $\nP_{\alpha,k}$ together with Lemma~\ref{inner-bound2} and the set equation \eqref{L-foldGra3} we proved the claim. The fact that $\nP_{\alpha,k}$ is regular for each $k\ge 1$ follows from Lemma~\ref{inner-bound2}.

Now we will show that $(\nP_{\alpha,k})_{k\geq 1}$ is a decreasing sequence of partitionings. First we prove  that $\nP_{\alpha,k+1}$ is a refinement of $\nP_{\alpha,k}$ for each $k\ge 1$. Indeed, by the set equation \eqref{L-foldGra2} the closure $f_{d_1\dots d_{k-1}d_{k}}(\B_{\Bold{\beta}^{(k)}})$ of each element of 
$\nP_{\alpha,k+1}$ is contained in the closure $f_{d_1\dots d_{k-1}}(\B_{\Bold{\beta}^{(k-1)}})$ of some element of $\nP_{\alpha,k}$. Taking interiors we get the refinement assertion. The maximum of the diameters of the elements of $\nP_{\alpha,k}$ approaches zero because $f_d$ is a contraction for each $d\in \D$. 
\end{proof}

We now show that $(\nP_{\alpha,k})_{k\geq 1}$ satisfy \eqref{Bing-curve1} and \eqref{Bing-curve2} of Proposition~\ref{Bing-curve}. We start with the following lemma.

\begin{lemma}\label{lem:bdov}
For $\alpha,\beta\in \nS$ with $\{\alpha,\beta\}\in G_2(\nS)$ we have
\begin{equation}\label{boud-subdivide3}
\partial_{L_\alpha} \B_{\alpha,\beta}= 
\bigcup_{
\begin{subarray}{c}
\gamma\in\nS\\
\{\alpha,\beta,\gamma\}\in G_3(\nS)
\end{subarray}
}
\B_{\alpha,\beta,\gamma}.
\end{equation}
More generally, for each $k\ge 1$ and each $X\in \mathcal{L}_{\alpha,k}$ we have 
\begin{equation}\label{boud-subdivide3fine}
\partial_{L_\alpha} X= \bigcup_{Y \in \mathcal{L}_{\alpha,k} \setminus\{X\}} (X\cap Y).
\end{equation}
\end{lemma}

\begin{proof}
The ambient space in this proof is $L_\alpha$. Let $\mathcal{B}(x,\delta)=\{y\in L_\alpha;\, |x-y|<\delta\}$. The fact that the left hand side of \eqref{boud-subdivide3} is contained in the right hand side follows from \eqref{eq:incl23bf}. To prove the reverse inclusion, suppose that for some $\gamma \in \nS$ with $\{\alpha,\beta,\gamma\}\in G_3(\nS)$ there exists $x\in \B_{\alpha,\beta,\gamma}\setminus{\partial_{L_\alpha}{\B_{\alpha,\beta}}}$. 
Since $\B_{\alpha,\beta,\gamma} \subset \B_{\alpha,\beta}$, this implies that $x\in \B_{\alpha,\beta}^{\circ}$ and, hence, there exists $\delta >0$ with $\mathcal{B}(x,\delta)\subset \B_{\alpha,\beta}$. Since $\B_{\alpha,\beta,\gamma} \subset \B_{\alpha,\gamma}$, we also have $x\in \B_{\alpha,\gamma}$ and by Lemma~\ref{inner-bound2} there exists $y\in \B_{\alpha,\gamma}^{\circ}\cap \mathcal{B}(x,\delta)$. Thus there exists $\delta'>0$ such that $\mathcal{B}(y,\delta') \subset \B_{\alpha,\gamma}\cap\mathcal{B}(x,\delta)$.
This implies $\mathcal{B}(y,\delta')\subset \B_{\alpha,\beta}\cap \B_{\alpha,\gamma}=\B_{\alpha,\beta,\gamma}.$
By Theorem~\ref{Main-1}~\eqref{Main1.point}, $\B_{\alpha,\beta,\gamma}$ is single point for each vertex $\{\alpha,\beta,\gamma\}$ in $G_3(\nS)$. However, a single point cannot contain a ball in $L_\alpha$ because this set is a Peano continuum  by Lemma~\ref{lem:2lc}. This contradiction finishes the proof for $\partial_{L_\alpha} \B_{\alpha,\beta}$. 

The second assertion is proved along the same lines. Indeed, by \eqref{boud-subdivide3fineOneSide}  the left hand side of \eqref{boud-subdivide3fine} is contained in the right hand side. For the reverse inclusion assume that for some $Y \in \mathcal{L}_{\alpha,k} \setminus\{X\}$ there exists $x\in (X\cap Y) \setminus \partial_{L_\alpha}X$. 
Thus $x\in X^{\circ}$ and, hence, there exists $\delta >0$ with $\mathcal{B}(x,\delta)\subset X$. Since we also have $x\in Y$, by Lemma~\ref{inner-bound2} there exists $y\in Y^{\circ}\cap \mathcal{B}(x,\delta)$. Thus there exists $\delta'>0$ such that $\mathcal{B}(y,\delta') \subset Y\cap\mathcal{B}(x,\delta)\subset X\cap Y$. By Lemma~\ref{lem:chain}, this set is a singleton which cannot contain a ball in the Peano continuum $L_\alpha$, a contradiction. 
\end{proof}

Assertion (2) of the next proposition implies Theorem~\ref{Main-1}~(\ref{Main1.loop}).

\begin{proposition}\label{lem:Main1.loop}\mbox{}
\begin{itemize}
\item[(1)] If $\alpha \in\nS$ then $L_\alpha$ is a simple closed curve.
\item[(2)] If $\Bold{\beta} \subset \Z^3\setminus\{0\}$ contains $2$ elements then the $3$-fold intersection $\B_{\Bold{\beta}}$ is homeomorphic to an arc if $\Bold{\beta}\in G_2(\nS)$ and $\B_{\Bold{\beta}}=\emptyset$ otherwise.
\end{itemize}
\end{proposition}

\begin{proof}
To show (1) we work in the ambient space $L_\alpha$. Our goal is to apply Proposition~\ref{Bing-curve}. Let $(\nP_{\alpha,k})_{k\ge 1}$ be the sequence given in \eqref{3-partioning}. This sequence is a decreasing sequence of regular partitionings of $L_\alpha$  by Lemma~\ref{lem-partioning2}. We now have to prove that \eqref{3-partioning} satisfies the two conditions of Proposition~\ref{Bing-curve}. First we claim that the boundary of each $X \in \nP_{\alpha,k}$ is a pair of distinct points for each $k\in\N$. By Lemmas~\ref{inner-bound2} and~\ref{lem:chain}, we know that $\overline{X}$ intersects the elements in the union 
\[
\bigcup_{Y \in \nP_{\alpha,k} \setminus\{X\}} (\overline{X}\cap \overline{Y}) = \bigcup_{Y \in\mathcal{L}_{\alpha,k} \setminus\{\overline{X}\}} (\overline{X}\cap Y)
\] 
in exactly two points. Thus \eqref{boud-subdivide3fine} implies the claim and, hence, Proposition~\ref{Bing-curve}~\eqref{Bing-curve1}. 
The fact that there are no three elements of $\nP_{\alpha,k}$ having a common boundary point is an immediate consequence of Lemma~\ref{lem:chain} yielding Proposition~\ref{Bing-curve}~\eqref{Bing-curve2}.  Now we can apply Proposition~\ref{Bing-curve} which yields that $L_\alpha$ is a simple closed curve. 

To prove (2) let $\Bold{\beta}\in G_2(\nS)$ and choose $\alpha \in \Bold{\beta}$. Since $\B_{\Bold{\beta}}$ is a Peano continuum which is a proper subset of the simple closed curve $L_\alpha$, it is an arc. 
If $\Bold{\beta}\not\in G_2(\nS)$ then $\B_{\Bold{\beta}}=\emptyset$ by Proposition~\ref{Char-boundary}.
\end{proof}

The content of the following lemma on dimension theory can be found in Kuratowski~\cite[\S\S 25 and 27]{Kuratowski66}.

\begin{lemma}\label{lem:KurTopDim} 
For a set $X\subset \R^3$, denote its topological dimension by $\mathrm{dim}(X)$.
\begin{enumerate}
\item\label{lem:KurTopDim1} If $X \subset Y \subset \R^3$, then $\mathrm{dim}(X) \le \mathrm{dim}(Y)$.
\item\label{lem:KurTopDim2} Let $Y\subset \R^3$. If $X_1,\ldots,X_n$ are closed in $Y$ with $Y=X_1\cup\dots\cup X_n$, then $\mathrm{dim}(Y) \le \max_{1\le i\le n}\mathrm{dim}(X_i)$.  
\end{enumerate}
\end{lemma}


\begin{lemma}\label{inner-bound1}\mbox{}
For each $\alpha\in \nS$, we have $\overline{\B_{\alpha}^{\circ}}=\B_{\alpha}$ w.r.t.\ the subspace topology on $\partial T$. More generally, the same holds for each element of the subdivision $\mathcal{C}_k^{(1)}$ defined in \eqref{eq:Cellk} for $k\ge1$.
\end{lemma}

\begin{proof}
The ambient space in this proof is $\partial T$. Recall first that by \eqref{boun_1} the collection $\{\B_\gamma;\, \gamma \in \nS\}$ is a finite collection of compact sets which covers $\partial T$. Thus for each $\alpha\in \nS$ the boundary $\partial_{\partial T}\B_\alpha$ is covered by $\bigcup_{\gamma\not=\alpha}\B_\gamma$ which implies that 
\begin{equation}\label{eq:subpart1}
\partial_{\partial T}\B_{\alpha}\subset L_\alpha.
\end{equation}
By  Proposition~\ref{lem:Main1.loop}, $L_\alpha$ is a finite union of arcs (having topological dimension $1$). Thus Lemma~\ref{lem:KurTopDim} implies that $\mathrm{dim}(\partial_{\partial T} \B_{\alpha})\le1$. On the other hand, since $\partial T$ forms a cut of $\R^3$, we have $\mathrm{dim}(\partial T)\ge 2$ by~\cite[\S59, II, Theorem~1]{Kuratowski68}. Therefore, by Lemma~\ref{lem:KurTopDim}~\eqref{lem:KurTopDim2} there exists $\beta \in \nS$ such that $\mathrm{dim}(\B_\beta)\ge 2$. Since $G(\nS)$ is strongly connected, the same is true for each $\beta \in \nS$ by Lemma~\ref{lem:KurTopDim}~\eqref{lem:KurTopDim1}. Now choose $\varepsilon > 0$ and $x\in\B_{\alpha}$ arbitrary. Subdivide $\B_{\alpha}$ according to the set equation \eqref{L-foldGra3} for $k$ large enough to yield the existence of $X=M^{-k+1}(\B_{\beta}+c) \in \mathcal{C}_k(\alpha)$ with $\beta\in \nS$ and $c\in\Z^3$ such that $X$ is a subtile of $\B_{\alpha}$ with diameter less than $\varepsilon$ and $x\in X$. As $X \subset \B_{\alpha}$ with $\mathrm{dim}(X)\ge 2$ and $\mathrm{dim}(\partial_{\partial T} \B_{\alpha})\le1$ there is a point $y\in X$ with $y\in \B_{\alpha}^{\circ}$. Since $\varepsilon$ was arbitrary, such a point $y$ can be chosen arbitrarily close to $x$. This proves the result for $\B_{\alpha}$. 

To prove the general case fix $k\ge 1$. Since $\mathcal{C}_k^{(1)}$ is a finite collection of compact sets covering $\partial T$, for each $X \in \mathcal{C}_k^{(1)}$ we have
\begin{equation}\label{eq:bdCovGen}
\partial_{\partial T} X \subset \bigcup_{Y \in \mathcal{C}_k^{(1)} \setminus \{X\}} X \cap Y.
\end{equation}
By Lemma~\ref{lem:piecesresults}~(2) each set $X\cap Y$ in the union on the right hand side is an affine image of a set $\B_{\Bold{\beta}}$ with $|\Bold{\beta}|\ge 2$ and, hence, homeomorphic to an arc or to a point by Proposition~\ref{lem:Main1.loop}~(2) and Theorem~\ref{Main-1}~\eqref{Main1.point}. Thus $\mathrm{dim}(\partial_{\partial T} X)\le1$ and we may continue in the same way as in the special case.
\end{proof}

\begin{lemma}\label{lem:bdov2}
For $\alpha\in \nS$, we have $\partial_{\partial T} \B_{\alpha}=L_\alpha$.
\end{lemma}

\begin{proof}
The ambient space in this proof is $\partial T$. Let $\mathcal{B}(x,\delta)=\{y\in \partial T;\, |x-y|<\delta\}$. The fact that $\partial_{\partial T} \B_{\alpha}\subset L_\alpha$ follows from \eqref{eq:subpart1}. To prove the reverse inclusion, suppose that for some $\beta\in \nS$ with $\{\alpha,\beta\}\in G_2(\nS)$ there exists $x\in \B_{\alpha,\beta}\setminus{\partial_{\partial T}{\B_{\alpha}}}$.
Since $\B_{\alpha,\beta} \subset \B_\alpha$ this implies that $x\in \B_{\alpha}^{\circ}$ and, hence, there exists $\delta >0$ with $\mathcal{B}(x,\delta)\subset \B_{\alpha}$. Since $\B_{\alpha,\beta} \subset \B_\beta$, we also have $x\in \B_{\beta}$ and by Lemma~\ref{inner-bound1} there exists $y\in \B_{\beta}^{\circ}\cap \mathcal{B}(x,\delta)$. Thus there exists $\delta'>0$ such that $\mathcal{B}(y,\delta') \subset \B_{\beta}\cap\mathcal{B}(x,\delta)$.
This implies $\mathcal{B}(y,\delta')\subset \B_{\alpha}\cap \B_{\beta}=\B_{\alpha,\beta}.$
By Proposition~\ref{lem:Main1.loop}~(2), $\B_{\alpha,\beta}$ is an arc for each vertex $\{\alpha,\beta\}$ in $G_2(\nS)$ which cannot contain a ball in $\partial T$ by a dimension theoretical argument analogous to the one in the proof of Lemma~\ref{inner-bound1}. This contradiction finishes the proof.
\end{proof}

Together with Proposition~\ref{lem:Main1.loop}~(1), Lemma~\ref{lem:bdov2} immediately yields the following result.

\begin{proposition}\label{bourdary-divide}
For each $\alpha\in\nS$ the boundary $\partial_{\partial T} \B_{\alpha}$ is a simple closed curve.
\end{proposition}

\subsection{Topological characterization of $2$-fold intersections and of the boundary of $T$}
We now prove Theorem~\ref{Main-1}~(\ref{Main1.sphere}) and ~(\ref{Main1.disk}). An important ingredient of this proof is the following characterization of a $2$-sphere which is also due to Bing~\cite{Bing51}.

\begin{theorem}[{see~\cite[Theorem~9]{Bing51}}]\label{surface}
A necessary and sufficient condition that a Peano continuum  $\mathcal{X}$ be a $2$-sphere is that one of its decreasing sequences of regular partitionings $\nQ_1, \nQ_2,\dots$ have the following properties (for $k\ge 1$):
\begin{itemize}
\item[(1)] The boundary  of each element of $\nQ_k$ is a simple closed curve.
\item[(2)] The intersection of the boundaries of $3$ elements of $\nQ_k$ contains no arc.
\item[(3)] If $U$ is an element of $\nQ_{k-1}$ (take $U=\mathcal{X} \text{ if } k=1$) the elements of $\nQ_k$ in $U$ may be ordered as $U_1,U_2,\dots, U_n$ so that the boundary of $U_j$
intersects $\partial U\cup \partial U_1\cup\dots \cup \partial U_{j-1}$ in a nondegenerate connected set (for $1\le j \le n$).
\end{itemize}
\end{theorem}

Since we want to apply this theorem to $\partial T$, we start with a sequence of partitionings for $\partial T$. To construct this sequence, for $\alpha^{(0)} \in\nS$, let 
\begin{equation}\label{1-partioning}
\nQ_{k}(\alpha^{(0)})=\{f_{d_1d_2\dots d_{k-1}}(\B_{\alpha^{(k-1)}})^{\circ};\; \alpha^{(0)}\xrightarrow{d_1}\alpha^{(1)}\xrightarrow{d_2}\cdots\xrightarrow{d_{k-1}}\alpha^{(k-1)}\in G(\nS)\} \quad(k\ge 1),
\end{equation}
where the interior is taken w.r.t.\ the subspace topology on $\partial T$. 
Using \eqref{1-partioning} we define the sequence of collections $(\nQ_k)_{k\geq 1}$ by
\begin{equation}\label{2-partioning}
\nQ_{k}=\bigcup_{\alpha\in G(\nS)} \nQ_{k}(\alpha) \quad (k\geq 1).
\end{equation}

\begin{lemma}\label{lem-partioning1}
The sequence $(\nQ_{k})_{k\geq 1}$ in \eqref{2-partioning} is a decreasing sequence of regular partitionings of $\partial T$.
\end{lemma}

\begin{proof}
Throughout this proof $\partial T$ is our ambient space. We claim that $\nQ_k$ is a partitioning of $\partial T$ for every $k\geq 1$. To prove this claim fix $k\ge 1$. Firstly, the closure of the union of all elements in $\nQ_k$ is $\partial T$ by Lemma~\ref{inner-bound1}, equation \eqref{boun_1}, and the set equation \eqref{L-foldGra3}.
Secondly, each element of $\nQ_k$ has the form $f_{d_1\dots d_{k-1}}(\B_{\alpha^{(k-1)}})^{\circ}$, and, hence, is open. Thirdly we have to show that the elements of $\nQ_k$ are mutually disjoint. Suppose that this is wrong. Then there exist $f_{d_1\dots d_{k-1}}(\B_{\alpha^{(k-1)}})$ and $f_{d'_1\dots  d'_{k-1}}(\B_{\alpha^{'(k-1)}})$ whose intersection $X$ has nonempty interior.
By arguing as in the proof of Lemma~\ref{inner-bound1} this implies that $\mathrm{dim}(X)\ge 2$. 
However, by Lemma~\ref{lem:piecesresults}~(2), $X$ is a shrinked copy of an $\ell$-fold intersection for some $\ell \ge 2$. More precisely, multiplying $X$ by $M^{k-1}$ and shifting it appropriately we see that $X$ is homeomorphic to $\B_{\Bold{\beta}}$ for some $\Bold{\beta} \in G_\ell(\nS)$ with $\ell \ge 2$. Thus $X$ is an arc or a point by Proposition~\ref{lem:Main1.loop}~(2) and Theorem~\ref{Main-1}~\eqref{Main1.point} and, hence, $\mathrm{dim}(X)\le1$. This contradiction proves mutual disjointness of the elements of $\nQ_k$, and the claim is established. Since the elements of $\nQ_k$ are open sets, $\nQ_k$ is regular for each $k\ge 1$ by Lemma~\ref{inner-bound1}.

Next we will check that $(\nQ_k)_{k\geq 1}$ is a decreasing sequence. First, we show that $\nQ_{k+1}$ is a refinement of $\nQ_k$ for each $k\ge 1$. Indeed, by the set equation \eqref{L-foldGra2} the closure $f_{d_1\dots d_{k-1}d_{k}}(\B_{\alpha^{(k)}})$ of each element of $\nQ_{k+1}$ is contained in the closure
$f_{d_1\dots d_{k-1}}(\B_{\alpha^{(k-1)}})$ of some element of $\nQ_k$. Taking interiors we get the refinement assertion. The maximum of the diameters of the elements of $\nQ_k$ approaches zero because $f_d$ is a contraction for each $d\in \D$. 
\end{proof}

The following result shows that the boundary operator  $\partial_{\partial T}$ commutes with certain affine maps.

\begin{lemma}\label{lem:commute}
Let $\partial T$ be the ambient space and let $f_{d_1\dots d_{k-1}}(\B_\alpha)^\circ  \in \nQ_k$ for some $k\ge 1$. Then 
\begin{equation}\label{eq:commute}
f_{d_1\dots d_{k-1}}(\partial_{\partial T}(\B_\alpha))=
\partial_{\partial T} ( f_{d_1\dots d_{k-1}}(\B_\alpha))=
\partial_{\partial T} ( f_{d_1\dots d_{k-1}}(\B_\alpha)^\circ ). 
\end{equation}
\end{lemma}

\begin{proof}
The second identity in \eqref{eq:commute} is a consequence of Lemma~\ref{inner-bound1}. 
To prove the first identity, let $f=f_{d_1\dots d_{k-1}}$ and $Y= f(\B_\alpha)^\circ$ 
for convenience. As $Y\in \nQ_k$,  
there are $\beta,\beta'\in \Z^3$ such that 
$Y = (U \cap V)^\circ$ 
with $U:=M^{-k+1}(T+\beta) \subset T$ and $V:=M^{-k+1}(T+\beta') \subset T +\alpha'$ for some $\alpha'\in\nS$. Then $U=f(T)$ and $Y \subset \partial U$. 
Since $f$ is a homeomorphism satisfying $f(\partial T) = \partial U$, 
we have $f(\partial_{\partial T}(\B_\alpha))  = \partial_{\partial U} (f(\B_\alpha))=\partial_{\partial U} \overline{Y}$. 
Thus  to prove the first identity in \eqref{eq:commute} we have to show 
\begin{equation}\label{eq:UY}
\partial_{\partial U}\overline{Y} = \partial_{\partial T}\overline{Y}.
\end{equation}

\begin{figure}[h]
\includegraphics[width=0.5\textwidth]{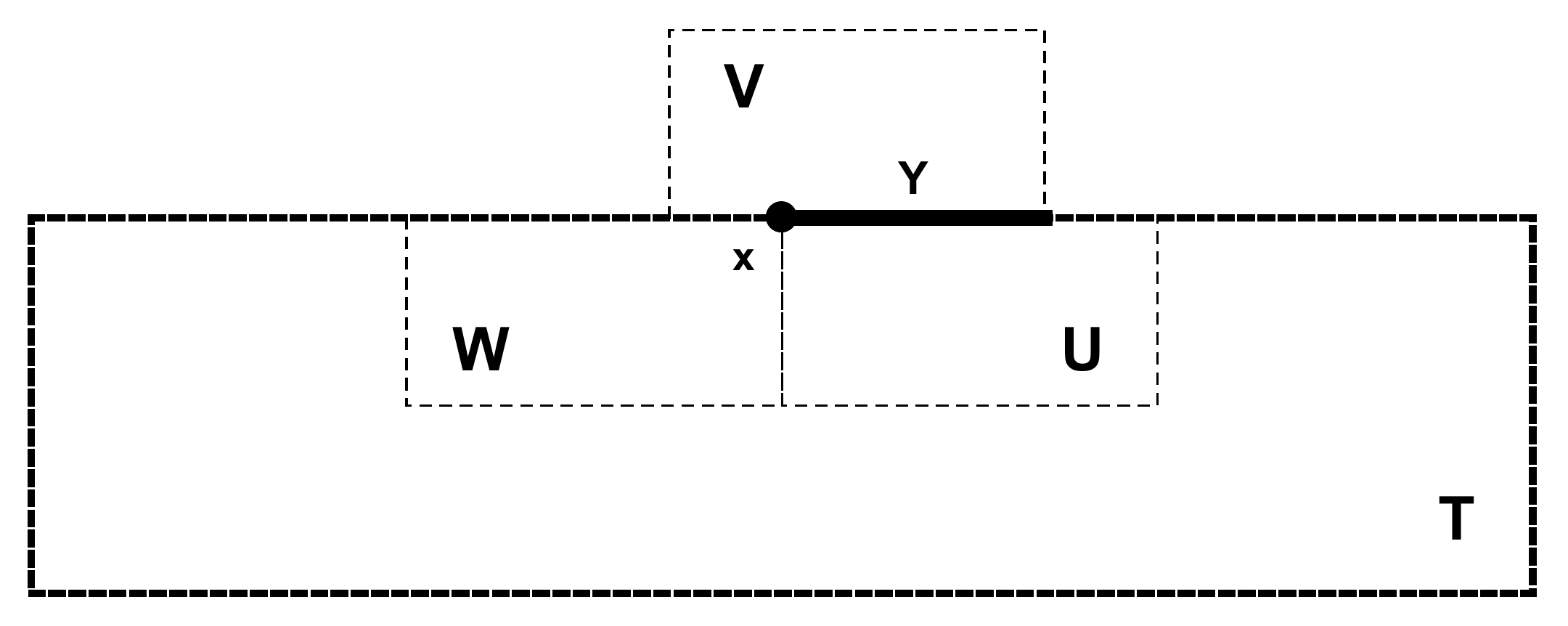}
\caption{If we take a small neighborhood $\mathcal{N}$ of $x$ we see that $\mathcal{N} \cap \partial T$ is different from $\mathcal{N} \cap \partial U$. This is what causes the difficulties in the proof. Note that $T,U,V,W$ are 3-dimensional objects, $Y$ is $2$-dimensional and $x$ is an arc. So this figure is just a schematic illustration of what is going on in a ``slice'' of $T$.
 \label{fig:ill}}
\end{figure}

Suppose first that $x\in \partial_{\partial U}\overline{Y}$. Then in each $\R^3$-neighborhood of $x$ there is a point $x'$ with $x' \in Y \subset \partial T$. On the other hand, by Lemma~\ref{lem:bdov2} 
(shifted by $\beta$ and multiplied by $M^{-k+1}$) 
there is $\gamma\in \Z^3\setminus\{\beta,\beta'\}$ such that $x\in W:=M^{-k+1}(T+\gamma)$ (see Figure~\ref{fig:ill} for an illustration). Summing up, we have $x \in U\cap V\cap W = M^{-k+1}(\B_{\beta'-\beta,\gamma-\beta}+\beta)$. 
Assume that $W
\subset T$ (the contrary case is treated in the same way), then $V\cap W = M^{-k+1}(\B_{\gamma-\beta'} +\beta')  \subset \partial T$.  Each element of any subdivision of  $V \cap W$ has topological dimension at least $2$ (see the proof of Lemma~\ref{inner-bound1}), while $U\cap V \cap W$ is an arc by Proposition~\ref{lem:Main1.loop}~(2) and, hence, has topological dimension $1$. Thus we can find an element
$
x''\in
V\cap W \setminus U\cap V\cap W  
=
V\cap W \setminus U\cap V 
\subset \partial T \setminus \overline{Y}
$
in each $\R^3$-neighborhood of $x$ and, hence, $x \in \partial_{\partial T} \overline{Y}$. 

Suppose now that $x\in \partial_{\partial T}\overline{Y}$. Then each $\R^3$-neighborhood of $x$ there is a point $x'$ with $x' \in Y \subset \partial U$. On the other hand, by \eqref{eq:bdCovGen} there is $\gamma\in \Z^3\setminus\{\beta,\beta'\}$ such that $x \in U\cap V\cap W=M^{-k+1}(\B_{\beta'-\beta,\gamma-\beta}+\beta)$ with $W:=M^{-k+1}(T+\gamma)$.  Since $U\cap W = M^{-k+1}(\B_{\gamma-\beta} +\beta)\subset \partial U$ and each element of any subdivision of $U \cap W$ has topological dimension at least $2$, while 
$U\cap V\cap W$ has topological dimension $1$, as before we can find an element $x'' \in \partial U \setminus \overline{Y}$ in each $\R^3$-neighborhood of $x$. Thus $x \in \partial_{\partial U} \overline{Y}$. 

Summing up we proved \eqref{eq:UY} and the result is established.
\end{proof}

We can now verify the first two conditions of Theorem~\ref{surface} for $(\nQ_k)_{k\ge 1}$.

\begin{lemma}\label{condition-1-2}\mbox{}
\begin{enumerate}
\item\label{condition-1-2.1} The boundary $\partial_{\partial T} X$ is a simple closed curve for each $X\in \nQ_k$ and each $k\ge 1$. 
\item\label{condition-1-2.2} The intersection of the boundary of three elements of $\nQ_k$ contains no arc for each $k\geq 1$.
\end{enumerate}
\end{lemma}
\begin{proof}

To prove assertion \eqref{condition-1-2.1} let $X \in \nQ_k$. Then
$\partial_{\partial T}X$ is an affine copy of $\partial_{\partial T}\B_\alpha$ for some $\alpha\in \nS$ by Lemma~\ref{lem:commute}. The assertion follows because $\partial_{\partial T}\B_\alpha$ is a simple closed curve by  Proposition~\ref{bourdary-divide}. 

To prove assertion \eqref{condition-1-2.2} we note that 
$
\partial_{\partial T}{\B_{\alpha}^{\circ}}\cap \partial_{\partial T}{\B_{\beta}^{\circ}}\cap \partial_{\partial T}{\B_{\gamma}^{\circ}} \subset \B_{\alpha} \cap \B_{\beta} \cap \B_{\gamma} = \B_{\alpha,\beta,\gamma},
$
so the intersection of the boundary of three elements of $\nQ_1$  contains at most one point because $\B_{\alpha,\beta,\gamma} $ contains at most one point. The same is true for $\nQ_k$ because triple intersections of boundaries of elements in $\nQ_k$ are just affine images of triple intersections of boundaries of elements in $\nQ_1$.
\end{proof}

It remains to verify the third condition of Theorem~\ref{surface}. 

For a fixed $\B_{\alpha}$, $\alpha\in \nS$, it is easy to determine the {\em neighbors of $\B_{\alpha}$ in\ $\partial T$}, {\it i.e.}, the elements $\B_{\beta}$ with $\B_{\alpha,\beta}\not=\emptyset$. Indeed, in view of Lemma~\ref{lem:bdov2} we know the neighbors of $\B_{\alpha}$ in $\partial T$ immediately from the right side of the identities in \eqref{eq:circchainLgamma}. This information allows to construct the Hata graph of $\{\B_{\alpha};~\alpha\in \nS\}$ which we denote by $H(\nS)$. This graph is depicted in Figure~\ref{HataGraph}.

\begin{remark}\label{rem:crystal}
The graph in Figure~\ref{HataGraph} shows again a relation to polyhedral geometry. The dual polyhedron of the truncated octahedron is the so-called {\it tetrakis hexahedron} (a {\it Catalan polyhedron}, see Conway {\it et al.}~\cite[p.~284]{CBG:08}; Tappert and Tappert~\cite[p.~26]{TT:11} mention its relevance in the crystallography of rough diamonds). The graph in the figure is the {\it tetrakis hexahedral graph}, the graph of vertices and edges of this polyhedron.
\end{remark}

\begin{figure}[h]
\begin{minipage}{\textwidth}
\begin{minipage}[c]{.7 \textwidth}
\includegraphics[width=0.8 \textwidth]{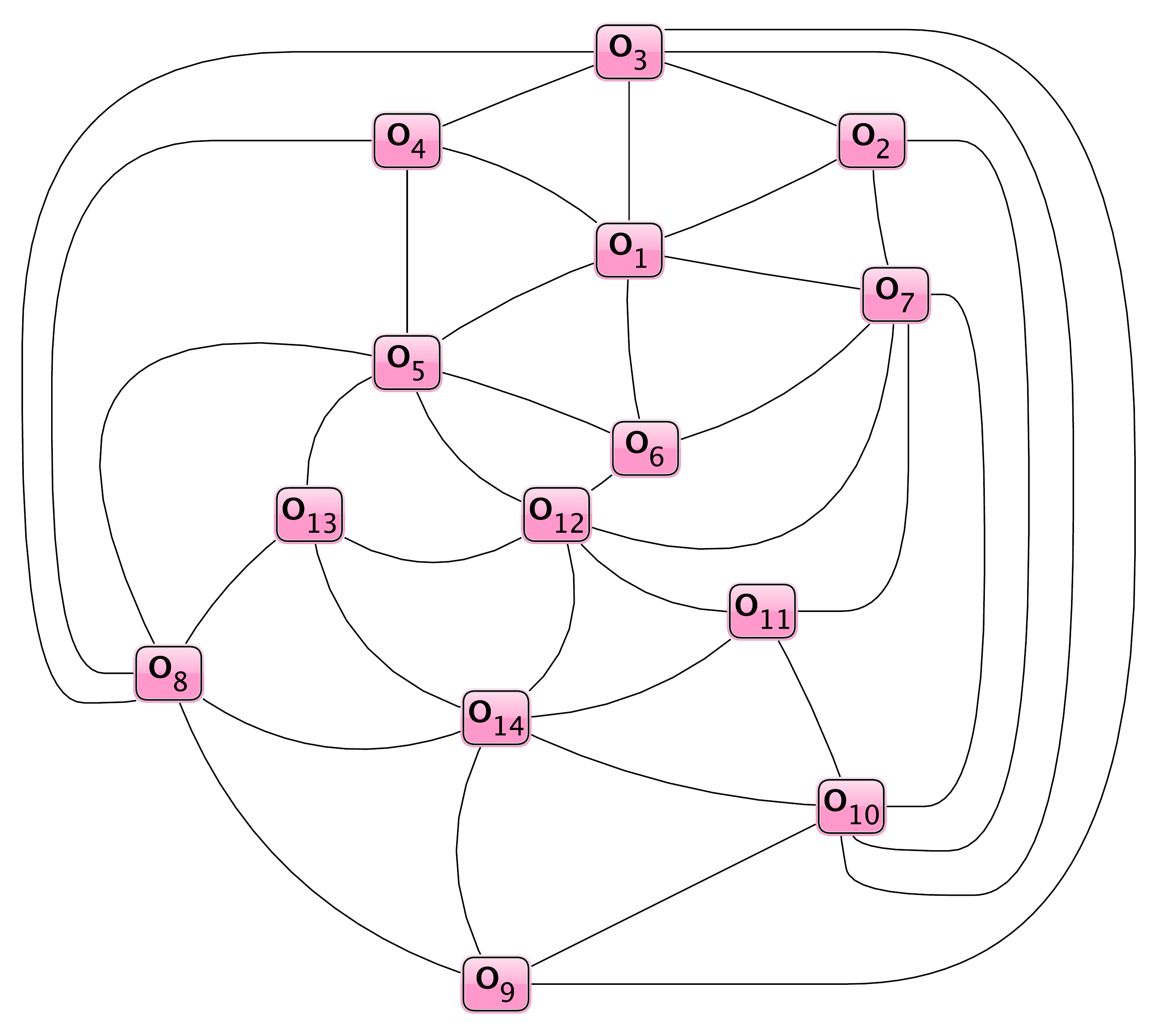}
\end{minipage}
\begin{minipage}[c]{.05 \textwidth}
\begin{tabular}{|l|c|c|l|}
\hline 
Vertex & ``O-notation''\\
\hline 
$\B_{\{\overline{Q-P}\}}$& $\BO_1$\\
\hline
$\B_{\{{N-Q+P}\}}$& $\BO_2$\\
\hline
$\B_{\{{N-Q}\}}$& $\BO_3$\\
\hline
$\B_{\overline{Q}}$& $\BO_4$\\
\hline
$\B_{\overline{N}}$& $\BO_5$\\
\hline
$\B_{\{\overline{N-P}\}}$& $\BO_6$\\
\hline
$\B_{{P}}$& $\BO_7$\\
\hline
$\B_{\overline{P}}$& $\BO_8$\\
\hline
$\B_{\{N-P\}}$& $\BO_9$\\
\hline
$\B_{{N}}$& $\BO_{10}$\\
\hline
$\B_{{Q}}$& $\BO_{11}$\\
\hline
$\B_{\{\overline{N-Q}\}}$& $\BO_{12}$\\
\hline
$\B_{\{\overline{N-Q+P}\}}$& $\BO_{13}$\\
\hline
$\B_{\{Q-P\}}$& $\BO_{14}$\\
\hline
\end{tabular}
\end{minipage}
\caption{The Hata graph $H(\nS)$ which is known as the {\it tetrakis hexahedral graph} and the table which determines the order of the vertices $\{\B_{\alpha};\,\alpha\in \nS\}$.\label{HataGraph}}
\end{minipage}
\end{figure}
We give an order to the $2$-fold intersections $\B_\alpha$ by setting $\BO_i:= \B_{\alpha_i}$ ($1\le i\le 14$) according to the right side of Figure~\ref{HataGraph}. We have the following lemma.

\begin{lemma}\label{check-partition-1}
Let $\partial T$ be our ambient space. Then  
$
 \partial \BO_i\cap (\partial \BO_1\cup \partial \BO_2\cup\dots\cup \partial \BO_{i-1})
$
is a nondegenerate connected set for each $i\in \{2,\ldots,14\}$.
\end{lemma}
\begin{proof}
Let $\BO_{j,k} := \BO_j \cap \BO_k$ ($1\le j,k \le 14$). First, by Lemma~\ref{lem:bdov2} we have 
$$
\BA_i :=\partial \BO_i\cap (\partial \BO_1\cup \partial \BO_2\cup\dots\cup \partial \BO_{i-1})=\BO_{i,1}\cup \BO_{i,2}\cup\dots\cup \BO_{i,i-1}.
$$
From the Hata graph $H(\nS)$, we can read off which of the sets $\BO_{j,k}$ is nonempty. Together with the table in Figure \ref{HataGraph} this information leads to the following identities.
\begin{align*}
\BA_2 & = \BO_{2,1}= \B_{\mybinom{\overline{Q-P}}{N-Q+P} },\\
\BA_3&=\BO_{3,1}\cup  \BO_{3,2}=\B_{\mybinom{\overline{Q-P}}{N-Q}} \cup \B_{\mybinom{N-Q}{N-Q+P} },\\
\BA_4&=\BO_{4,1}\cup  \BO_{4,2}\cup  \BO_{4,3}=\BO_{4,1}\cup \BO_{4,3}=\B_{\mybinom{\overline{Q}}{N-Q} } \cup \B_{\mybinom{\overline{Q}}{\overline{Q-P}}}, \\
\BA_5&=\BO_{5,1}\cup  \dots \cup \BO_{5,4}=\BO_{5,4}\cup \BO_{5,1}=\B_{\mybinom{\overline{N}}{\overline{Q}} }\cup \B_{\mybinom{\overline{N}}{\overline{Q-P}} },\\
\BA_6&= \BO_{6,1}\cup\dots\cup\BO_{6,5}=\BO_{6,1}\cup \BO_{6,5} =\B_{\mybinom{\overline{N-P}}{\overline{Q-P}} }\cup \B_{\mybinom{\overline{N}}{\overline{N-P}} },\\
\BA_7&= \BO_{7,1}\cup\dots\cup\BO_{7,6} =\BO_{7,6}\cup \BO_{7,1}\cup \BO_{7,2}=\B_{\mybinom{\overline{N-P}}{P} }\cup \B_{\mybinom{\overline{Q-P}}{P}}  \cup \B_{\mybinom{P}{N-Q+P}},\\
\BA_8&=\BO_{8,1}\cup\dots\cup\BO_{8,7}=\BO_{8,3}\cup \BO_{8,4}\cup \BO_{8,5}=\B_{\mybinom{\overline{P}}{N-Q}}\cup \B_{\mybinom{\overline{P}}{\overline{Q}}}\cup \B_{\mybinom{\overline{P}}{\overline{N}}},\\
\BA_9&=\BO_{9,1}\cup\dots\cup\BO_{9,8}=\BO_{9,3}\cup \BO_{9,8}=\B_{\mybinom{N-P}{N-Q}}\cup \B_{\mybinom{N-P}{\overline{P}}},\\
\BA_{10}&=\BO_{10,1}\cup\dots\cup\BO_{10,9}=\BO_{10,7}\cup \BO_{10,2}\cup \BO_{10,3}\cup \BO_{10,9}
\\&
=\B_{\mybinom{N}{P}}\cup \B_{\mybinom{N}{N-Q+P}}\cup \B_{\mybinom{N}{N-Q}}\cup \B_{\mybinom{N}{N-P}},\\
\BA_{11}&=\BO_{11,1}\cup\dots\cup\BO_{11,10}=\BO_{11,7}\cup \BO_{11,10}=\B_{\mybinom{Q}{P}}\cup \B_{\mybinom{Q}{N}},\\
\BA_{12}&=\BO_{12,1}\cup\dots\cup\BO_{12,11}=\BO_{12,5}\cup \BO_{12,6}\cup \BO_{12,7}\cup \BO_{12,11}
\\&
=\B_{\mybinom{\overline{N-Q}}{\overline{N}}}\cup \B_{\mybinom{\overline{N-Q}}{\overline{N-P}}}\cup \B_{\mybinom{\overline{N-Q}}{{P}}}\cup \B_{\mybinom{\overline{N-Q}}{Q}},\\
\BA_{13}&=\BO_{13,1}\cup\dots\cup\BO_{13,12}=\BO_{13,8}\cup \BO_{13,5}\cup \BO_{13,12}=\B_{\mybinom{\overline{N-Q+P}}{\overline{P}}}\cup \B_{\mybinom{\overline{N-Q+P}}{\overline{N}}}\cup \B_{\mybinom{\overline{N-Q+P}}{\overline{N-Q}}},\\
\BA_{14}&=\BO_{14,1}\cup\dots\cup\BO_{14,13}=\partial \BO_{14}.
\end{align*}
We can now read off the graph $G_3(\nS)$ that $\BA_{i}$ is connected for each $2\le i \le 14$. The fact that it is nondegenerate follows because each $3$-fold intersection is an arc by Proposition~\ref{lem:Main1.loop}~(2).
\end{proof}

Note that $\partial_{\partial T} \BO_j^{\circ}=\partial_{\partial T} \BO_j$ for $j\in\{1,\ldots,14\}$ by Lemma~\ref{lem:commute} and $\partial_{\partial T}\partial T = \emptyset$. Thus Lemma~\ref{check-partition-1} implies that $(\nQ_k)_{k\geq 1}$ satisfies condition (3) of Theorem~\ref{surface} for the case $k=1$ (here we set $\nQ_0=\{\partial T\}$). 

To show that Theorem~\ref{surface}~(3) is true for $k\ge 2$, we need the following results on intersections.

\begin{lemma}\label{intersect-1}
Let $\alpha\in \nS$, $1\le j \le C-1$, and $j\leq i\leq C-1$, then 
\begin{itemize}
\item[(1)]
$\big(\B_{\alpha}+(i-j)P)\big)\cap \big(\B_{\alpha}+iP\big)=\emptyset$ and
\item[(2)]
$\big(\B_{\alpha}+(i-j)P)\big)\cap \big(\B_{\alpha+P}+iP\big)=\emptyset.$
\end{itemize}
\end{lemma}

\begin{proof}
Shifting by $-(i-j)P$, we see that $\big(\B_{\alpha}+(i-j)P\big)\cap \big(\B_{\alpha}+iP\big)$
is homeomorphic to $\B_{\alpha,jP,jP+\alpha}$. Looking at Figure~\ref{quad-graph}, we see that $\{\alpha,jP,jP+\alpha\}$ is not a vertex of $G_3(\nS)$. Thus Theorem~\ref{Main-1}~\eqref{Main1.point} yields (1). The second assertion follows in a similar way.
\end{proof}

\begin{lemma}\label{intersect-2}
If $\alpha\in\{Q,~ N,~N-Q+P,~\overline{N-Q},~\overline{Q-P},~\overline{N-P}\}$, then 
\begin{itemize}
\item[(1)]
$\B_{\alpha}\cap \big(\B_{\alpha-P}+P\big)\neq \emptyset$ and
\item[(2)]
$\B_{\alpha}\cap  \B_{\alpha-P} \neq \emptyset$.
\end{itemize}
\end{lemma}
\begin{proof}
Since $\B_{\alpha}\cap \big(\B_{\alpha-P}+P\big)=\B_{\alpha,P}$ and $\{\alpha,P\}$ is a vertex of $G_2(\nS)$ for each $\alpha\in \{Q,~ N,~N-Q+P,~\overline{N-Q},~\overline{Q-P},~\overline{N-P}\}$ (see Table~\ref{tab:triple-graph}), assertion (1) follows from Proposition~\ref{lem:Main1.loop}. Assertion (2) is proved in the same way.
\end{proof}

With help of these lemmas we can prove that the subdivisions of $\B_\alpha$ have a linear order.

\begin{corollary}\label{cor:intersect3}
Each $2$-fold intersection $\B_{\alpha}$, $\alpha\in\nS$, can be generated by the following ordered set equations (we only need to give the equations for the following $7$ elements of $\nS$ by symmetry). 
\begin{equation}\label{intersect-3}
\begin{split}
M\B_P&\circeq \Big(\bigcup_{i=0}^{C-A-1}(\B_{Q-P}\cup\B_Q)+iP\Big)\cup \big(\B_{Q-P}+(C-A)P\big),\\
M\B_Q&\circeq\Big(\bigcup_{i=0}^{C-B-1}(\B_{N-P}\cup \B_N)+iP\Big)\cup \big(\B_{N-P}+(C-B)P\big),\\
M\B_N&\circeq \B_{\overline{P}},\\
M\B_{Q-P}&\circeq\Big(\bigcup_{i=0}^{C-B+A-2}(\B_{N-Q}\cup\B_{N-Q+P})+iP\Big)\cup (\B_{N-Q}+(C-B+A-1)P),\\
M\B_{N-Q+P}&\circeq\Big(\bigcup_{i=0}^{B-A-1}(\B_{\overline{N-Q+P}}\cup \B_{\overline{N-Q}})+iP\Big)\cup (\B_{\overline{N-Q+P}}+(B-A)P),\\
M\B_{N-P}&\circeq\Big(\bigcup_{i=0}^{A-2}(\B_{\overline{Q}}\cup\B_{\overline{Q-P}})+iP\Big)\cup (\B_{\overline{Q}}+(A-1)P),\\
M\B_{N-Q}&\circeq\Big(\bigcup_{i=0}^{B-2}(\B_{\overline{N}}\cup B_{\overline{N-P}})+iP\Big)\cup (\B_{\overline{N}}+(B-1)P).\\
\end{split}
\end{equation}
Here we use ``$\circeq$'' to emphasize that the union on the right hand side is given by the order indicated in Figure~\ref{fig:orderfig} and that only the sets being adjacent in this order have nonempty intersection. Each of these intersections is an arc. 
\end{corollary}

\begin{proof}
By Lemma~\ref{intersect-1} and Lemma~\ref{intersect-2}, we conclude that the sets belonging to the union on the right hand side intersect if and only if they are adjacent in the order illustrated in Figure~\ref{fig:orderfig}. Their intersection is an arc by Proposition~\ref{lem:Main1.loop}.
\end{proof}

\begin{figure}[htbp]
\includegraphics[width=0.5 \textwidth]{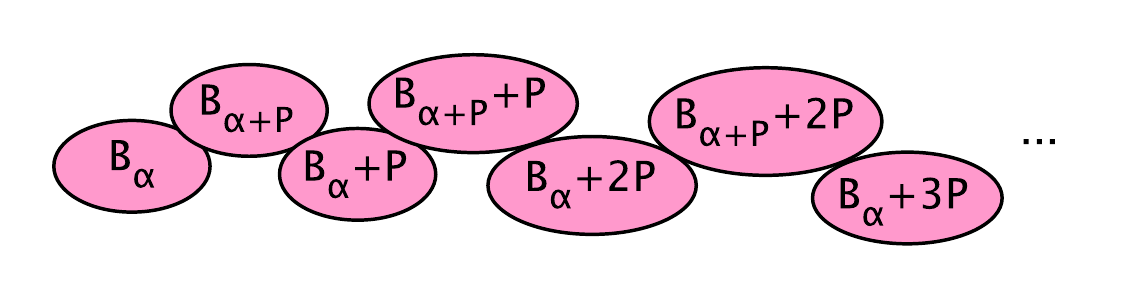}\caption{Order of the intersections on the right hand side of the identities in~\eqref{intersect-3}.\label{fig:orderfig}}
\end{figure}

\begin{proposition}~\label{prop:Main1.surface}
The decreasing sequence of regular partitionings $(\nQ_k)_{k\geq 1}$ of $\partial T$ defined in \eqref{2-partioning} satisfies the conditions in Theorem~\ref{surface}.  Hence, $\partial T$ is a $2$-sphere.
\end{proposition}

\begin{proof}
Throughout this proof, $\partial T$ is our ambient space. Conditions $(1)$ and $(2)$ of Theorem~\ref{surface} are satisfied by Lemma \ref{condition-1-2}. Lemma \ref{check-partition-1} and the remark after it shows that condition $(3)$ of Theorem~\ref{surface} is true for $k=1$. 

We now prove that condition $(3)$ of Theorem~\ref{surface} holds for $k= 2$. Indeed, $\nQ_1=\{\B_{\alpha}^\circ\,;\,\alpha\in \nS\}$ and for each $\alpha\in \nS$ the elements of the refinement $\nQ_2$ contained in $\B_{\alpha}^\circ$  are given by $(U_j)_{j=1}^{\ell_{\alpha}}$, where the sets $M \overline{U_1},\ldots,M \overline{U_{\ell_{\alpha}}}$ are the sets in the union on the right hand side of \eqref{intersect-3}. By the linear ordering of these sets proved in Corollary~\ref{cor:intersect3}
\[
\partial U_j\cap (\partial \B_{\alpha}^\circ \cup \partial U_1\cup \dots\cup \partial U_{j-1}) =
\begin{cases}
\overline{\partial U_j \setminus (\overline{U_j}\cap \overline{U_{j+1}})},& j <\ell_\alpha,\\
\partial U_j,& j=\ell_\alpha.
\end{cases}
\]
Since $\partial U_j$ is a simple closed curve by Lemma~\ref{condition-1-2}~(1), and $\overline{U_j}\cap \overline{U_{j+1}}$ is a subarc of this curve by Corollary~\ref{cor:intersect3}, we conclude that $\partial U_j\cap (\partial \B_{\alpha}^\circ \cup \partial U_1\cup \dots\cup \partial U_{j-1})$ is an arc or a simple closed curve, and, hence, nondegenerate and connected.

Let now $k > 2$ be arbitrary. Each $U \in \nQ_{k-1}$ is of the form  $U=f(\B_\alpha)^\circ$ for some $f=f_{d_1\dots d_{k-2}}$ with $d_1,\ldots, d_{k-2}\in \D$ and $\alpha\in\nS$. By Lemma~\ref{lem:commute} we have $U=f(\B_\alpha^\circ)$ and thus if  $(U_j)_{j=1}^{\ell_{\alpha}}$ are the elements of $\nQ_2$ contained in $\B_{\alpha}^\circ$ then $(f(U_j))_{j=1}^{\ell_{\alpha}}$ are the elements of $\nQ_{k}$ contained in $U$.  Therefore, again by Lemma~\ref{lem:commute}
\[
\partial f(U_j)\cap (\partial U \cup \partial f(U_1)\cup \dots\cup \partial f(U_{j-1}))
=
f(\partial U_j\cap (\partial \B_{\alpha}^\circ \cup \partial U_1\cup \dots\cup \partial U_{j-1})),
\]
and the elements of  $\nQ_{k}$ contained in $U$ satisfy condition $(3)$ of Theorem~\ref{surface} because the elements of $\nQ_2$ contained in $\B_{\alpha}^\circ$ satisfy it. Thus condition $(3)$ of Theorem~\ref{surface} holds for $k> 2$ as well.

Summing up we may apply Theorem~\ref{surface} and the result follows.
\end{proof}

Proposition~\ref{prop:Main1.surface} proves Theorem~\ref{Main-1}~\eqref{Main1.sphere}. Finally,  the next proposition implies Theorem~\ref{Main-1}~\eqref{Main1.disk}.

\begin{proposition}\label{lem:Main1.disk}
Assume that $\alpha \in \Z^3\setminus\{0\}$. Then $\B_{\alpha}$ is homeomorphic to a closed disk if $\alpha\in \nS$ and empty otherwise.
\end{proposition}
 
\begin{proof}
For $\alpha\in \nS$, the intersection $\B_\alpha$ is a subset of the $2$-sphere $\partial T$ (by Proposition~\ref{prop:Main1.surface}) whose boundary $\partial_{\partial T}\B_\alpha$ is a simple closed curve (by Proposition~\ref{bourdary-divide}). Thus  $\B_{\alpha}$ is homeomorphic to a closed disk by the Sch\"onflies Theorem. If $\alpha\not \in \nS$, then $\B_{\alpha}=\emptyset$ by the definition of $\nS$ in \eqref{eq:neigh}.
\end{proof}

\section{Perspectives}\label{sec:perspective}

We conclude this paper by mentioning some topics for further research. A first natural question is whether each self-affine tile satisfying the conditions of Theorem~\ref{Main-1} is homeomorphic to a $3$-ball. 
For a single example this can be checked by an algorithm given by Conner and Thuswaldner~\cite[Section~7]{ConnerThuswaldner0000}. However, we currently do not know how to do this for a whole class of tiles. Although Conner and Thuswaldner~\cite[Section~8.2]{ConnerThuswaldner0000} exhibited a self-affine tile whose boundary is a $2$-sphere but which is itself not a $3$-ball (a self-affine {\em Alexander horned sphere}), we conjecture the following to be true.

\begin{conjecture}
A self-affine tile that satisfies the conditions of Theorem~\ref{Main-1} is homeomorphic to a $3$-ball.
\end{conjecture}

Besides that we think that using the results of Bing~\cite{Bing51} and Kwun~\cite{Kwun61} one could prove more topological results for self-affine tiles (and attractors of iterated function systems in the sense of Hutchinson~\cite{Hutchinson81} in general). In particular, getting information on the topology of $3$-dimensional Rauzy fractals (see {\it e.g.}~\cite{EiItoRao06,ItoRao06a,SiegelThuswaldner10}) would be interesting. Even topological results for higher dimensional self-affine tiles should be tractable by using modifications of our theory. However, particularly for manifolds of dimension $4$ and higher, according to Kwun's result, one has to deal with more complicated conditions which lead to new challenges.

Let $T$ be a $2$-dimensional self-affine tile. Recently, Akiyama and Loridant~\cite{AkiyamaLoridant11} provided H\"older continuous surjective mappings $h:\mathbb{S}^1 \to \partial T$ whose H\"older exponent, which is defined in terms of the Hausdorff dimension of $\partial T$, is optimal. This has been considered in a more general framework in Rao and Zhang~\cite{RaoZhang16}. We formulate the following problem for mappings from the $2$-sphere to the boundary of a $3$-dimensional self-affine tile.

\begin{problem}
For a $3$-dimensional self-affine tile whose boundary is a $2$-sphere find a homeomorphism $h:\mathbb{S}^2 \to \partial T$ which is H\"older continuous. What is the optimal H\"older exponent for such a homeomorphism?
\end{problem}

\subsection*{Acknowledgement}
We thank the anonymous referee for many valuable comments. Among other suggestions, the referee pointed out the relation to polyhedral geometry and crystallography.

\bibliographystyle{siam}  
\bibliography{biblio}

\end{document}